\documentclass{amsart}

\usepackage[utf8]{inputenc}

\usepackage{color}
\usepackage{amsmath}
\usepackage{amssymb}
\usepackage{amsfonts}
\usepackage{amsthm}
\usepackage{mathabx}
\usepackage{bbm}

\usepackage[mathscr]{eucal}

\usepackage{algorithmic}
\usepackage{algorithm}

\usepackage{xparse}
\usepackage{xspace}

\usepackage{graphicx}
\graphicspath{{resources/images/}{resources/drawings/}}

\usepackage{tikz}
\usetikzlibrary{decorations.pathreplacing}
\usetikzlibrary{calc}

\usepackage[labelformat=empty]{subcaption}
\newif\ifmustfixme
\mustfixmefalse
\definecolor{fixmecolour}{rgb}{0.86,0.20,0.18}
\definecolor{todocolour}{rgb}{0.16,0.63,0.60}
\definecolor{FIXMECOLOUR}{rgb}{0.86,0.20,0.18}
\definecolor{TODOCOLOUR}{rgb}{0.16,0.63,0.60}
\mustfixmetrue
\ifmustfixme
	\definecolor{titlecolor}{rgb}{0.86,0.20,0.18}
\else
	\definecolor{titlecolor}{rgb}{0.0,0.0,0.0}
\fi

\theoremstyle{plain}
\newtheorem{theorem}{Theorem}[section]
\newtheorem{lemma}[theorem]{Lemma}
\newtheorem{lemma-definition}[theorem]{Lemma-Definition}
\newtheorem{corollary}[theorem]{Corollary}

\newtheorem{proposition}[theorem]{Proposition}

\newtheorem{definition}[theorem]{Definition}

\newtheorem*{fact*}{Fact}
\newtheorem*{claim}{Claim}
\newtheorem*{notation}{Notation}

\newtheorem{alphcase}{Case}

\theoremstyle{definition}
\newtheorem{remark}[theorem]{Remark}
\newtheorem*{remark*}{Remark}
\newtheorem{proofstep}{Step}
\newtheorem{proofcase}{Case}
\makeatletter
\@addtoreset{proofpart}{theorem}
\@addtoreset{proofstep}{theorem}
\@addtoreset{proofcase}{theorem}

\makeatother
\newcounter{parentnumber}

\newcommand{\INPUT}{\REQUIRE}
\newcommand{\OUTPUT}{\ENSURE}

\newcommand{\N}{\mathbb{N}} % Natural numbers
\newcommand{\Z}{\mathbb{Z}} % Integer ring
\newcommand{\Q}{\mathbb{Q}} % Quotient field
\newcommand{\R}{\mathbb{R}} % Real numbers
\newcommand{\F}{\mathbb{F}} % Finite field (or residue class field in our case too)

\newcommand{\PP}{\mathbb{P}} % Set of primes

\makeatletter
\let\@@mod\mod
\DeclareRobustCommand{\mod}{\@ifstar\@mods\@@mod}
\def\@mods#1{\mkern1mu{\operator@font mod}\mkern 6mu#1}
\makeatother

\newcommand{\sii}{\Longleftrightarrow}

\newcommand{\set}[1]{\left\{ #1 \right\}}
\newcommand{\ind}[1]{\left( #1 \right)}

\newcommand{\floor}[1]{\lfloor #1 \rfloor}

\newcommand{\Min}{\operatorname{Min}}
\newcommand{\oo}{\mathcal{O}} % Ring of integers
\newcommand{\oox}{\oo[x]} % Ring of integers
\newcommand{\ooL}{\oo_{L}} % Ring of integers
\newcommand{\oop}{\oo_{\p}} % Ring of integers

\newcommand{\Kbar}{\overline{K}}
\newcommand{\Kx}{K[x]}
\newcommand{\Fy}{\F[y]}
\newcommand{\oov}{\oo_{v}} % Completion
\newcommand{\oovx}{\oov[x]}

\newcommand{\Kv}{K_{v}}
\newcommand{\Kvx}{\Kv[x]}

\newcommand{\Kvbar}{\overline{K}_{v}}

\newcommand{\Lp}{L_{\p}}

\newcommand{\pp}{\mathcal{P}} % Set of prime ideals
\newcommand{\p}{\mathfrak{p}} % prime ideal P_{0}
\newcommand{\pz}{\mathfrak{p}_{0}} % prime ideal P
\newcommand{\q}{\mathfrak{q}}  % prime ideal Q
\renewcommand{\l}{\mathfrak{l}}  % prime ideal L
\newcommand{\B}{\mathcal{B}} % Basis
\newcommand{\m}{\mathfrak{m}} % maximal ideal
\newcommand{\Max}{\operatorname{Max}}
\newcommand{\Spec}{\operatorname{Spec}}

\newcommand{\type}[0]{\mathbbm{t}}
\newcommand{\tree}[0]{\mathfrak{T}}
\newcommand{\treenop}[0]{\mathfrak{T}^{\nop}}

\newcommand{\Num}[0]{\mathcal{N}}

\newcommand{\fbar}{\overline{f}}

\newcommand{\eto}[1]{e_{1} \cdots e_{#1}}

\newcommand{\epp}[0]{e(\p / \m)}

\newcommand{\Newton}[0]{\operatorname{N}}
\newcommand{\Respol}[0]{\operatorname{R}}
\newcommand{\Rep}[0]{\operatorname{Rep}}
\newcommand{\App}[0]{\operatorname{App}}

\newcommand{\Trunc}[0]{\operatorname{Trunc}}
\newcommand{\ord}[0]{\operatorname{ord}}

\newcommand{\Ref}[0]{\operatorname{Ref}}

\newcommand{\Disorder}[0]{\operatorname{D}}

\newcommand{\nop}[0]{\operatorname{nop}}

\newcommand{\clam}[2]{\lambda_{#1}^{#2}}
\newcommand{\clamnop}[2]{(\lambda_{#1}^{#2})^{\nop}}
\newcommand{\minclam}[2]{\Min\set{\clam{#1}{#2},\clam{#2}{#1}}}
\newcommand{\lmin}[0]{\lambda_{\min}}
\newcommand{\lmax}[0]{\lambda_{\max}}

\newcommand{\ml}[0]{m_{\ell}}

\newcommand{\wvalueext}[3]{#1{w}#2\left(#3\right)}

\DeclareDocumentCommand	\w	{ o m } {\wvalueext{}{ \IfValueT{#1}{_{#1}} }{#2}}
\DeclareDocumentCommand	\wt	{ o m } {\wvalueext{}{ \IfValueT{#1}{_{#1}} }{#2(\theta)} }
\DeclareDocumentCommand	\wv	{ o m } {\wvalueext{\vec}{ \IfValueT{#1}{_{#1}} }{#2}}
\DeclareDocumentCommand	\wdg { o m } {\wvalueext{\hat}{ \IfValueT{#1}{_{#1}} }{#2}}
\DeclareDocumentCommand	\wdt	{ o m } {\wvalueext{\hat}{ \IfValueT{#1}{_{#1}} }{#2(\theta)} }
\DeclareDocumentCommand	\v	{ o m } {v{\IfValueT{#1}{_{#1}} }\left( #2 \right)}
\DeclareDocumentCommand	\vt	{ o m } {v{\IfValueT{#1}{_{#1}} }\left( #2(\theta) \right)}

\newcommand{\Ok}{\operatorname{Ok}}
\DeclareDocumentCommand{\Canonify}{oo}{\operatorname{Canonify}\IfValueT{#1}{\left(#1\right)}}
\DeclareDocumentCommand{\Simplify}{oo}{\operatorname{Simplify}\IfValueT{#1}{_{#1}}\IfValueT{#2}{\left(#2\right)}}
\DeclareDocumentCommand{\Transfer}{ooo}{\IfValueTF{#1}{\underset{#1 \rightarrow #2}{\operatorname{Transfer}}}{\operatorname{Transfer}}\IfValueT{#3}{\left(#3\right)}}

\newcommand{\mm}[0]{\operatorname{MaxMin}}

\newcommand{\ii}[0]{\mathbbm{i}}
\newcommand{\jj}[0]{\mathbbm{j}}
\newcommand{\uu}[0]{\mathbbm{u}}

\newcommand{\Smin}[0]{S_{\min}}

\newcommand{\bigo}[1]{O\left(#1\right)}

\newcommand{\disc}[0]{\operatorname{disc}}

\title{Triangular bases of integral closures}

\author[Stainsby]{Hayden D. Stainsby}
\address{Departament de Matem\`{a}tiques,
         Universitat Aut\`{o}noma de Barcelona,
         Edifici C, E-08193 Bellaterra, Barcelona, Catalunya, Spain}
\email{hds@mat.uab.cat}

\thanks{Partially supported by MTM2013-40680-P from the Spanish MEC}
\date{}

\makeatletter
\@namedef{subjclassname@2010}{%
  \textup{2010} Mathematics Subject Classification}


\begin{document}

\begin{abstract}
	In this work, we consider the problem of computing triangular bases of integral closures of one-dimensional local rings.
	
	Let $(K, v)$ be a discrete valued field with valuation ring $\mathcal{O}$ and let $\mathfrak{m}$ be the maximal ideal. We take $f \in \mathcal{O}[x]$, a monic irreducible polynomial of degree $n$ and consider the extension $L = K[x]/(f(x))$ as well as $\mathcal{O}_{L}$ the integral closure of $\mathcal{O}$ in $L$, which we suppose to be finitely generated as an $\mathcal{O}$-module.
	
	The algorithm $\operatorname{MaxMin}$, presented in this paper, computes triangular bases of fractional ideals of $\mathcal{O}_{L}$. The theoretical complexity is equivalent to current state of the art methods and in practice is almost always faster. It is also considerably faster than the routines found in standard computer algebra systems, excepting some cases involving very small field extensions.
\end{abstract}

\maketitle

%-------------------------------------------------------------------------------
% SECTION: Introduction
%-------------------------------------------------------------------------------
\section{Introduction}\label{secIntroduction}
%\section{Reduced triangular bases of integral closures}\label{secAB}

Let $(K, v)$ be a discrete valued field with valuation ring $\oo$. Let $\m$ be the maximal ideal, $\pi \in \m$ a generator of $\m$ and $\F = \oo / \m$ the residue class field.

Let $K_v$ be the completion of $K$, and retain $v: \Kbar_v^* \to \Q$ the canonical extension of $v$ to a fixed algebraic closure of $K_v$. Let $\oov$ be the valuation ring of $\Kv$.

Let $f \in \oo[x]$ be a monic, irreducible polynomial of degree $n$ and fix a root $\theta \in \Kbar$ in an algebraic closure of $K$. Let $L = K(\theta)$ be the corresponding finite extension of $K$ and let $\oo_L$ be the integral closure of $\oo$ in $L$, which is a Dedekind domain. We denote the set of non-zero prime ideals of $\oo_L$ by $\pp$.

We suppose that $\oo_L$ is finitely generated as an $\oo$-module. This condition holds under very natural assumptions; for instance, if $L/K$ is separable, or $(K, v)$ is complete, or $\oo$ is a finitely generated algebra over a field \cite[Ch.I, \textsection 4]{Serre:1968}. Under this hypothesis, $\oo_L$ is a free $\oo$-module of rank $n$. An $\oo$-basis of $\oo_L$ is called a \emph{$v$-integral basis} of $\oo_L$.

The computation of $v$-integral bases is an essential task in computational arithmetic geometry. We are interested in constructing \emph{reduced} \emph{triangular} $v$-integral bases. Let us clarify these concepts.

\begin{definition}
	\label{def:wp}
	For each prime ideal $\p \in \pp$, we consider the normalised valuation:
	\begin{align*}
		w_{\p} := e(\p/\m)^{-1} v_{\p} \colon L \longrightarrow e(\p/\m)^{-1}\Z\cup\set{\infty},
	\end{align*}
	where $v_{\p}$ is the canonical discrete valuation of $L$ attached to $\p$ and $e(\p/\m)$ is the ramification index. 
Also, we define
	\begin{align*}
		w \colon L\longrightarrow \Q\cup\set{\infty},\quad \alpha\mapsto\Min\set{w_\p(\alpha)\colon \p\in\pp}. 
	\end{align*}
Clearly, an element $\alpha\in L$ belongs to $\oo_L$ if and only if $w(\alpha)\ge 0$.
\end{definition}

\begin{definition}
		A triangular family of elements in $\oo_L$, are elements
		\begin{align*}
			\frac{g_{0}(\theta)}{\pi^{k_{0}}},\, \frac{g_{1}(\theta)}{\pi^{k_{1}}},\, \dots,\, \frac{g_{n-1}(\theta)}{\pi^{k_{n-1}}},
		\end{align*}
		such that for all $0\le i< n$ the polynomial $g_{i}(x) \in \oo[x]$ is monic of degree $i$, and $k_i$ is a non-negative integer such that $k_{i} \le \wt{g_{i}}$.

A triangular basis of $\oo_L$ is a triangular family which is a $v$-integral basis.
\end{definition}
	
\begin{definition}
	A family $\alpha_{1},\, \dots,\, \alpha_{n} \in \oo_L$ is called reduced if for any family $a_{1},\, \dots,\, a_{n} \in \oo$:
	\begin{align*}
		w\left( \sum_{i = 1}^{n} a_{i} \alpha_{i} \right) = \Min\set{ w( a_{i} \alpha_{i})  : 1 \leq i \leq n }.
	\end{align*}
\end{definition}

There are well-known conditions on a triangular family characterising when it is a reduced basis.
  	
\begin{theorem}
	\label{thm:basis}
	Let $\B=\left(g_{0}(\theta)/\pi^{\floor{\nu_{0}}},\, \dots,\, g_{n-1}(\theta)/\pi^{\floor{\nu_{n-1}}}\right)$ be a triangular family in $\oo_L$ with $\nu_i=w(g_i(\theta))$ for all $0\le i<n$. Then,
	\begin{itemize}
		\item[(1)] $\B$ is a triangular basis if and only if $\floor{\nu_{i}} \geq \floor{\wt{h}}$ for all $h \in \oo[x]$ monic of degree $i$ and all $0 \leq i < n$.
		\item[(2)] $\B$ is a reduced triangular basis if and only if $\nu_{i} \geq \wt{h}$ for all $h \in \oo[x]$ monic of degree $i$ and all $0 \leq i < n$.
	\end{itemize}
\end{theorem}
	
Therefore, to construct a reduced triangular $\oo$-basis of $\oo_L$ it suffices to find for each $0\le i<n$ a polynomial $g_{i} \in \oo[x]$ monic of degree $i$ such that $\wt{g_{i}}$ is maximal amongst all monic polynomials in $\oo[x]$ of the same degree.

The aim of the paper is to present the \emph{MaxMin algorithm}, a very simple procedure to construct these optimal polynomials $g_{i}$.

The desire to compute triangular local bases comes from their utility in constructing global bases. Let $A$ be a PID and let $B$ be the integral closure of $A$ in a finite extension of its field of fractions. If $B$ is free as an $A$-module, then an $A$-basis of $B$ can be computed by patching local integral bases of $B_{\m}$ as an $A_{\m}$-module for an adequate finite set of non-zero prime ideals $\m \in \Spec(A)$.

This patching process, usually based on the Chinese remainder theorem, is efficient only if the local bases are triangular. In this case, the global $A$-basis of $B$ one obtains is triangular too.

The property of \textit{reducedness} is important for certain applications in function fields \cite{Bauch:2014:thesis}. The MaxMin algorithm has the advantage of producing local bases which are also reduced.

MaxMin is an OM algorithm, it works with data derived from an \emph{OM factorisation} of the polynomial $f$ in $\oo_v[x]$. In 1999, J. Montes extended some ideas of Ore and MacLane and implemented an algorithm to compute a representation of the prime ideals of $\ooL$ by way of factoring the defining polynomial $f$ over $\oovx$. This ``Montes algorithm'' coupled with work by K. Okutsu on constructing explicit integral bases of local fields, gave rise to several theoretical developments concerning \emph{OM representations} of prime ideals \cite{Fernandez:2013wg,Guardia:2011:algorithm,Guardia:2012:hn,Guardia:2013:newapp,Guardia:2013:genetics,Guardia:2011:sfl,Okutsu:1982:i+ii}.

There are other methods for the computation of integral bases based on a previous computation of an OM factorisation of the defining polynomial $f$. In \cite[Sec. 6]{Guardia:2013:newapp} a method was presented based in the computation of certain \emph{multipliers}, following an old idea of Ore. In \cite{Guardia:2009:bases} a more direct \emph{method of the quotients} was presented, which obtains the numerators of a reduced basis by multiplying certain polynomials obtained as a by-product of the OM factorisation algorithm.

These OM methods are extremely fast in practice and their theoretical complexity is lower than that of the traditional methods based mainly on the Round-2 and Round-4 routines by Zassenhaus and Ford. However, both OM methods yield non-triangular bases, and so a triangularisation routine must be applied to the local bases before they can be used to construct a global basis. These linear procedures are slow in practice and constitute a bottleneck for the whole process.

The MaxMin algorithm yields reduced triangular local bases by a direct method, which avoids the use of linear techniques. This makes MaxMin much more efficient in practice. Another advantage of MaxMin is that the method is equally valid for the computation of bases of fractional ideals. This is particularly useful for the computation of bases of Riemann-Roch spaces of divisors of algebraic curves, which requires the computation of bases of certain fractional ideals attached to the divisor.    

It is common in many computational algebra systems, to provide bases in Hermite Normal Form (HNF). This serves two purposes, the first is that canonical bases simplify the comparison of the rings that they generate. The second is that bases in HNF are triangular, and so patching of global bases from local bases is more efficient. However, the routines used to compute the HNF of a given basis require considerably more time than the computation of a basis.

By using the MaxMin algorithm we can offer two distinct improvements. By computing a triangular basis directly, in many circumstances, we do not need HNF at all, resulting in a significant improvement in execution time. Secondly, if HNF is indeed required, it is faster to compute the HNF of a basis which is already triangular, compared to a random basis.

We review OM representations of prime ideals in Section \ref{sec:montes}. In Section \ref{sec:maxmin} the algorithm MaxMin will be introduced. In Section \ref{sec:computational_examples} we will discuss some computational examples. Finally, the proofs of the two main theorems from Sections \ref{sec:montes} and \ref{sec:maxmin} will be deferred until Sections \ref{sec:proof_opt_polys} and \ref{sec:proof_maxmin}.

%-------------------------------------------------------------------------------
% SECTION: OM representations of prime ideals
%-------------------------------------------------------------------------------
\section{OM representations of prime ideals}
\label{sec:montes}

The prime ideals of $\oo_L$ are in 1-to-1 correspondence with the prime factors of $f$ in $\oovx$. Let $f=\prod_{\p\in\pp}F_\p$ be the factorisation of $f$ into a product of monic irreducible polynomials $F_\p\in\oovx$. Let $n_\p=\deg F_\p=e(\p/\m)f(\p/\m)$.

Inspired by ideas of Ore and MacLane, J. Montes developed an algorithm to compute \emph{OM representations} 
\begin{align}\label{OMrep}
	\type_{\p} = \left( \psi_{0,\p}; (\phi_{1,\p}, \lambda_{1,\p}, \psi_{1,\p}); \dots; (\phi_{r_\p,\p}, \lambda_{r_\p,\p}, \psi_{r_\p,\p}); (\phi_{\p}, \lambda_{\p}, \psi_{\p}) \right),
\end{align}
of each prime factor $F_{\p}$. An object $\type_\p$ as in \eqref{OMrep} is a \emph{type}; it contains several data structured into levels, encoding relevant arithmetic information about the polynomial $F_\p$ and the prime ideal $\p$.
For the precise definition of a type, we refer to \cite{Guardia:2013:genetics}. We now recall some of the properties of the invariants of a type.

The number $r_{\p}+1$ of levels is called the \emph{order} of the type and $r_{\p}$ is the \emph{Okutsu depth} of $F_{\p}$. The family of polynomials $[ \phi_{1,\p}, \dots, \phi_{r_\p,\p} ]$ is an \emph{Okutsu frame} of $F_{\p}$. These are monic polynomials in $\oox$ which are irreducible in $\oovx$.

If we denote $m_i=\deg\phi_{i,\p}$, we have
\begin{align}
	\label{eq:ms}
	m_{1} \divides \cdots \divides m_{r_\p} \divides m_{r_{\p}+1} = n_\p, \qquad
	m_{1} < \cdots < m_{r_\p} < m_{r_{\p}+1} = n_\p.
\end{align}

The polynomial $\phi_{r_{\p}+1,\p} := \phi_\p$ is an \emph{Okutsu approximation} to $F_\p$; it is a monic polynomial in $\oox$ of degree $n_\p$ which is ``sufficiently close" to $F_\p$ for many purposes. More precisely,
\begin{align*}
	\vt{\phi_{\p}} &> \frac{n_{\p}}{m_{r_{\p}}} \vt{\phi_{r_{\p},\p}}.
\end{align*}

The data $\lambda_{i,\p},\lambda_\p$ are positive rational numbers called the \emph{slopes} of the type. Typically, one denotes $\lambda_{i,\p}=h_i/e_i$, $\lambda_{r_{\p}+1,\p} := \lambda_{\p} = h_{r_{\p}+1}/e_{r_{\p}+1}$ the positive (coprime) numerator and denominator of the slope. One has $e_{r_{\p}+1} = 1$.

The type determines a chain of finite extensions of the residue field $\F$:
\begin{align*}
	\F = \F_{0} \subseteq \F_{1} \subseteq \cdots \subseteq \F_{r_{\p}+1} = \F_{\p}, 
\end{align*}
where $\F_{\p}$ is isomorphic to the residue class field of the finite extension of $K_{v}$ determined by $F_{\p}$. The polynomial $\psi_{0,\p}\in\F[y]$ is one of the prime factors of the reduction of $f$ modulo $\m$. The \textit{residual polynomials} $\psi_{i,\p} \in \F_{i}[y]$ are monic and irreducible; they determine the next extension: $\F_{i+1} = \F_{i}[y] / (\psi_{i,\p})$, for all $0 \le i \le r_{\p}$. The monic polynomial $\psi_{r_{\p}+1,\p} := \psi_{\p} \in \F_{\p}[y]$ has degree one. For $i > 0$, one has $\psi_{i,\p} \ne y$.

If we denote $f_{i} := \deg \psi_{i,\p}$, we have
\begin{align}
	\label{eq:es_fs_ms}
	e(\p/\m) = e_{1} \cdots e_{r_{\p}}, \qquad
	f(\p/\m) = f_{0} \cdots f_{r_{\p}}, \qquad
	m_{i} = e_{1} f_{1} \cdots e_{i-1} f_{i-1}.
\end{align}

We may consider each type $\type_{\p}$ as a path with root note $\psi_{0}$ where each level is written along the edges:
\begin{align}
	\begin{aligned}
		\begin{tikzpicture}[scale=0.8]
			\filldraw (1, 0) circle (2.5pt);
			\draw[anchor=east] (0.8, 0) node {$\psi_{0,\p}$};
			\filldraw (4.5, 0) circle (2.5pt);
			\draw (1, 0) -- (4.5, 0);
			\draw[anchor=south] (2.75, 0) node {$(\phi_{1,\p}, \lambda_{1,\p}, \psi_{1,\p})$};
			\draw (5.5, 0) node {$\cdots$};
			\filldraw (6.5, 0) circle (2.5pt);
			\filldraw (10, 0) circle (2.5pt);
			\draw (6.5, 0) -- (10, 0);
			\draw[anchor=south] (8.25, 0) node {$(\phi_{r_{\p},\p}, \lambda_{r_{\p},\p}, \psi_{r_{\p},\p})$};
			\filldraw (13.5, 0) circle (2.5pt);
			\draw[dotted] (10, 0) -- (13.5, 0);
			\draw[anchor=south] (11.75, 0) node {$(\phi_{\p}, \lambda_{\p}, \psi_{\p})$};
			\draw[anchor=west] (13.7, 0) node {$\type_{\p}$};
		\end{tikzpicture}
	\end{aligned}
\end{align}

Each node $\type$ of this path is identified with the type obtained by gathering all level data from the edges joining $\type$ with the root node.

The last edge is dotted to emphasise that $e_{r_{\p}+1} f_{r_{\p}+1} = 1$, while for all other edges we have $e_{i} f_{i} > 1$ as \eqref{eq:ms} and \eqref{eq:es_fs_ms} show.

Actually, the depth $r_{\p}$ and the data $(\phi_{i,\p}, \lambda_{i,\p}, \psi_{i,\p})$ of all levels $i \le r_{\p}$ are (up to certain equivalence relation) canonical data attached to $F_{\p}$ \cite[Sec. 3.3]{Guardia:2013:genetics}. On the other hand, the last level $(\phi_{\p}, \lambda_{\p}, \psi_{\p})$ strongly depends on the choice of $\phi_{\p}$.

For the proof of the next result, see \cite[Sec. 4]{Guardia:2013:genetics}.

\begin{lemma}
	\label{lem:other_type}
	Let $\phi \in \oox$ be any monic polynomial of degree $n_{\p}$ that satisfies $\vt{\phi} \geq \vt{\phi_{\p}}$. Then, for adequate choice of $\lambda \in \Z_{> 0}$ and $a \in \F_{\p}^{*}$, the object:
	\begin{align*}
		\type'_{\p} = \left( \psi_{0,\p}; (\phi_{1,\p}, \lambda_{1,\p}, \psi_{1,\p}); \dots; (\phi_{r_\p,\p}, \lambda_{r_\p,\p}, \psi_{r_\p,\p}); (\phi, \lambda, y - a) \right),
	\end{align*}
	is a type which constitutes an OM representation of $F_{\p}$ too.
\end{lemma}

%Using path notation, a set of prime ideals may be viewed as a tree structure, where levels of different types are combined where they are equal, creating a visual representation of the similarity of the prime ideals. Consider an tree with three end vertices:

The paths corresponding to the different OM representations computed by the Montes algorithm form a tree $\tree$ of types. The leaves of the tree each represent one prime factor of $f$ in $\oo_{v}[x]$. The number of connected components of this tree (i.e. the number of root nodes) is in one-to-one correspondence with the set of prime factors of $\overline{f} = f \pmod{\m}$ in $\F[y]$. For instance:
\begin{align}
	\begin{aligned}
		\begin{tikzpicture}[scale=0.8]
			\filldraw (0, 0) circle (2.5pt);
			\filldraw (2, 0) circle (2.5pt);
			\filldraw (2, 0) circle (2.5pt);
			\filldraw (4, 0.7) circle (2.5pt);
			\filldraw (6, 1.4) circle (2.5pt);
			\filldraw (6, 0) circle (2.5pt);
			\filldraw (4, -0.7) circle (2.5pt);
			\filldraw (6, -0.7) circle (2.5pt);
			\filldraw (8, -0.7) circle (2.5pt);
			\draw (0, 0) -- (2, 0) -- (4, 0.7);
			\draw (2, 0) -- (4, -0.7) -- (6, -0.7);
			\draw[dotted] (4, 0.7) -- (6, 1.4);
			\draw[dotted] (4, 0.7) -- (6, 0);
			\draw[dotted] (6, -0.7) -- (8, -0.7);
			\draw[anchor=east] (-0.2, 0) node {$\psi_{0}$};
			\draw[anchor=west] (6.2, 1.4) node {$\type_{\p}$};
			\draw[anchor=west] (6.2, 0) node {$\type_{\q}$};
			\draw[anchor=west] (8.2, -0.7) node {$\type_{\l}$};
		\end{tikzpicture}
	\end{aligned}
\end{align}

In this example we have three prime ideals, which all share a common first level. Additionally, the prime ideals $\p$ and $\q$ share a common second level as well. Since the tree $\tree$ is connected, $\fbar$ is a power of the prime polynomial $\psi_{0}$ in $\F[y]$. The polynomials $F_{\p}$, $F_{\q}$ have Okutsu depth 2, while $F_{\l}$ has depth 3.

\begin{definition}
	For two prime ideals $\p, \q \in \pp$, the \emph{index of coincidence} $\ell = i(\p, \q)$ is the first different level of their respective types. More precisely, $\ell = 0$ if $\type_{\p}$ and $\type_{\q}$ have different root nodes. If $\psi_{0,\p} = \psi_{0,\q}$, then $\ell$ is minimal such that
	\begin{align*}
		(\phi_{\ell,\p}, \lambda_{\ell,\p}, \psi_{\ell,\p}) \ne (\phi_{\ell,\q}, \lambda_{\ell,\q}, \psi_{\ell,\q}).
	\end{align*}
\end{definition}

One advantage of OM representations of prime ideals, is that they yield explicit formulas for the $\p$-valuation of the $\phi$-polynomials at each level of the type $\type_{\q}$. The following proposition, extracted from \cite[Thm. 3.1]{Guardia:2012:hn} and \cite[Prop. 4.7]{Guardia:2013:newapp}, will be heavily used throughout the paper. It involves certain polynomials $\phi(\p, \q) \in \oox$ and certain \emph{hidden slopes} $\clam{\p}{\q}$, which are secondary data that have been conveniently stored along the running of the Montes algorithm \cite[Sec. 4]{Guardia:2013:newapp}.

\begin{proposition}
	\label{prop:explicit_valuations}
	Let $\p \in \pp$ be a prime ideal of $\oo_L$. Then for any $1 \leq i \leq r_{\p}+1$,
	\begin{align*}
		w_\p(\phi_{i,\p}(\theta)) &= \frac{V_{i,\p} + \lambda_{i,\p}}{e_{1,\p} \cdots e_{i-1,\p}},
	\end{align*}
	where  $V_{1,\p} = 0$ and $V_{i+1,\p} = e_{i,\p} f_{i,\p} ( e_{i,\p} V_{i,\p} + h_{i,\p} )$ for all $1 \le i \le r_{\p}$.
	
	Let $\q \in \pp$ be another prime ideal such that $\p \neq \q$ and with index of coincidence $\ell = i(\p, \q)$. For any $1 \leq i \leq r_{\q}+1$,
	\begin{align*}
%		\label{eq:om_vals}
		w_\p(\phi_{i,\q}(\theta)) &=
		\begin{cases}
			0,
				& \text{if } \ell = 0, \medskip\\
			\dfrac{V_{i} + \lambda_{i}}{\eto{i-1}},
				& \text{if }  i < \ell, \medskip\\
			\dfrac{V_{\ell} + \clam{\p}{\q}}{\eto{\ell-1}},
				& \text{if }  i = \ell \text{ and } \phi_{\ell, \q} = \phi(\p, \q), \medskip\\
			\dfrac{m_{i,\q}}{m_{\ell}} \cdot \dfrac{V_{\ell} + \minclam{\p}{\q}}{\eto{\ell-1}},
				& \text{otherwise}.
		\end{cases}
	\end{align*}
	
	In these formulas, we omit the subscript $\p$, $\q$ when the invariants of the two types coincide.
\end{proposition}

\subsection*{Okutsu bases of the integral closure of $\oo_v$}
\label{sec:okutsu_bases}

Let $\Lp$ be the completion of $L$ with respect to the $\p$-adic topology. We may consider a topological embedding $L \subset \Lp \subset \Kbar_v$, so that $\Lp$ may be identified to a finite extension of $\Kv$ of degree $n_{\p}$. We denote by $\oop$ the integral closure of $\oov$ in $\Lp$.

Let $r$ be the Okutsu depth of $F_{\p}$ and suppose that
\begin{align*}
	\type &= \left( \psi_{0}; (\phi_{1}, \lambda_{1}, \psi_{1}); \dots; (\phi_{r+1}, \lambda_{r+1}, \psi_{r+1}) \right),
\end{align*}
is the leaf corresponding to $F_{\p}$ in the tree of an OM factorisation of $f$. 

The Okutsu frame $[ \phi_{1}, \dots, \phi_{r} ]$ determines optimal polynomials $g_{0,\p}$, $\dots$, $g_{n_{\p}-1,\p}$ in $\oo[x]$ as follows. Each $0 \leq i < n_{\p}$ may be expressed in a unique way as:
\begin{align*}
	i &= a_{0} + a_{1} m_{1} + \cdots + a_{r} m_{r},		& 0 \leq a_{j} < m_{j+1}/m_{j} = e_{j} f_{j}.
\end{align*}
Thus, the polynomials:
\begin{align*}
	g_{i,\p} &:= x^{a_{0}} \prod_{j = 1} \phi_{j}^{a_{j}},		\qquad 0 \leq i < n_{\p},
\end{align*}
are monic polynomials in $\oo[x]$ of degree $\deg g_{i,\p} = i$.

The following result is due to Okutsu \cite{Okutsu:1982:i+ii}.

\begin{theorem}
	\label{thm:okutsu_basis_is_basis}
	For all $0 \leq i < n_{\p}$, the rational number $\wt[\p]{g_{i,\p}}$ is maximal amongst all monic polynomials in $\oovx$ of degree $i$.
\end{theorem}

By Theorem \ref{thm:basis}, we get a reduced triangular $\oo_v$-basis of $\oop$ by taking the images under the embedding $L \subset L_{\p}$ of the elements:
\begin{align}
	\label{eq:redtri_basis}
	\alpha_{i} &:= \pi^{- \floor{w_{\p}(g_{i,\p}(\theta))}} g_{i,\p}(\theta),		\qquad 0 \leq i < n_{\p}.
\end{align}
We call $\left( \alpha_{i} \right)_{0 \leq i < n_{\p}}$ the \emph{Okutsu basis} of $\oop$, or simply the Okutsu $\p$-basis.

If $f$ is irreducible in $\oovx$, then $\pp = \set{\p}$ and $w = w_{\p}$. Theorems \ref{thm:basis} and \ref{thm:okutsu_basis_is_basis} show that in this case, \eqref{eq:redtri_basis} is a reduced triangular $v$-integral basis of $\ooL$. Thus, from now on, we may assume that $\#\pp > 1$.

For further purposes, the family of numerators of an Okutsu $\p$-basis is extended by adding an Okutsu approximation to $F_{\p}$,
\begin{align}
	\label{eq:okutsu_numerators}
	\Num_{\p} &:= \set{ 1 = g_{0,\p},\, \dots,\, g_{n_{\p}-1,\p},\, g_{n_{\p},\p} := \phi_{\p} }.
\end{align}

%\begin{definition}
%	\label{def:wS}
%	Let $S \subseteq \pp$ be a set of prime ideals of $\ooL$, then the $w_{S}$-valuation of an element of $L$ is
%	%
%	\begin{align*}
%		\w[S]{\alpha} &:= \Min\set{ \w[\p]{\alpha} : \p \in S },	\qquad \forall\, \alpha \in L.
%	\end{align*}
%	%
%	Retain the same notation as Definition \ref{def:wp}, we have $w = w_{\pp}$. Additionally, let $n_{S} := \sum_{\p \in S}$.
%\end{definition}

\begin{definition}
	Let $S \subseteq \pp$ be a subset of prime ideals of $\ooL$ and consider the Okutsu set
	\begin{align*}
		\Ok(S) = \set{ \prod_{\p \in S} g_{i_{\p},\p} : 0 \leq i_{\p} \leq n_{\p} }\subset\oo[x],
	\end{align*}
	of all polynomials that are a product of exactly one extended Okutsu $\p$-numerator for each $\p \in S$.
\end{definition}

By construction, all polynomials in $\Ok(S)$ are monic. Note that the Okutsu set depends on the choice of an Okutsu approximation $\phi_\p\approx F_\p$ for each $\p \in S$.  

%\begin{definition}
%	Let $S \subseteq \pp$ be a set of prime ideals of $\ooL$. An \emph{Okutsu $S$-basis} is a reduced triangular family with numerators $g_{i}$ in $\Ok(S)$ for $0 \le i < n_{S}$, such that the $w_{S}$-valuation of each numerator $g_{i}$ is maximal amongst all monic polynomials in $\oovx$ of degree $i$.
%\end{definition}
%
%The following result shows that an Okutsu $S$-basis always exists.

\begin{theorem}
	\label{thm:okutsu_basis_comb_optimal}
	Let $h \in \oox$ be a monic polynomial of degree $0 \le i < n_{S}$. For appropriate choices of the Okutsu approximations $\phi_{\p}$, the set $\Ok(S)$ contains a polynomial $g$ of degree $i$ such that
	\begin{align*}
		\wt[\p]{g} \ge \wt[\p]{h}, \quad \forall\,\p\in S.
	\end{align*}
\end{theorem}

This is one of the main results of the paper, whose proof will be postponed to Section \ref{sec:proof_opt_polys}.

Up to finding the right Okutsu approximations, Theorems \ref{thm:basis} and \ref{thm:okutsu_basis_comb_optimal} show that we may find a reduced triangular basis of $\ooL$ just by finding polynomials of degree $0, 1, \dots, n-1$, with maximal $w$-value in the finite set $\Ok(\pp)$ (c.f. Theorem \ref{thm:maxmin_s_basis}).

Note that a brute force algorithm testing all possible factors $g_{i_{\p},\p}$ of the Okutsu bases leading to polynomials of a fixed degree $i$ would be exponential.

A simple and very efficient algorithm, presented in Section \ref{sec:maxmin}, can be employed to choose an optimal combination of basis numerators for each degree $i$.

\subsection*{Okutsu bases of fractional ideals}

Let $I$ be a non-zero fractional ideal of $\ooL$,
\begin{align*}
	I &= \prod_{\p \in \pp} \p^{a_{\p}}
\end{align*}

We consider a map giving a shifted valuation for a prime ideal $\p$ as a factor of the fractional ideal $I$:
\begin{align*}
%	\label{eq:fracideal_valuation}
	\begin{aligned}
		w_{\p,I} : L &\longrightarrow \Q \cup \{\infty\}, \\
		\alpha &\longmapsto \w[\p,I]{\alpha} = \w[\p]{\alpha} - a_{\p}/e(\p/\m).
	\end{aligned}
\end{align*}
Also, for a subset $S \subseteq \pp$ we define:
\begin{align*}
	\w[S,I]{\alpha} = \Min\set{ \w[\p,I]{\alpha} : \p \in S }.
\end{align*}
Note that an element $\alpha \in L$ belongs to $I$ if and only if $\w[\pp,I]{\alpha} \ge 0$.

Clearly these maps are consistent with the functions given in Definition \ref{def:wp}, as $w_{\p} = w_{\p,\ooL}$ and $w = w_{\pp,\ooL}$.

\begin{definition}
	Let $I$ be a fractional ideal of $\ooL$. Let $S \subseteq \pp$ be a subset of prime ideals of $\ooL$ and denote $n_{S} := \sum_{\p \in S} n_{\p}$. An Okutsu $S$-basis of $I$ is a triangular family
	\begin{align*}
		g_{0}(\theta)/\pi^{\floor{\nu_{0}}}, \dots, g_{n_{S}-1}(\theta)/\pi^{\floor{\nu_{n_{S}-1}}},
	\end{align*}
	with numerators $g_{i}$ in $\Ok(S)$ of degree $i$, such that $\nu_{i} = \wt[S,I]{g_{i}} \ge \wt[S,I]{h}$ for all monic polynomials $h \in \oovx$ of degree $i$, and for all $0 \le i < n_{S}$.
\end{definition}

A result analogous to Theorem \ref{thm:basis} holds: an Okutsu $\pp$-basis of $I$ is a reduced triangular basis of $I$ as an $\oo$-module. Also, Theorem \ref{thm:okutsu_basis_comb_optimal} holds if we replace $w_{\p}$ with $w_{\p,I}$, because both functions differ only in a constant shift. Therefore, just as for the maximal order, it makes sense to compute an Okutsu $S$-basis of $I$ by looking for polynomials in $\Ok(S)$ with a maximal $w_{S,I}$-value amongst all polynomials in $\Ok(S)$ of a given degree. The $\mm$ algorithm serves this purpose.

\section{MaxMin}
\label{sec:maxmin}

\subsection{Formal extension of the Okutsu $\p$-bases}
\label{sec:formal_extension_okutsu_p_basis}

The aim of the $\mm$ algorithm is, given a set of prime ideals $S \subseteq \pp$ and a fractional ideal $I$, to perform an efficient search for $w_{S,I}$-optimal polynomials in $\Ok(S)$.

To decide which numerators are chosen for each degree, we need only to know the values $w_\q(g_{i_{\p},\p}(\theta))$ for all $\p, \q \in S$ and $0 \leq i_{\p} \leq n_{\p}$. As presented in Section \ref{sec:montes}, these values are given by invariants present in an OM factorisation $\tree$ of $f$. The exception is $\w[\p]{\phi_{\p}}$, which can be arbitrarily large, depending on the choice of $\phi_{\p}$ the Okutsu approximation to $F_{\p}$.

For this reason, we do not choose a concrete polynomial $\phi_{\p}$ beforehand, but rather run the algorithm as if $w_\p(\phi_{\p}(\theta))$ (formally) takes the value $\infty$. 

\begin{definition}
	For all $\p \in S$ we define the following function on the Okutsu set:
	\begin{align*} % \phi_{\p}^{a_{r_{\p}+1}}
		w_{\p}:\  \Ok(S) &\longrightarrow \Q \cup \set{ \infty },\quad
		g \longmapsto \begin{cases}
			w_\p(g(\theta)),	& \text{if } \phi_{\p} \nmid g, \\
			\infty,			& \text{if } \phi_{\p} \mid g.
		\end{cases}
	\end{align*}
\end{definition}

This function does not depend on the choice of the Okutsu approximations $\phi_{\p}$ as by Lemma \ref{lem:other_type} and  Proposition \ref{prop:explicit_valuations} the value of $\wt[\q]{\phi_{\p}}$ for $\q \neq \p$ only depends on $\p, \q$ and not the choice of $\phi_{\p}$. Thus, it makes sense to consider symbolic polynomials $\phi_{\p}$ of degree $n_{\p}$. 

We consider a similar extension of the functions $w_{\p,I}$, $w_{S,I}$ to the Okutsu set:
\begin{align*}
	\w[p,I]{g} := \w[\p]{g} - a_{\p} / \epp, \qquad
	\w[S,I]{g} := \Min\set{\w[\p,I]{g} : \p \in S}.
\end{align*}
We have $\w[S,I]{g} < \infty$ for all $g \in \Ok(S)$ with the exception of a single polynomial $\phi_{S} := \prod_{\p \in S} \phi_{\p}$. Also, $\w[S,I]{g} \ge \wt[S,I]{g}$, and equality holds for adequate choices of all $\phi_{\p}$ (depending on the given polynomial $g \in \Ok(S)$, $g \ne \phi_{S}$).

The algorithm will provide a recipe to construct polynomials $g_{i} \in \Ok(S)$ of degree $i$ with a maximal value of $\w[S,I]{g_{i}}$ among all polynomials of degree $i$ in $\Ok(S)$. The corresponding member of the triangular basis will be
\begin{align*}
	\alpha_{i} &= g_{i}(\theta) \pi^{-\lfloor \w[S,I]{g_{i}} \rfloor},	\qquad 0 \leq i < n_{S}.
\end{align*}

For a practical computation of $\alpha_{i}$, we must apply the Single-Factor Lifting algorithm \cite{Guardia:2011:sfl} to find concrete Okutsu approximations $\phi_{\p}$, with a value $\wt[\p]{\phi_{\p}}$ large enough to guarantee that $\w[S,I]{g_{i}} = \wt[S,I]{g_{i}}$ for all $0 \leq i < n_{S}$.

%-------------------------------------------------------------------------------
% Section: The MaxMin algorithm
%-------------------------------------------------------------------------------
\subsection{The MaxMin algorithm}
\label{sec:maxmin_algorithm}

For each type $\type$ in the tree $\tree$, we denote by $S_{\type} \subseteq S$ the subset of prime ideals $\p \in S$ such that $\type$ is one of the nodes in the path joining the leaf $\type_{\p}$ with its root node.

We fix an ordering $S = \set{\p_{1}, \dots, \p_{s}}$ on the set $S$, with the property that for all types $\type$ in $\tree$, the subset $S_{\type} \subseteq S$ is an interval of $S$. That is, there exist indices $1 \leq a_{\type} \leq b_{\type} \leq s$ such that,
\begin{align}
	\label{eq:ordering}
	S_{\type} = [a_{\type}, b_{\type}] := \set{ \p_{j} : a_{\type} \leq j \leq b_{\type} }.
\end{align}

As the branches of $\tree$ do not cross one-another, the reader will easily be convinced that it is always possible to consider such an ordering.

We consider multi-indices $\ii = (i_\p)_{\p \in S}$ of degree $\deg \ii := \sum_{\p} i_{\p}$, leading to monic polynomials $g_{\ii} := \prod_{\p} g_{i_{\p},\p} \in \oox$ in the Okutsu set $\Ok(S)$, with $\deg g_{\ii} = \deg \ii$.

\begin{definition}
	A multi-index $\ii = (i_{\p})_{\p \in S}$ is said to be \emph{maximal} if
	\begin{align*}
		\w[S,I]{g_{\ii}} \geq \w[S,I]{g_{\jj}},
	\end{align*}
	for all multi-indices $\jj$ with $\deg \jj = \deg \ii$.
	
	In this case, we also say that $g_{\ii}$ is a \emph{maximal numerator}.
\end{definition}

\begin{notation}
	For $1 \leq j \leq s$ we denote by $\uu_{j}$ the multi-index with coordinates $i_{\p} = 0$ for all $\p \neq \p_{j}$ and $i_{\p_{j}} = 1$.
\end{notation}

\begin{algorithm}[h]
	\caption{$\mm[S]$ algorithm}
	\label{alg:maxmin_s}
	\begin{algorithmic}[1]
		\INPUT A fractional ideal $I$ of $\ooL$ and Okutsu numerators $\set{g_{i,\p} : 0 \leq i \leq n_{\p}}$ of $\p$-bases for each $\p \in S$.
						
		\OUTPUT A family $\ii_{0}, \ii_{1}, \dots, \ii_{n_{S}} \in \N^{s}$ of multi-indices of degree $0, 1, \dots, n$ respectively.
				
		\STATE $\ii_{0} \gets (0, \dots, 0)$
		
		\FOR {$k = 0 \to n_{S}-1$}
		
			\STATE $j \gets \Min\set{ 1 \leq i \leq s : \w[\p_{i},I]{g_{\ii_{k}}} = \w[S,I]{g_{\ii_{k}}}}$
			
			\STATE $\ii_{k+1} \gets \ii_{k} + \uu_{j}$
			
		\ENDFOR
	\end{algorithmic}
\end{algorithm}

The next result is the second fundamental result in the paper. Its proof will be given in Section \ref{sec:proof_maxmin}.

\begin{theorem}
	\label{thm:maxmin_works}
	All output multi-indices of $\mm$ are maximal. 
\end{theorem}

This gives the name $\mm$ for the algorithm, because it finds the maximal value amongst the minima of certain numerical data. This provides a computation of a reduced triangular basis  as follows.

\begin{theorem}
	\label{thm:maxmin_s_basis}
	Let $\ii_{0}, \ii_{1}, \dots, \ii_{n_{S}}$ be an output of $\mm$. Choose Okutsu approximations $\phi_{\p}$ of all $\p \in \pp$, such that
	\begin{align*}
		\w[S,I]{g_{\ii_{k}}} &= \wt[S,I]{g_{\ii_{k}}},	& 0 \leq k < n_{S}.
	\end{align*}
	
	Then, $g_{\ii_{0}}, g_{\ii_{1}}, \dots, g_{\ii_{n-1}}$ are numerators of an Okutsu $S$-basis of $I$.
\end{theorem}

\begin{proof}
	Let $h \in \oovx$ be a monic polynomial of degree $0 \le k < n_{S}$. By Theorem \ref{thm:okutsu_basis_comb_optimal}, there exists $g \in \Ok(S)$ (for adequate choices of all $\phi_{\p}$ for $\p \in S$) of degree $k$ such that $\wt[S,I]{g} \ge \wt[S,I]{h}$.
	
	On the other hand, regardless of the choices of the $\phi_{\p}$, we have $\w[S,I]{g_{\ii_{k}}} \ge \w[S,I]{g}$ by the maximality of $\ii_{k}$. Hence,
	\begin{align*}
		\wt[S,I]{g_{\ii_{k}}} &= \w[S,I]{g_{\ii_{k}}} \ge \w[S,I]{g} \ge \wt[S,I]{g} \ge \wt[S,I]{h}.
		\qedhere
	\end{align*}
\end{proof}

We will now present some remarks about the behaviour of the algorithm. We assume $I = \ooL$ for simplicity.

\subsubsection{Guaranteed termination}

$\mm$ always terminates after exactly $n_{S}$ iterations.

Thanks to the convention $w_\p(\phi_{\p}) = \infty$, the index $j$ in step 3 indicates a prime $\p_{j}$ such that for the multi-index $\ii_{k} = (i_{\p})_{\p \in S}$, we will always have $i_{\p_{j}} < n_{\p_{j}}$. Thus, the next multi-index $\ii_{k+1} = (i'_{\p})_{\p \in S}$ constructed in step 4 has indices $i'_{\p} \leq n_{\p}$ for all $\p$.

Furthermore, the first and last output multi-indices are $\ii_{0} = (0, \dots, 0)$ and $\ii_{n_{s}} = (n_{\p_{1}}, \dots, n_{\p_{s}})$. As such, $g_{\ii_{0}} = 1$ and $g_{\ii_{n_{S}}} = \prod_{\p \in S} \phi_{\p}$.

\subsubsection{Polynomial products are not computed}

The algorithm does not compute the products $g_{\ii_{k}}$. It only computes the values $w_\p(g_{\ii_{k}})$ for all $\p \in S$, which are determined by the 3-dimensional array of data $w_{\p_{k}}(g_{j_i,\p_{i}})$ indexed by $i$, $j_i$, and $k$ in the ranges $1 \leq i \leq s$, $0 \leq j_{i} \leq n_{\p_{i}}$, and $1 \leq k \leq s$, respectively.

\subsubsection{MaxMin is not a universal maximiser}
If the numbers $w_{\p_{k}}(g_{j_i,\p_{i}})$ are replaced by arbitrary, non-negative rational numbers $\nu_{k, j_{i}, i} \in \Q_{>0}$ and we take
\begin{align*}
	\nu_{k, \ii} := \sum_{i = 1}^{s} \nu_{k, j_{i}, i},
\end{align*}
with $\ii = (j_{i})_{1 \leq i \leq s}$ a multi-index as above, the $\mm$ routine may fail to compute
\begin{align*}
	\Max\set{ \Min\set{ \nu_{k, \ii} : 1 \leq k \leq s } : \deg \ii = d },
\end{align*}
a maximal multi-index of degree $d$.

\subsubsection{Initial conditions}

Suppose $\ii = (i_{\p})_{\p \in S}$ is a multi-index with degree $\deg \ii = d$, such that $\w[S,I]{g_{\ii}}$ is maximal amongst all multi-indices of degree $d$. Then, it may not be true that by increasing an adequate index by one, we get a multi-index $\jj$, of degree $d+1$, which renders a maximal value of $w(g_{\jj})$ amongst all multi-indices of degree $k+1$.

For instance, let us consider the example presented in Section \ref{sec:maxmin_example}. The output index of degree 3 is $\ii_{3} = (1, 2, 0)$, resulting in the polynomial $g_{3} = \phi_{1,\p} \phi_{1,\q}^{2}$ with valuation vector $\wv{g_{3}} = (18, 12, 12)$ for $w_{\p}$, $w_{\q}$ and $w_{\l}$ respectively.

We could choose an alternative index $\jj_{3} = (1, 1, 1)$ which would give a polynomial $g'_{3} = \phi_{1,\p} \phi_{1,\q} \phi_{1,\l}$ with the exact same valuations $\wv{g'_{3}} = (18, 12, 12)$. However, none of the indices $(2, 1, 1)$, $(1, 2, 1)$, $(1, 1, 2)$ is maximal. For instance, $\jj_{4} = (1, 2, 1)$ determines the polynomial $g'_{4} = \phi_{1,\p} \phi_{1,\q}^{2} \phi_{1,\l}$ with valuations $\wv{g'_{4}} = (24, 16, 16)$. This is clearly not maximal as the polynomial constructed by MaxMin $g_{4} = \phi_{1,\p} \phi_{2,\q}$ has valuations $\wv{g_{4}} = (18, 22, 21)$.

It is remarkable that the extremely simple strategy that $\mm$ employs to choose successive maximal multi-indices is able to avoid these pathological cases.

\subsubsection{Ordering of input prime ideals}

Theorem \ref{thm:maxmin_works} shows that $\mm$ produces a sequence of maximal multi-indices regardless of the choice of ordering on $S$, as long as it satisfies \eqref{eq:ordering}. However, the numerators $g_{\ii_{k}}$ produced from these multi-indices do depend on the choice of ordering.

\subsubsection{Complexity}

To compute a triangular $v$-integral basis of $L$, a number of steps are required:
\begin{enumerate}
	\item Use the Montes algorithm to produce an OM representation $\tree$ of $f$.
	\item Run $\mm[\pp]$ to generate a family of maximal indices $\ii_{0}$, $\dots$, $\ii_{n-1}$.
	\item Apply the Single Factor Lifting algorithm from \cite{Guardia:2011:sfl} to get an adequate improvement of the Okutsu approximation of each prime factor of $f$.
	\item Compute the numerators of the Okutsu basis $g_{0}, \dots, g_{n-1}$ specified by the maximal indices.
	\item Divide the Okutsu numerators by the appropriate power of $\pi$ to create an integral basis.
\end{enumerate}

The total complexity is equivalent to that of other OM based routines and is given in the following result.

\begin{theorem}[{\cite[Thm. 6.1]{Stainsby:2014:thesis}}]
	\label{thm:total_complexity}
	Suppose that $\F$ is a finite field with $q$ elements and $f$ is a separable polynomial. Take $\delta := \v{\disc(f)}$. The total cost of the computation of a $v$-integral basis of $L$ by the application of the Montes and the MaxMin algorithms is
	\begin{align*}
		\bigo{n^{2+\epsilon} \delta^{1+\epsilon} + n^{1+\epsilon} \delta \log(q) + n^{1+\epsilon} \delta^{2+\epsilon}},
	\end{align*}
	operations in $\F$. If we assume $q$ small, this will give us a refined estimation of $\bigo{n^{2+\epsilon} \delta^{1+\epsilon} + n^{1+\epsilon} \delta^{2+\epsilon}}$ bit operations.
\end{theorem}

\subsection{MaxMin example}
\label{sec:maxmin_example}

We will now present a small example for $I = \ooL$ and $S = \pp = \set{ \p, \q, \l }$, where $\tree$ is connected. The tree is shown in Figure \ref{fig:maxmin_example}, where we indicate only the data $(\phi, \lambda)$ for each edge.

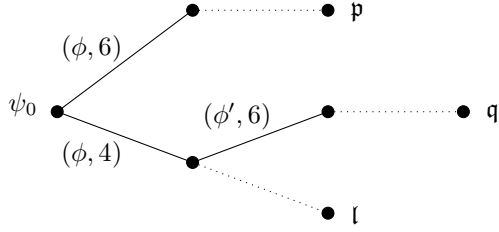
\begin{figure}[htb]
	\centering
	\begin{tikzpicture}[scale=0.9]

		\draw[anchor=south east] (-0.15, -0.2) node {$\psi_{0}$};
		\filldraw (0, 0) circle (2.5pt);
		
		\filldraw (2, 1.5) circle (2.5pt);
		\filldraw (4, 1.5) circle (2.5pt);
		\filldraw (2, -0.75) circle (2.5pt);
		\filldraw (4, 0) circle (2.5pt);
		\filldraw (6, 0) circle (2.5pt);
		\filldraw (4, -1.5) circle (2.5pt);
		
		\draw (0, 0) -- (2, 1.5);
		\draw[dotted] (2, 1.5) -- (4, 1.5);
		\draw (0, 0) -- (2, -0.75);
		\draw (2, -0.75) -- (4, 0);
		\draw[dotted] (4, 0) -- (6, 0);
		\draw[dotted] (2, -0.75) -- (4, -1.5);
		
		\draw[anchor=south east] (1.1, 0.6) node {$(\phi, 6)$};
		\draw[anchor=north east] (1.1, -0.3) node {$(\phi, 4)$};
		\draw[anchor=south east] (3.3, -0.4) node {$(\phi', 6)$};

%		\draw[anchor=north east] (1.4, -0.4) node {$\lambda_{\q}^{\p} = \lambda_{\l}^{\p} = 2$};
%		\draw[anchor=north east] (3.0, -1.0) node {$\lambda_{\l}^{\q} = 5$};
		
		\draw[anchor=west] (4.2, 1.475) node {$\p$};
		\draw[anchor=west] (6.2, -0.025) node {$\q$};
		\draw[anchor=west] (4.2, -1.525) node {$\l$};
	\end{tikzpicture}
	\caption{Example connected tree $\tree$ of types.}
	\label{fig:maxmin_example}
\end{figure}

Since all slopes have integer values, all denominators $e_{i}$ are equal to one. We assume that $f_{0} = m_{1} = \deg \psi_{0} = 1$ and:
\begin{align*}
	\begin{aligned}
	\p :\quad   &	e_{1} = 1, f_{1} = 4, h_{1} = 6;  \\
	\q :\quad   &	e_{1} = 1, f_{1} = 3, h_{1} = 4;  &&e_{2} = 1, f_{2} = 2, h_{2} = 6; \\
	\l :\quad   &	e_{1} = 1, f_{1} = 3, h_{1} = 4.
	\end{aligned}
\end{align*}

Note that $n_{\p} = 4$, $n_{\q} = 6$, and $n_{\l} = 3$, so that $n = 13$.

The data corresponding to the edges leading to a leaf are not specified as we do not need them to run $\mm$.

Suppose moreover that%
\begin{align*}
	\phi(\p, \q) &= \phi(\p, \l) = \phi_{1,\p} = \phi_{1,\q} = \phi_{1,\l} = \phi, \qquad
	\phi(\q, \l) = \phi_{2,\q} = \phi',
\end{align*}
and the hidden slopes are:
\begin{align*}
%	\begin{aligned}
	\clam{\p}{\q} &= \clam{\p}{\l} = 6, \qquad
	\clam{\q}{\p} = \clam{\l}{\p} = 4, \qquad
	\clam{\q}{\l} = 6, \qquad
	\clam{\l}{\q} = 5.
%	
%	\clam{\p}{\q} &= \clam{\p}{\l} = \lambda_{1,\p} = 6, \qquad
%	\clam{\q}{\p} = \clam{\l}{\p} = \lambda_{1,\q} = \lambda_{1,\l} = 4, \quad
%	\clam{\q}{\l} &= \lambda_{2,\q} = 6, \qquad
%	\clam{\l}{\q} = \lambda_{2,\l} = 5.
%	\end{aligned}
\end{align*}

The numerators of the extended Okutsu bases of each prime ideal will be,
\begin{align*}
	\Num_{\p} &:\  1,\  \phi_{1,\p},\  \phi_{1,\p}^{2},\  \phi_{1,\p}^{3},\  \phi_{\p}; \\
	\Num_{\q} &:\  1,\  \phi_{1,\q},\  \phi_{1,\q}^{2},\  \phi_{2,\q},\  \phi_{2,\q} \phi_{1,\q},\  \phi_{2,\q} \phi_{1,\q}^{2},\  \phi_{\q}; \\
	\Num_{\l} &:\  1,\  \phi_{1,\l},\  \phi_{1,\l}^{2},\  \phi_{\l}.
\end{align*}

Using the explicit formulas of Proposition \ref{prop:explicit_valuations}, we may compute the valuations of each of the $\phi$-polynomials. We write them as a tuple $\vec{w} = (w_{\p}, w_{\q}, w_{\l})$.
\begin{align*}
	\begin{aligned}
	&\wv{\phi_{1,\p}} = (6, 4, 4), \quad && \wv{\phi_{\p}} = (\infty, 16, 16), \\
	&\wv{\phi_{1,\q}} = (6, 4, 4), \quad && \wv{\phi_{2,\q}} = (12, 18, 17) \quad && \wv{\phi_{\q}} = (24, \infty, 34), \\
	&\wv{\phi_{1,\l}} = (6, 4, 4), \quad && \wv{\phi_{\l}} = (12, 17, \infty).
	\end{aligned}
\end{align*}

We can now step through the results of running $\mm$. The ``minimal'' valuation is underlined at each step. This indicates the index which will be incremented in the following step.
\begin{center}
	\begin{tabular}{c | c | c | c}
		$i$ & $g_{i}$ & $\wv{g_{i}}$ & $w(g_{i})$ \\
		\hline
		$0$  & $1 \cdot 1 \cdot 1$                              					& $(\underline{0}, 0, 0)$ &  0 \\
		$1$  & $\phi_{1,\p} \cdot 1 \cdot 1$ \quad			    			& $(6, \underline{4}, 4)$ & 4 \\
		$2$  & $\phi_{1,\p} \cdot \phi_{1,\q} \cdot  1$          				& $(12, \underline{8}, 8)$ & 8 \\
		$3$  & $\phi_{1,\p} \cdot \phi_{1,\q}^{2} \cdot  1$      				& $(18, \underline{12}, 12)$ & 12 \\
		$4$  & $\phi_{1,\p} \cdot \phi_{2,\q} \cdot  1$          				& $(\underline{18}, 22, 21)$ & 18 \\
		$5$  & $\phi_{1,\p}^{2} \cdot \phi_{2,\q} \cdot  1$      				& $(\underline{24}, 26, 25)$ & 24 \\
		$6$  & $\phi_{1,\p}^{3} \cdot \phi_{2,\q} \cdot  1$      				& $(30, 30, \underline{29})$ & 29 \\
	\end{tabular}
	
	\begin{tabular}{c | c | c | c}
		$7$  & $\phi_{1,\p}^{3} \cdot \phi_{2,\q} \cdot  \phi_{1,\l}$       		& $(36, 34, \underline{33})$ & 33 \\
		$8$  & $\phi_{1,\p}^{3} \cdot \phi_{2,\q} \cdot  \phi_{1,\l}^{2}$   		& $(42, 38, \underline{37})$ & 37 \\
		$9$  & $\phi_{1,\p}^{3} \cdot \phi_{2,\q} \cdot  \phi_{\l}$         		& $(\underline{42}, 47, \infty)$ & 42 \\
		$10$ & $\phi_{\p} \cdot \phi_{2,\q} \cdot  \phi_{\l}$              			& $(\infty, \underline{51}, \infty)$ & 51 \\
		$11$ & $\phi_{\p} \cdot \phi_{2,\q} \phi_{1,\q} \cdot  \phi_{\l}$   		& $(\infty, \underline{55}, \infty)$ & 55 \\
		$12$ & $\phi_{\p} \cdot \phi_{2,\q} \phi_{1,\q}^{2} \cdot  \phi_{\l}$	& $(\infty, \underline{59}, \infty)$ & 59 \\
		$13$ & $\phi_{\p} \cdot \phi_{\q} \cdot  \phi_{\l}$                     		& $(\infty, \infty, \infty)$ & $\infty$
	\end{tabular}
\end{center}

The final element $g_{13}$  is not included in the $v$-integral basis.

% Stick this code into the MaxMin server to get the diagram of the result. Include this here.
%
%{
%  "j": [
%    [1, 1], 
%    [2]
%  ],
%  "hidden": [
%    [0, 0, 0], 
%    [0, 0, 0], 
%    [0, 0, 0]
%  ],
%  "types": [
%    [
%      {"e": 1, "f": 4, "h": 6}, 
%      {"e": 1, "f": 1, "h": 18}
%    ],
%    [
%      {"e": 1, "f": 3, "h": 4}, 
%      {"e": 1, "f": 2, "h": 6}, 
%      {"e": 1, "f": 1, "h": 12}
%    ],
%    [
%      {"e": 1, "f": 3, "h": 4}, 
%      {"e": 1, "f": 1, "h": 5}
%    ]
%  ]
%}

\section{Computational examples}
\label{sec:computational_examples}

In this section, we will present a number of example computations using an implementation of the MaxMin algorithm for the computer algebra system Magma \cite{Bosma:1997:magma}. We compare MaxMin's execution time for computing $v$-integral bases with that of the method of the quotients, another OM-based algorithm, as well as the internal routines found in Magma.

All executions were performed on GNU/Linux running on 8-core 3.0GHz nodes with 32GB main memory. Each execution ran in a single core, using Magma 2.18-5.

Examples will be given for number fields $L = \Q[x]/(f)$ for polynomials $f \in \Z[x]$ and function fields $L = \F_{q}(t)[x]/(f)$ over a finite field $\F_{q}$ for polynomials $f \in \F_{q}[t][x]$. The example defining polynomials are taken from \cite{Guardia:2011:sfl}.

The first example is comprised of the $B$-class of polynomials,
\begin{align*}
	B_{p,k}(x) = (x^2 - 2x + 4)^{3} + p^{k},
\end{align*}
of degree 6. We take $f(x) = B_{13,k}(x) \in \Z[x]$ in the number field case and $f(x) = B_{t^{3}+2,k}(x) \in \F_{7}[t,x]$ in the function field case, with $k \le 5000$. The execution times for computing a Hermitian $v$-integral basis of $L$ are shown in Figure \ref{fig:B_maxorder}.

\begin{figure}[htb]
	\centering
	\begin{subfigure}[b]{0.47\textwidth}
		\includegraphics[width=\textwidth]{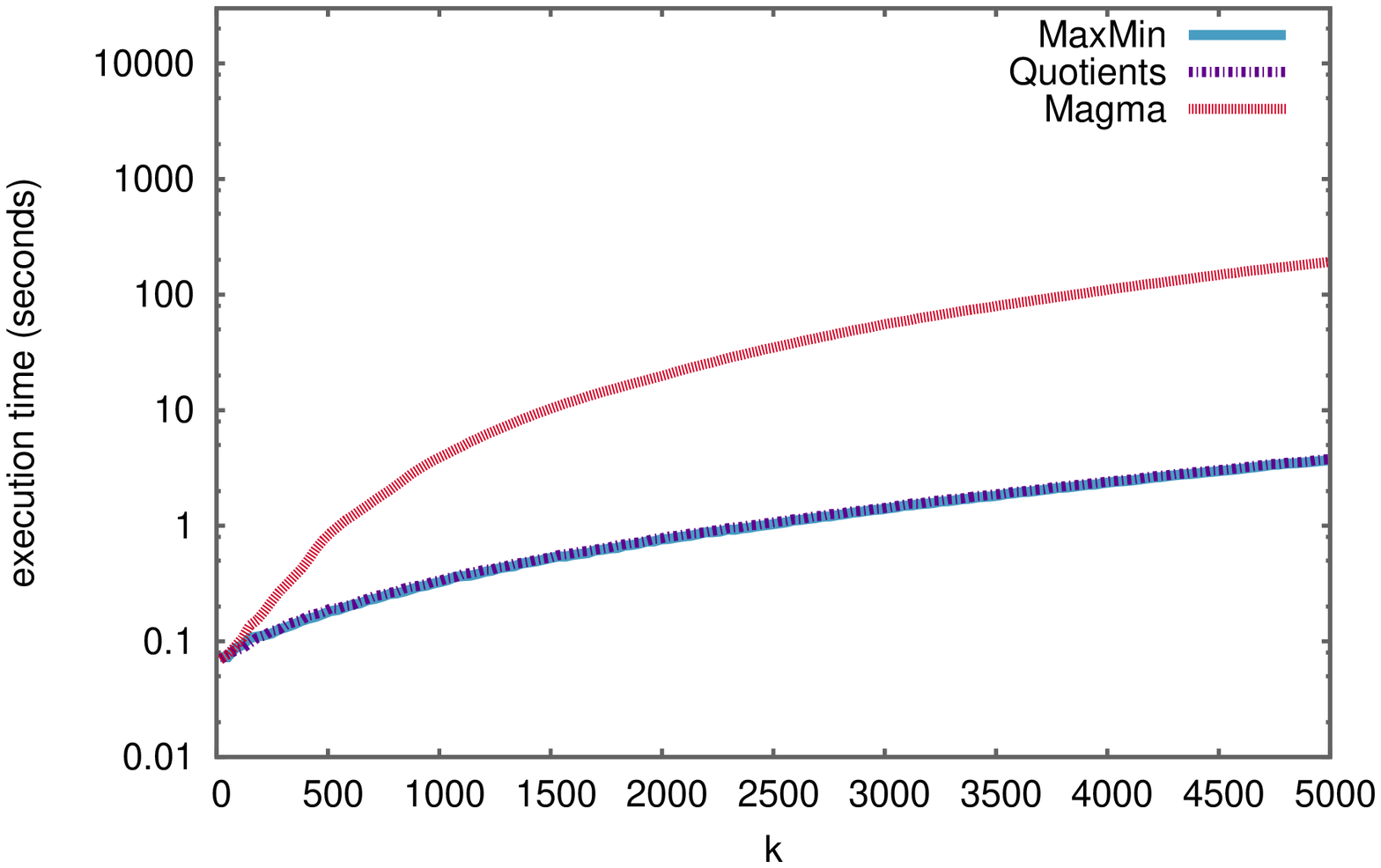}
		\caption{$B_{13,k}(x) \in \Z[x]$.}
		\label{fig:nf_B_maxorder}
	\end{subfigure}
	\hfill
	\begin{subfigure}[b]{0.47\textwidth}
		\includegraphics[width=\textwidth]{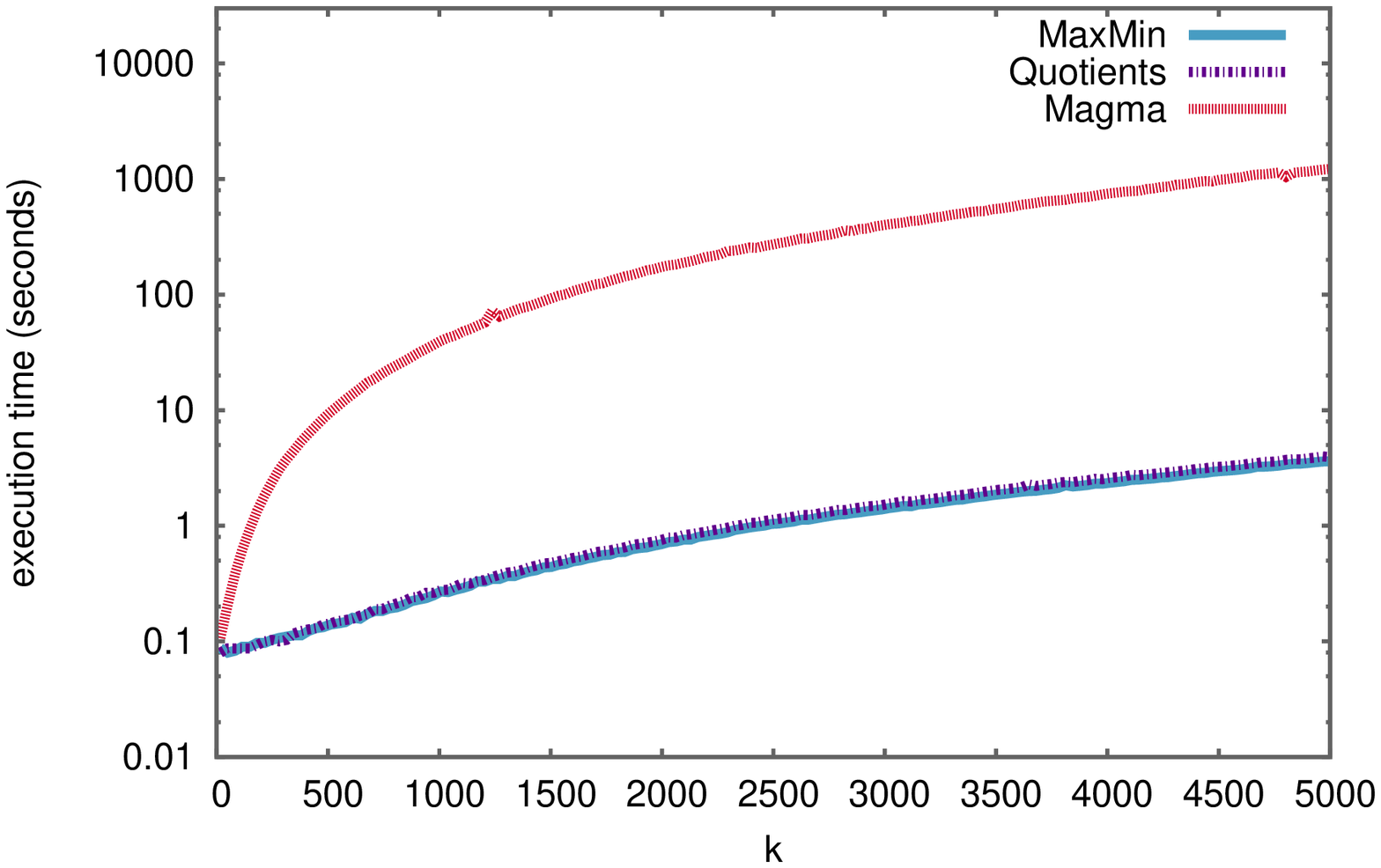}
		\caption{$B_{t^{3}+2,k}(x) \in \F_{7}[t,x]$.}
		\label{fig:ff_B_maxorder}
	\end{subfigure}
	\caption{Running time for maximal order Hermitian $p$-basis computation defined by polynomials $B_{p,k}(x)$ with $k \le 5000$.}
	\label{fig:B_maxorder}
\end{figure}

Due to the low degree of the field extensions, computing the HNF of a basis is negligible compared to computing the basis itself. It can be seen that both OM-based routines have a similar performance and are roughly 100 times faster than the internal Magma routines in the number field case for $k = 5000$ and 1000 times faster in the function field case.

In order to demonstrate the performance of MaxMin on larger polynomials, we will consider the $A$-class of polynomials
\begin{align*}
	A_{p,n,k}^{m}(x) = (x^{n} + 2 p^{k})( (x + 2)^{n} + 2 p^{k}) \cdots ( (x + 2m - 2)^{n} + 2p^{k} ) + 2p^{nmk}.
\end{align*}
These polynomials have degree $nm$. In the number field case, we take $f(x) = A_{101,n,29}^{4}(x) \in \Z[x]$ with $n \le 100$. Figure \ref{fig:nf_Am_maxorder} shows the times for the OM-based methods and the total times when we include the time to compute the Hermitian basis.

\begin{figure}[htb]
	\centering
	\begin{subfigure}[b]{0.47\textwidth}
		\includegraphics[width=\textwidth]{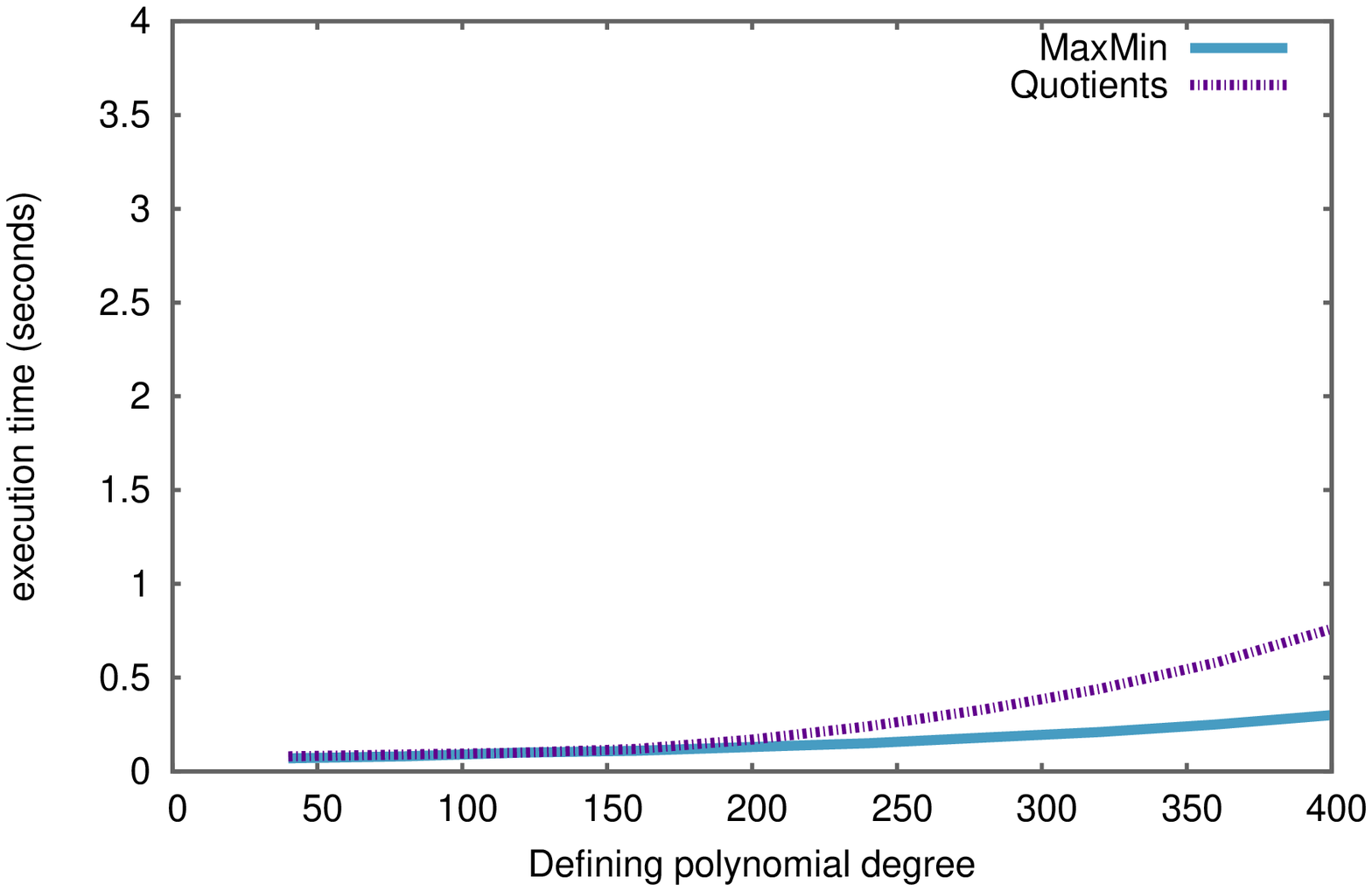}
		\caption{Without HNF.}
		\label{fig:nf_Am_maxorder_nohnf}
	\end{subfigure}
	\hfill
	\begin{subfigure}[b]{0.47\textwidth}
		\includegraphics[width=\textwidth]{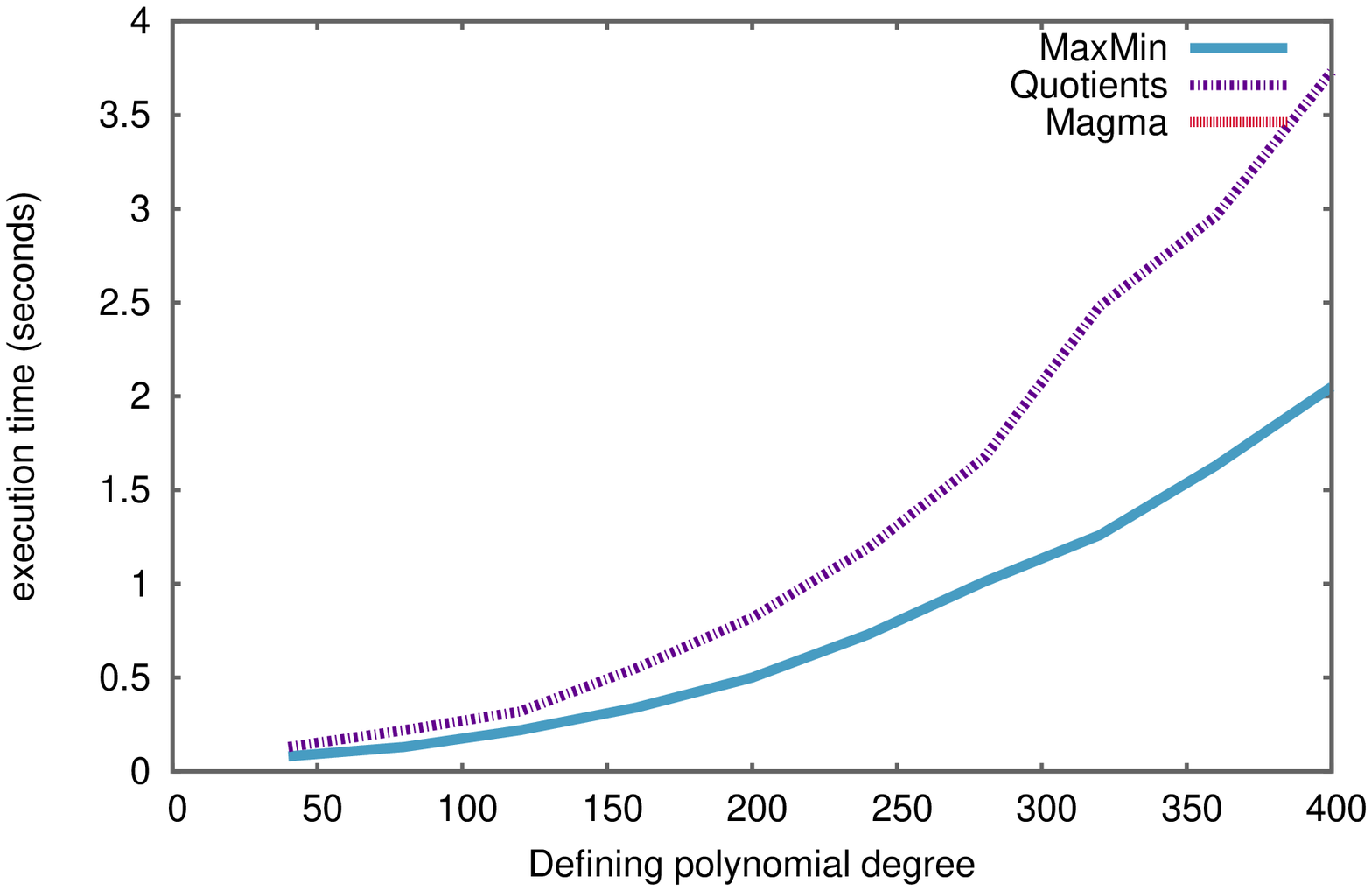}
		\caption{With HNF.}
		\label{fig:nf_Am_maxorder_hnf}
	\end{subfigure}
	\caption{Running time for maximal order $p$-basis computation defined by polynomials $A_{101,n,29}^{4}(x)$ with $\deg(A_{101,n,29}^{4}) \le 400$.}
	\label{fig:nf_Am_maxorder}
\end{figure}

From this example, we can see that MaxMin is somewhat faster than the Method of the Quotients, however when the time to compute the HNF of the resulting basis is included, we see the advantage of the triangular basis computed by MaxMin. In this example, Magma took 257 seconds to compute the basis for $\deg f = 40$ (and does not appear in the figure), and was unable to complete the computation for $\deg f = 80$ due to main memory limitations.

We consider a slightly smaller example in the function field case, with $f(x) = A_{t^2+2,n,6}^{3} \in \F_{7}[t,x]$, once again for $n \le 100$. The resulting execution times are presented in Figure \ref{fig:ff_Am_maxorder}.

\begin{figure}[htb]
	\centering
	\begin{subfigure}[b]{0.47\textwidth}
		\includegraphics[width=\textwidth]{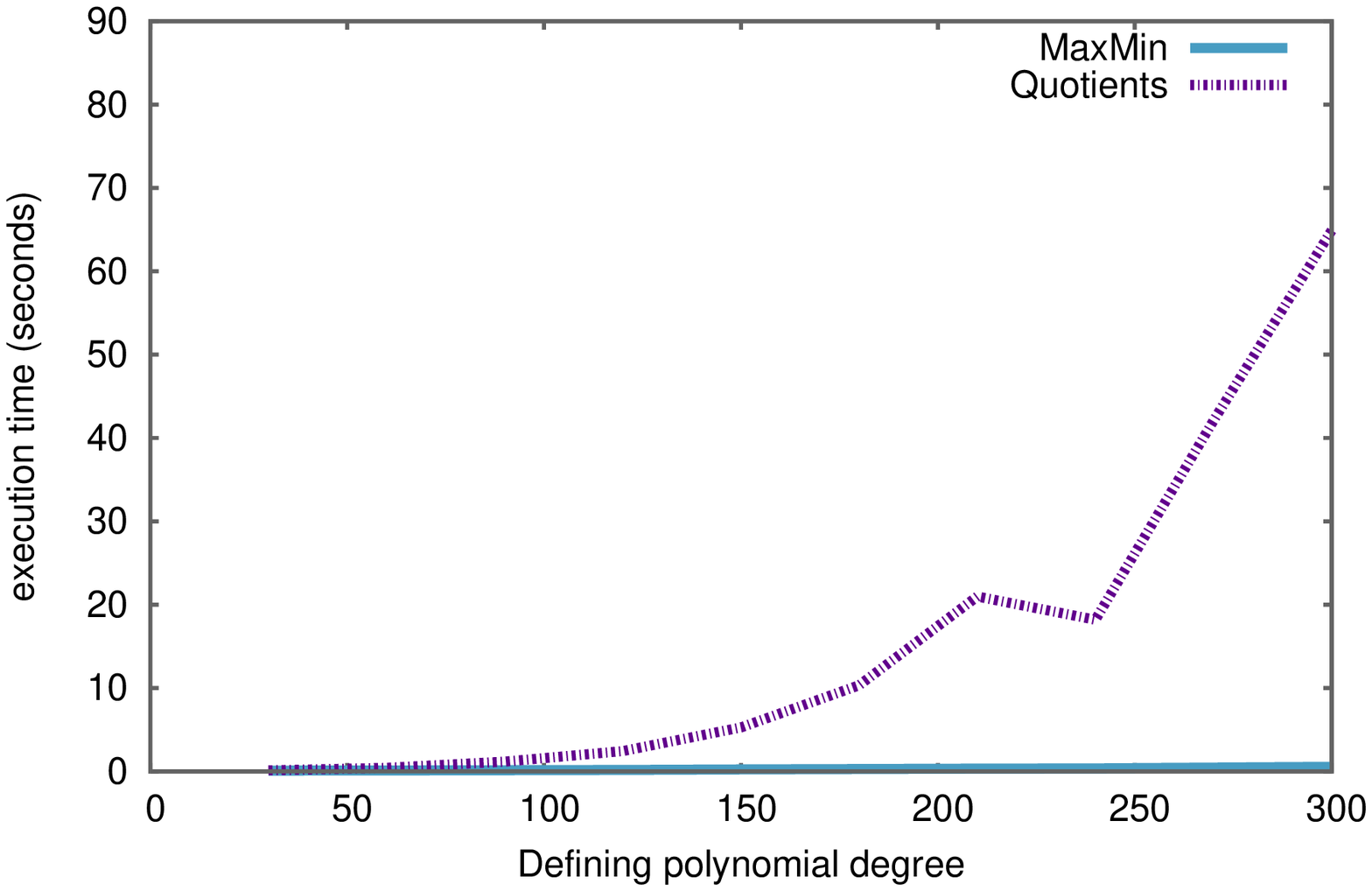}
		\caption{Without HNF.}
		\label{fig:ff_Am_maxorder_nohnf}
	\end{subfigure}
	\hfill
	\begin{subfigure}[b]{0.47\textwidth}
		\includegraphics[width=\textwidth]{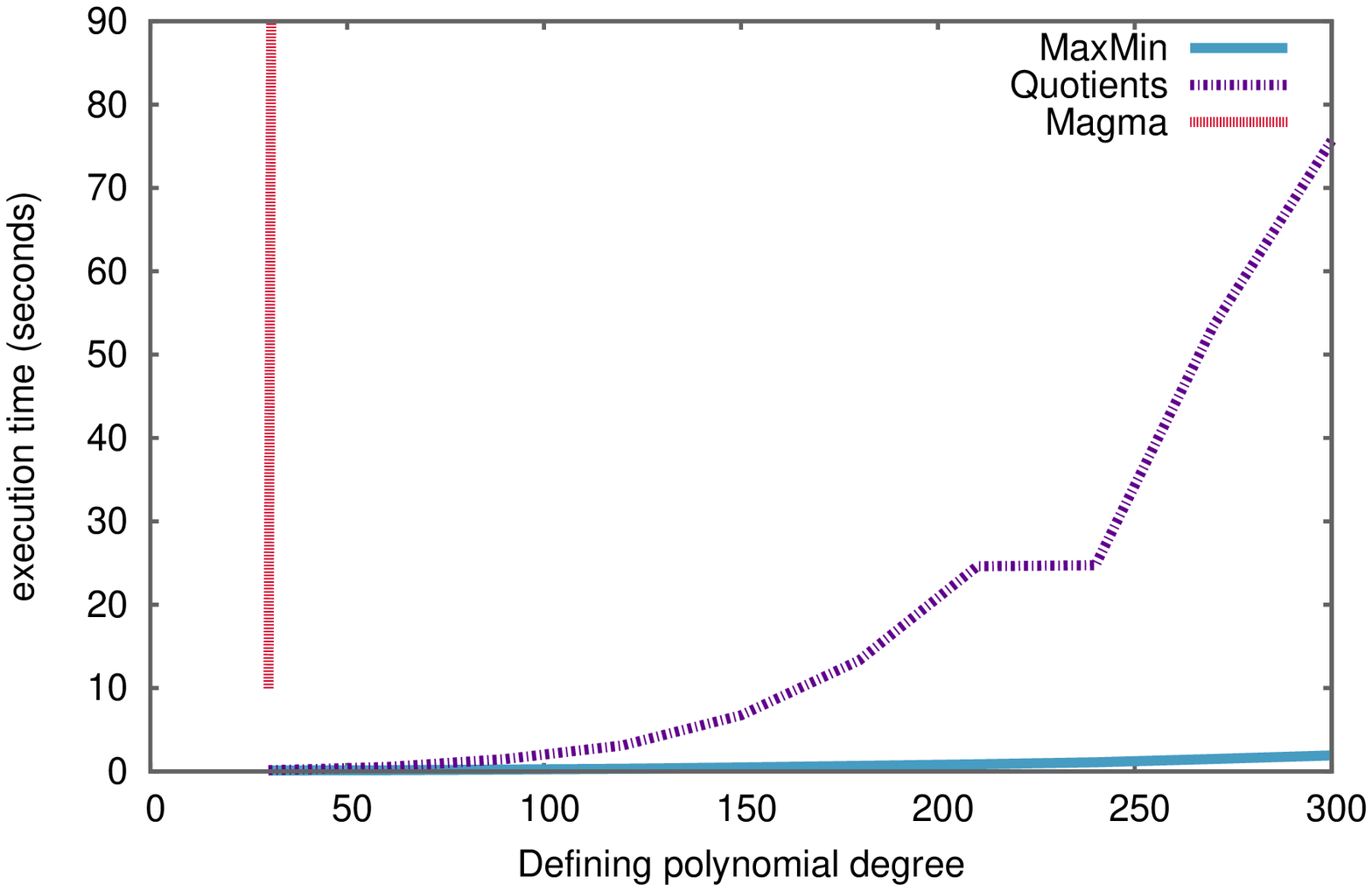}
		\caption{With HNF.}
		\label{fig:ff_Am_maxorder_hnf}
	\end{subfigure}
	\caption{Running time for maximal order $p$-basis computation defined by polynomials $A_{101,n,29}^{3}(x)$ with $\deg(A_{101,n,29}^{4}) \le 400$.}
	\label{fig:ff_Am_maxorder}
\end{figure}

Here, MaxMin computes a basis considerably faster than the Method of the Quotients, however the improvement in computing the HNF from the triangular basis is not as pronounced as in the number field case. The second data point for the Magma routine is at $\deg f = 60$, which took 3304 seconds to compute a Hermitian basis.

Finally, we will consider some very large examples. We take the recursively defined $EC$-class polynomials defined in this case as $EC_{p,j}(x) = E_{p,j}(x) \cdot C_{p,28} + p^{900}$, where the $E$-class and $C$-class polynomials are given in \cite{Guardia:2011:sfl}. Firstly, we consider $f(x) = EC_{101,8}(x) \in \Z[x]$, a degree $900$ polynomial which splits into six degree 6 factors and one degree 864 factor over the $p$-adics. Secondly, we consider $f(x) = EC_{t^{2}+4,4}(x) \in \F_{7}[t,x]$, a degree 72 polynomial that splits into four degree 9 factors and one degree 36 factor over the $p(t)$-adics.

%In Table \ref{tab:EC_maxorder} the times to compute standard and Hermitian bases are given in seconds.

\begin{table}[htb]
	\centering
	\caption{Running time (in seconds) for maximal order $p$-basis computation defined by polynomials $EC_{p,j}$.}
	\label{tab:EC_maxorder}
	\begin{subtable}[b]{0.49\textwidth}
		\centering
		\caption{$EC_{101,8}(x) \in \Z[x]$.}
		\label{fig:nf_EC_maxorder}
		\begin{tabular}{l | c | c}
			\textit{Algorithm}& \textit{Basis}	& \textit{HNF basis} \\
			\hline
			\hline
			MaxMin		& 9.9		& 112.6 \\
			\hline
			Quotients		& 21.1	& 429.3 \\
		\end{tabular}
	\end{subtable}
	\begin{subtable}[b]{0.49\textwidth}
		\centering
		\caption{$EC_{t^{2}+4,4}(x) \in \F_{7}[t,x]$.}
		\label{fig:ff_EC_maxorder}
		\begin{tabular}{l | c | c}
			\textit{Algorithm}& \textit{Basis}	& \textit{HNF basis} \\
			\hline
			\hline
			MaxMin		& 13.3	& 21.5 \\
			\hline
			Quotients		& 89.5	& 8353.8 \\
		\end{tabular}
	\end{subtable}
\end{table}

These final examples show advantages of using MaxMin to compute a local basis, whether or not a subsequent step to pass the basis to HNF is required. However, they also illustrate the advantages of using a triangular basis instead of a basis in HNF where it is possible.

\section{Optimal polynomials in the Okutsu set}
\label{sec:proof_opt_polys}

In this section, we will present the additional details necessary to prove Theorem \ref{thm:okutsu_basis_comb_optimal}. This will be broken into two parts, Proposition \ref{thm:optimal_in_phi} and Proposition \ref{prop:okutsu_basis_comb_optimal}, which together prove the theorem.

\subsection{Operators associated with a type}
\label{sec:type_operators}

Consider a type $\type$ of order $r$ over $(K, v)$:
\begin{align*}
	\type = \left( \psi_{0}; (\phi_{1}, \lambda_{1}, \psi_{1}); \dots; (\phi_{r}, \lambda_{r}, \psi_{r}) \right).
\end{align*}

The type $\type$ carries three kinds of operators. There are normalised valuations:
\begin{align*}
	v_{i} &: \Kvx \longrightarrow \Z \cup \set{\infty}, \qquad 0 \le i \le r.
\end{align*}
The last valuation $v_{r}$ is also denoted $v_{\type}$. Also, we have Newton polygon operators:
\begin{align*}
	\Newton_{i} = \Newton_{v_{i-1},\phi_{i}} &: \Kvx \longrightarrow 2^{\R^{2}}, \qquad 1 \le i \le r,
\end{align*}
where $2^{\R^{2}}$ is the set of subsets of $\R^{2}$. The image of $0$ is the empty set. For every non-zero $g \in \Kx$ we consider the canonical $\phi_{i}$-expansion $g = \sum_{0 \le s} a_{s} \phi_{i}^{s}$, where $a_{s} \in \Kx$ have degree less than $\deg \phi_{i}$. Then $\Newton_{i}(g)$ is the lower convex hull of the set of points $(s, \v[i-1]{a_{s}\phi_{i}^{s}})$.

Finally, we have residual polynomial operators:
\begin{align*}
	\Respol_{i} = \Respol_{v_{i-1},\phi_{i},\lambda_{i}} &: \Kvx \longrightarrow \F_{i}[y], \qquad 0 \le i \le r,
\end{align*}
which are multiplicative: $\Respol_{i}(gh) = \Respol_{i}(g) \Respol_{i}(h)$ for all $g, h \in \Kx$.

The valuation $v_{0}$ is defined as
\begin{align*}
	\v[0]{\sum_{0 \le s} a_{s} x^{s}} = \Min\set{\v{a_{s}} : 0 \le s}.
\end{align*}
The residual polynomial operator $\Respol_{0}$ is defined as:
\begin{align*}
	\Respol_{0}(g) = \pi^{-\v[0]{g}} g(y) \pmod{\m[y]}.
\end{align*}

For $1 \le i \le r$ the operators $v_{i}, \Respol_{i}$ are defined in a certain recurrent way \cite[Sec. 3.1]{Guardia:2013:genetics}.

\begin{definition}
	We say that $\type$ is \emph{optimal} if $m_{1} < m_{2} < \cdots < m_{r}$.
\end{definition}

\begin{definition}
	Let $\type$ be a type of order $r \ge 0$ and $g \in \Kvx$. We define $\ord_{\type}(g)$ as the non-negative integer $\ord_{\psi_{r}}(\Respol_{r}(g))$ where $\psi_{r}$ and $\Respol_{r}(g)$ are considered as polynomials in $\F_{r}[y]$. If $\ord_{\type}(g) > 0$ we say that $\type$ \emph{divides} $g$ and we write $\type \divides g$.
\end{definition}

\begin{definition}
	Let $\type$ be a type of order $r$ and $g \in \oox$. We say that $g$ is a \emph{representative} of $\type$ if $g$ is a monic polynomial of degree $m_{r+1} := e_{r} f_{r} m_{r}$ such that $\Respol_{r}(g) = \psi_{r}$. We denote by $\Rep(\type)$ the set of representatives of $\type$.
\end{definition}

If $\phi \in \oox$ is a representative of $\type$, then by definition $\tilde{\type} = (\type; (\phi, \lambda, \psi))$ is a type of order $r + 1$ for any choice of a positive rational number $\lambda$ and a monic irreducible polynomial $\psi \in \F_{r+1}[y]$, $\psi \ne y$.

\begin{definition}
	A \emph{prime polynomial} is a monic irreducible polynomial in $\oovx$.
\end{definition}

Since $\Respol_{r}$ is multiplicative and $\psi_{r}$ is irreducible, the representatives of $\type$ are prime polynomials.

The next result collects the essential properties of the polynomials which are divisible by a type. It is taken from \cite[Lem. 1.4, Thms. 3.1, 3.7]{Guardia:2012:hn}.

\begin{theorem}
	\label{thm:type_F_props}
	Let $\type$ be a type of order $r \ge 0$ and let $\phi \in \oox$ be a representative of $\type$. Let $F \in \oovx$ be a prime polynomial and choose $\theta \in \Kvbar$ a root of $F$.
	
	If $\phi \ne F$ and $\type \divides F$, then
	\begin{enumerate}
		\item $\Newton_{v_{r},\phi}(F)$ is one-sided of slope $-\lambda$, for a certain positive rational number $\lambda$ such that $\vt{\phi} = (V_{r+1} + \lambda) / (\eto{r})$.
		\item $\deg(F) = \deg(\phi^{\ell})$, where $\ell = \ell(\Newton_{v_{r},\phi}(F))$ is the \emph{length} of the Newton polygon; that is the abscissa of the right end point.
		\item $\deg(\Respol_{v_{r},\phi,\lambda}(F)) = e_{\lambda}^{-1} \ell$ and $\Respol_{v_{r},\phi,\lambda}(F) = \psi^{a}$, where $e_{\lambda}$ is the least positive denominator of $\lambda$ and $\psi \in \F_{r+1}[y]$ is a monic irreducible polynomial $\psi \ne y$.
		\item The type $\tilde{\type} = (\type; (\phi, \lambda, \psi))$ divides $F$.
	\end{enumerate}
\end{theorem}

For $0 \le i \le r$, let $\type_{i} = \Trunc_{i}(\type)$ be the type of order $i$ obtained by dropping all levels $(\phi_{j}, \lambda_{j}, \psi_{j})$ of order $j > i$ from $\type$.

By the definition of a type, each $\phi_{i}$ is a representative of $\type_{i-1}$, for $1 \le i \le r$.

The next result follows easily from Theorem \ref{thm:type_F_props} and \cite[Lem. 2.4]{Guardia:2012:hn}.

\begin{theorem}
	\label{thm:type_divides_F}
	With the above notation, $\type_{i} \divides F$ for all $0 \le i \le r$, $\Newton_{i}(F)$ is one sided of slope $-\lambda_{i}$, and $\Respol_{i}(F)$ is a power of $\psi_{i}$.
\end{theorem}

Finally, the following result follows from Theorem \ref{thm:type_F_props} and  \cite[Prop. 3.5]{Guardia:2012:hn}.

\begin{theorem}
	\label{thm:other_poly_value}
	With the above notation, let $g \in \oovx$ be another prime polynomial such that $\type \divides g$, $\tilde{\type} \ndivides g$. Then,
	\begin{align*}
		\vt{g} = \frac{\deg(g)}{\deg(\phi)} \frac{V_{r+1} + \Min\set{\lambda, \lambda'}}{\eto{r}},
	\end{align*}
	where $-\lambda'$ is the slope of $\Newton_{v_{r},\phi}(g)$.
\end{theorem}

\subsection{Non-optimised tree of types}
\label{sec:nonoptimised_tree}

Let us briefly describe how the Montes algorithm constructs the tree $\tree$ of OM representations of the prime factors of the input polynomial $f \in \oox$. Initially, $\fbar$ is factorised in $\Fy$. For each monic irreducible factor $\varphi$ of $\fbar$, a triplet $(\type, \phi, \omega)$ is considered, where $\type = (\varphi)$ is the type of order $0$ determined by $\varphi$, $\phi$ is a representative of $\type$ (that is, a monic lift of $\varphi$ to $\oox$), and $\omega = \ord_{\type}(f) = \ord_{\varphi}(\fbar)$. All these triplets $(\type, \phi, \omega)$ are stored in a stack.

Along the execution of the algorithm the stack always contains triplets $(\type, \phi, \omega)$, where $\type \divides f$, $\phi$ is a representative of $\type$ and $\omega = \ord_{\type}(f)$. The main loop of the algorithm takes such a triplet and attaches to the type $\type$ one or more branches $\type_{\lambda,\psi} := (\type; (\phi, \lambda, \psi))$ of $\type$ such that $\type_{\lambda,\psi} \divides f$ and the pairs $(\lambda,\psi)$ are considered as follows,
\begin{itemize}
	\item $-\lambda$ runs on the slopes of  $\Newton_{v_{\type},\phi}^{\omega}(f) := \Newton_{v_{\type},\phi}(f) \cap \left( [0, \omega] \times \R \right)$.
	\item $\psi$ runs on the prime factors of $\Respol_{v_{\type},\phi,\lambda}(f)$.
\end{itemize}

Let $\phi_{\lambda,\psi} \in \oox$ be a representative of $\type_{\lambda,\psi}$ and take $\omega_{\lambda,\psi} = \ord_{\psi}(\Respol_{v_{\type},\phi,\lambda}(f)) = \ord_{\type_{\lambda,\psi}}(f)$. If this positive integer is equal to one, then $\type_{\lambda,\psi}$ divides only one of the prime factors $F_{\p}$ of $f$ in $\oovx$. In this case, we add a final level to $\type$ to construct the leaf $\type_{\p}$ associated with this prime factor. On the other hand, if $\omega_{\lambda,\psi} > 1$, then the triplet $(\type_{\lambda,\psi}, \phi_{\lambda,\psi}, \omega_{\lambda,\psi})$ is pushed back onto the stack to bare further branching in future iterations of the main loop.

After a finite number of iterations of this process, the algorithm outputs a list $\type_{1}, \dots, \type_{N}$ of types parametrising  the prime factors of $f$ in $\oovx$. Let us denote by $\treenop$ the tree of types obtained by this procedure. Note that for every node $\type \in \treenop$ which is not a leaf, the edges with left end point $\type$ have the same $\phi$-polynomial; a tree of types with this property is said to be \emph{coherent}.

\begin{figure}[h]
	\centering
	\begin{tikzpicture}[scale=0.9]
		\path [use as bounding box] (0.4, -2.05) rectangle (5.75, 2.05);

		\draw[anchor=south] (2, 0.1) node {$\type$};
		\filldraw (2, 0) circle (2.5pt);
		
		\filldraw (5, 1.5) circle (2.5pt);
		\filldraw (5, 0) circle (2.5pt);
		\filldraw (5, -1.5) circle (2.5pt);
		
		\draw (1, 0) -- (2, 0);
		\draw (2, 0) -- (5, 1.5);
		\draw (2, 0) -- (5, 0);
		\draw (2, 0) -- (5, -1.5);

		\draw (0.7, 0) node {$\cdots$};
		
		\draw (5.5, 0) node {$\cdots$};
		\draw (5.4, 0.45) node {\reflectbox{$\ddots$}};
		\draw (5.4, -0.2) node {$\ddots$};

		\draw (5.5, 1.5) node {$\cdots$};
		\draw (5.4, 1.95) node {\reflectbox{$\ddots$}};
		\draw (5.4, 1.3) node {$\ddots$};

		\draw (5.5, -1.5) node {$\cdots$};
		\draw (5.4, -1.05) node {\reflectbox{$\ddots$}};
		\draw (5.4, -1.7) node {$\ddots$};
		
		\draw[anchor=south] (4, 0) node {$(\phi, \lambda', \psi')$};
		\draw[anchor=south east] (3.7, 0.7) node {$(\phi, \lambda, \psi)$};
		\draw[anchor=north east] (3.7, -0.7) node {$(\phi, \lambda'', \psi'')$};

	\end{tikzpicture}
	\caption{A segment of the non-optimised tree $\treenop$.}
	\label{fig:tree_coherent}
\end{figure}

This describes a kind of ``non-optimised'' Montes algorithm, yielding a ``non-optimised'' tree of types. The types $\type_{\lambda,\psi}$ may not be optimal. In fact, if $\lambda \in \Z$ and $\deg \psi = 1$, we have
\begin{align*}
	\deg \phi_{\lambda,\psi} &= e_{\lambda} \cdot \deg \psi \cdot \deg \phi = \deg \phi,
\end{align*}
where $e_{\lambda}$ is the positive denominator of $\lambda$. We must avoid this situation, because the numerical data attached to the types will not be intrinsic data of the prime factors of $f$.

For this reason, the Montes algorithm includes a ``refinement procedure'' which ensures that it only stores optimal types. However, a price must be paid; the output tree of OM representations is no longer coherent.

The optimised tree of OM representations (the real output of the Montes algorithm) may be derived from the non-optimised tree by an iterative application of the following transformation. Any path,
\begin{align}
	\label{eq:path_non_optimised}
	\begin{aligned}
	\begin{tikzpicture}[scale=0.9]
		\filldraw (-1.5, 0) circle (2.5pt);
		\draw[anchor=south east] (-1.65, -0.15) node {$\type$};
		\filldraw (1, 0) circle (2.5pt);
		\filldraw (3.5, 0) circle (2.5pt);
		\draw[dotted] (-1.5, 0) -- (1, 0) -- (3.5, 0);
		\draw[anchor=south] (-0.25, 0) node {$(\phi_{1}, \lambda_{1}, \psi_{1})$};
		\draw (4.05, 0) node {$\cdots$};
		\filldraw (4.5, 0) circle (2.5pt);
		\filldraw (7, 0) circle (2.5pt);
		\filldraw (9.5, 0) circle (2.5pt);
		\draw[dotted] (4.5, 0) -- (7, 0);
		\draw (7, 0) -- (9.5, 0);
		\draw[anchor=south] (8.25, 0) node {$(\phi_{n}, \lambda_{n}, \psi_{n})$};
		\draw[anchor=south west] (9.65, -0.15) node {$\type'$};
	\end{tikzpicture}
	\end{aligned}
\end{align}
in which all edges except for the final one are bad edges satisfying $\lambda_{i} \in \Z$, $\deg \psi_{i} = 1$ for $1 \leq i < n$, collapses into
\begin{align}
	\label{eq:path_optimised}
	\begin{aligned}
		\begin{tikzpicture}[scale=0.9]
			\filldraw (-2, 0) circle (2.5pt);
			\draw[anchor=south east] (-2.15, -0.15) node {$\type$};
			\filldraw (0.5, 0) circle (2.5pt);
			\draw (-2, 0) -- (0.5, 0);
			\draw[anchor=south] (-0.75, 0) node {$(\phi_{n}, \lambda^{*}, \psi_{n})$};
			\draw[anchor=south west] (0.65, -0.15) node {$\type''$};
			\draw[anchor=south west] (2.0, -0.2) node {with $\lambda^{*} = \lambda_{1} + \cdots + \lambda_{n}.$};
		\end{tikzpicture}
	\end{aligned}
\end{align}
The types $\type'$ and $\type''$ are ``equivalent'', and this means that $\PP(\type) = \PP(\type')$, where $\PP(\type)$ is the set of prime polynomials $g \in \oovx$ which are divisible by $\type$ \cite[Thm. 3.7]{Nart:2014:equiv}.

The existence of the non-optimised tree is useful in many situations. Let us see an example.

\begin{lemma}
	\label{lem:coherent_tree}
	Let $\type, \type' \in \treenop$. If $\type$ is a truncation of $\type'$, then $\PP(\type) \supset \PP(\type')$. If neither of these types is a truncation of the other, then $\PP(\type) \cap \PP(\type') = \emptyset$.
\end{lemma}

\begin{proof}
	The first statement is an immediate consequence of Theorem \ref{thm:type_divides_F}.

	The second statement is obvious if $\type$ and $\type'$ have different root nodes, because for all $F \in \PP(\type)$, the reduction $\overline{F}$ modulo $\m$ is a power of the monic irreducible polynomial $\psi_{0}$ corresponding to the root node of $\type$.
	
	Suppose that $\type, \type'$ have the same root node and let $\type_{0}$ be the greatest common node in the paths joining $\type, \type'$ with their root node. By the first statement we may assume that $\type$ and $\type'$ are branches of $\type_{0}$, in other words, that $\type_{0}$ is the previous node of both $\type$ an $\type'$. By the coherence of $\treenop$ we have
	\begin{align*}
		\type &= (\type_{0}; (\phi, \lambda, \psi)), \qquad
		\type' = (\type_{0}; (\phi, \lambda', \psi')),
	\end{align*}
	where either $\lambda \neq \lambda'$ or $\lambda = \lambda', \psi \neq \psi'$.
	
	Let $r$ be the order of $\type_{0}$ and $v_{r}$ its attached valuation. Now, for any $F \in \PP(\type)$, $F' \in \PP(\type')$, Theorem \ref{thm:type_divides_F} shows that $\Newton_{v_{r},\phi}(F)$ and $\Newton_{v_{r},\phi}(F')$ are one-sided of slopes $-\lambda$ and $-\lambda'$ respectively. Hence, $\lambda \neq \lambda'$ implies $F \neq F'$. On the other hand, if $\lambda = \lambda'$ then $\Respol_{\type_{0},\phi,\lambda}(F) = \psi^{a}$ and $\Respol_{\type_{0},\phi,\lambda}(F') = (\psi')^{a'}$ and this implies $F \neq F'$, because $\psi \neq \psi'$.
\end{proof}

This result may be false for arbitrary incoherent trees. However, Lemma \ref{lem:coherent_tree} is valid for the optimised tree $\tree$ of OM representations of the prime factors of $f$.

\begin{proposition}
	\label{prop:type_empy_intersect}
	Let $\type, \type' \in \tree$ be two nodes such that neither of them is a truncation of the other. Then $\PP(\type) \cap \PP(\type') = \emptyset$. In particular, $\pp_{\type} \cap \pp_{\type'} = \emptyset$.
\end{proposition}

\begin{proof}
	Clearly, the nodes $\type, \type'$ are equivalent to two nodes of the non-optimised tree, neither of them being a truncation of the other. Thus, the statement is an immediate consequence of Lemma \ref{lem:coherent_tree}.
	
	The final statement is a consequence of $\pp_{\type} = \set{ \p \in \pp : F_{\p} \in \PP(\type) }$.
\end{proof}

Consider the chain of refinements that take place between \eqref{eq:path_non_optimised} and \eqref{eq:path_optimised}. During each refinement that provokes branching of a type, the intermediate $\phi$ and $\lambda$ values are stored.

Let $\type_{\p}, \type_{\q} \in \tree$ be two leaves attached to prime ideals $\p$ and $\q$ with index of coincidence $i(\p, \q) = \ell$. Then suppose that at level $\ell$, each type has a list of stored refinements,
\begin{align}
	\label{eq:refinement_list}
	\begin{aligned}
	\Ref_{\ell}(\type_{\p}) &= \left[ (\phi_{(1)}^{\type_{\p}}, \lambda_{(1)}^{\type_{\p}}, \psi_{(1)}^{\type_{\p}}), \dots, (\phi_{(k)}^{\type_{\p}}, \lambda_{(k)}^{\type_{\p}}, \psi_{(k)}^{\type_{\p}}) \right], \\
	\Ref_{\ell}(\type_{\q}) &= \left[ (\phi_{(1)}^{\type_{\q}}, \lambda_{(1)}^{\type_{\q}}, \psi_{(1)}^{\type_{\q}}), \dots, (\phi_{(k')}^{\type_{\q}}, \lambda_{(k')}^{\type_{\q}}, \psi_{(k')}^{\type_{\q}}) \right].
	\end{aligned}
\end{align}

This allows us to extend the index of coincidence to a more precise indicator.

\begin{definition}
	\label{def:eioc_types}
	The \emph{minor index of coincidence} $\hat{\imath}(\p, \q)$ for two leaves $\type_{\p}, \type_{\q} \in \tree$, is the least index $\ell'$, such that for the refinement lists given in \eqref{eq:refinement_list},
	\begin{align*}
		(\phi_{(\ell')}^{\type_{\p}}, \lambda_{(\ell')}^{\type_{\p}}, \psi_{(\ell')}^{\type_{\p}}) &\neq (\phi_{(\ell')}^{\type_{\q}}, \lambda_{(\ell')}^{\type_{\q}}, \psi_{(\ell')}^{\type_{\q}}).
	\end{align*}

	We also define the \emph{extended index of coincidence} of two types as,
	\begin{align*}
		I(\p, \q):= [ i(\p, \q), \hat{\imath}(\p, \q) ].
	\end{align*}
	These extended indices of coincidence are ordered lexicographically.
\end{definition}

\begin{definition}
	\label{def:cphi_hiddenslope}
	Let $\type_{\p}, \type_{\q} \in \tree$ be two leaves with index of coincidence $i(\p, \q) = \ell$ and let the list of refinements of each type at level $\ell$ be as in \eqref{eq:refinement_list}.
	\begin{enumerate}
		\item The \emph{greatest common $\phi$-polynomial} of the prime ideals $\p, \q$ is $\phi(\p, \q) = \phi_{(j)}^{\type_{\p}} = \phi_{(j)}^{\type_{\q}}$, with $j$ maximal.
	
		\item The (non-optimised) \emph{hidden slopes} of the prime ideals $\p, \q$ are $\clamnop{\p}{\q} = \lambda_{(j)}^{\type_{\p}}$ and $\clamnop{\q}{\p} = \lambda_{(j)}^{\type_{\q}}$, for this maximal value of $j$.
		
		\item The (opimised) hidden slopes of the prime ideals $\p, \q$ are $\clam{\p}{\q} = \lambda_{(1)}^{\type_{\p}} + \cdots + \lambda_{(j)}^{\type_{\p}}$ and $\clam{\q}{\p} = \lambda_{(1)}^{\type_{\q}} + \cdots + \lambda_{(j)}^{\type_{\q}}$, for the same maximal value of $j$.
	\end{enumerate}
\end{definition}

\begin{remark}
	\label{rmk:nonop_valuations}
	(1) By \eqref{eq:path_optimised} $\lambda_{\ell,\p} = \sum_{i=1}^{k} \lambda_{(i)}^{\type_{\p}}$. In particular $\lambda_{\ell,\p} \ge \clam{\p}{\q}$, for all $\q \in \pp$ with $i(\p, \q) = \ell$.
	
	(2) Proposition \ref{prop:explicit_valuations} is easily deduced from Theorem \ref{thm:other_poly_value}. Since this theorem is valid for arbitrary types, it is clear that the formulas in Proposition \ref{prop:explicit_valuations} are valid for the $\phi$-polynomials of the non-optimised tree just by replacing the optimised hidden slopes with the non-optimised ones.
\end{remark}

%-------------------------------------------------------------------------------
% Section: Optimal polynomials as products of $\phi$-polynomials
%-------------------------------------------------------------------------------
\subsection{Optimal polynomials as products of $\phi$-polynomials}
\label{sec:optimal_polynomials}

Let $S \subseteq \pp$ be a subset of prime ideals. Let $\tree_{S}$ be the tree of OM representations of the prime ideals $\p \in S$ computed by the Montes algorithm. We keep the content of Section \ref{sec:montes} concerning the data attached to the different types $\type_{\p}$ for $\p \in S$. We recall that the polynomials $\phi_{\p} \in \oox$ are concrete choices of Okutsu approximations to the prime factors $F_{\p}$ of $f$.

The $\phi$-polynomials for all the prime ideals generate a semigroup.

\begin{definition}
	Let $S \subseteq \pp$ be a set of prime ideals. We denote by $\Phi(S) \subset \oox$ the multiplicative semigroup generated by
	\begin{align*}
		\set{ \phi_{i,\p} : \p \in S,\  0 \le i \le r_{\p} } \cup \bigcup_{\p \in S} \App(\p),
	\end{align*}
	where $\App(\p) = \set{ \phi \in \oox \text{ monic of degree } n_{\p} \text{ such that } \wt[\p]{\phi} \ge \wt[\p]{\phi_{\p}} }$.
		
	We use $\Phi(\p)$ to denote $\Phi(\set{\p})$.
\end{definition}

We are interested in showing that we can restrict our search for polynomials of a given degree $d$ with maximal $w_{S,I}$-value to those in the semigroup $\Phi(S)$.

\begin{definition}
	\label{def:degree_adjusted}
	Let $g \in \oox$. The \emph{degree adjusted} $w_{\p}$-valuation of the element $g(\theta) \in \ooL$ is defined as
	\begin{align*}
		\wdt[\p]{g} &:= \frac{\wt[\p]{g}}{\deg g}.
	\end{align*}
\end{definition}

\begin{lemma}
	\label{lem:nop_tree_vals}
	Let $\type$ be a node in the non-optimised tree $\treenop$ and let $g, h \in \oovx$ be two prime polynomials divisible by $\type$. Then, for any prime ideal $\p \in \pp \setminus \pp_{\type}$ we have $\wdt[\p]{g} = \wdt[\p]{h}$.
\end{lemma}

\begin{proof}
	If $\type$ and $\type_{\p}$ have different root nodes, we have $\wt[\p]{g} = 0 = \wt[\p]{h}$, because $\overline{F}_{\p}$ is a power of $\psi_{0,\p}$ and $\overline{g}$, $\overline{h}$ are powers of the root node of $\type$.

	If $\type$ and $\type_{\p}$ have the same root node, let $\type_{0}$ be the greatest common node in the paths of $\treenop$ joining $\type_{\p}$ and $\type$ with the root node. Since $\p \not\in \pp_{\type}$, the node $\type_{0}$ cannot be equal to $\type$. Since $\type_{\p}$ is a leaf of the tree, $\type_{0}$ cannot be equal to $\type_{\p}$ either. The structure of the non-optimised tree is shown in Figure \ref{fig:descendents_of_mm}.
	
	\begin{figure}[htb]
		\centering
		\begin{tikzpicture}[scale=0.7]
			\draw[anchor=east] (0, 0) node {$\psi_{0}$};
			
			\filldraw (0, 0) circle (2.5pt);
			\filldraw (2, 0) circle (2.5pt);
			\filldraw (4, 0) circle (2.5pt);
			\filldraw (6, 0) circle (2.5pt);
			
			\draw (0, 0) -- (2, 0);
			\draw (3, 0) node {$\cdots$};
			\draw (4, 0) -- (6, 0);
			\draw[anchor=south] (6, 0.2) node {$\type_{0}$};
			
			\filldraw (8, 1) circle (2.5pt);
			\filldraw (8, -1) circle (2.5pt);
			
			\filldraw (10, 1) circle (2.5pt);
			\filldraw (12, 1) circle (2.5pt);
			\filldraw (14, 1) circle (2.5pt);
			\filldraw (10, -1) circle (2.5pt);
			\filldraw (12, -1) circle (2.5pt);
			
			\draw (6, 0) -- (8, 1);
			\draw (9, 1) node {$\cdots$};
			\draw (10, 1) -- (14, 1);
			
			\draw (6, 0) -- (8, -1);
			\draw (9, -1) node {$\cdots$};
			\draw (10, -1) -- (12, -1);
	
			\draw[anchor=south west] (7.5, 1.2) node  {$\type'$};
			\draw[anchor=south west] (7.5, -1.9) node {$\type''$};
			\draw[anchor=south west] (13.6, 1.07) node {$\type_{\p}$};
			\draw[anchor=south west] (11.7, -1.9) node {$\type$};
		\end{tikzpicture}
		\caption{The node $\type_{0}$ is the greatest common node of $\type$ and $\type_{\p}$.}
		\label{fig:descendents_of_mm}
	\end{figure}
	
	Let $\type', \type''$ be the nodes following $\type_{0}$ in each of the two paths. Since the non-optimised tree is coherent, we have
	\begin{align*}
		\type' &= (\type_{0}; (\phi, \lambda', \psi')), \qquad
		\type'' = (\type_{0}; (\phi, \lambda'', \psi'')),
	\end{align*}
	with a common choice for the representative $\phi$ of $\type_{0}$.
	
	By Theorem \ref{thm:type_divides_F}, $\Newton_{v_{\type_{0}},\phi}(g)$ and $\Newton_{v_{\type_{0}},\phi}(h)$ are one-sided of slope $-\lambda''$. Proposition \ref{prop:type_empy_intersect} shows that $\type' \ndivides g$, $\type' \ndivides h$. By Theorem \ref{thm:other_poly_value}, we have
	\begin{align*}
		\wdt[\p]{g} &= \frac{1}{\deg \phi} \cdot \frac{V_{r+1} + \Min\set{\lambda', \lambda''}}{\eto{r}} = \wdt[\p]{h},
	\end{align*}
	where $r$ is the order of $\type_{0}$.
\end{proof}

The next result is the main aim of this section.

\begin{proposition}
	\label{thm:optimal_in_phi}
	Let $S \subseteq \pp$ be a set of prime ideals. 	For any $h \in \oovx$ monic of degree $0 \leq d < n$, there exists $\phi \in \Phi(\pp)$ also of degree $d$ such that,
	\begin{align}
		\w[\p]{\phi(\theta)} &\geq \w[\p]{h(\theta)}, \qquad \forall\  \p \in S. \label{eqn:phi_as_good_as_h}
	\end{align}
\end{proposition}

\begin{proof}
	The proof will proceed by induction on the degree $d$ of the polynomial. We will work in steps, in each one reducing the space in which we need to consider $h$.
	
	If $d = 0$, then $\phi = h = 1 \in \Phi(S)$.
	
	\begin{claim}
		It is sufficient to check \eqref{eqn:phi_as_good_as_h} for $h$ a prime polynomial.
	\end{claim}
	
	Let $h = h_{1} h_{2}$, with $h_{1}, h_{2} \in \oovx$ monic of degree $d_{1}, d_{2} > 0$ respectively.

	By the induction hypothesis, there exist $\phi_{i} \in \Phi(S)$ of degree $d_{i}$ such that,
	\begin{align*}
		\wt[\p]{\phi_{i}} &\geq \wt[\p]{h_{i}}, 	\qquad \forall\  \p \in S,
	\end{align*}
	for $i = 1, 2$. Then, $\phi = \phi_{1} \phi_{2} \in \Phi(S)$ clearly satisfies \eqref{eqn:phi_as_good_as_h}. This proves the claim.

	Now, assume that $h$ is a prime polynomial. If $\gcd(\overline{f}, \overline{h}) = 1$, then $\wt[\p]{h} = 0$ for all $\p \in S$. Thus, \eqref{eqn:phi_as_good_as_h} is obviously satisfied.

	Therefore, we can assume that $\overline{h} = \psi_{0}^{b}$, $b \in \N$, for $\psi_{0} \in \Fy$ a prime factor of $\fbar$.

	By hypothesis, the root node $\psi_{0}$ (thought of as a type of order zero) divides $h$. Let $\type$ be the highest order node in the non-optimised tree $\treenop_{S}$ such that $\type \divides h$, and let $i$ be the order of $\type$. We distinguish two cases according to $\type$ being a leaf or not.
	
	\begin{proofcase}
		$\type$ is a leaf.
		
		In this case, $S_{\type} = \set{\p_{0}}$ contains only one prime ideal. The $\phi$-polynomial in the last level of $\type$ is an Okutsu approximation $\phi_{\p_{0}}$ to $F_{\p_{0}}$. Since $\type \divides h$, Theorem \ref{thm:type_F_props} shows that $\deg h = \ell \cdot \deg \phi_{\p_{0}} = \ell \cdot n_{\p_{0}}$ for some positive integer $\ell$, and $\vt[\p_{0}]{h} > \vt[\p_{0}]{\phi_{\p_{0}}}$.
				
		Let us consider $\phi_{0} \in \App(\pz)$ and close enough to $F_{\p_{0}}$ so that
		\begin{align*}
			\wt[\p_{0}]{\phi_{0}} &\geq \wt[\p_{0}]{h} / \ell,
		\end{align*}
		and take $\phi = \phi_{0}^{\ell} \in \Phi(S)$. By construction, $\wt[\p_{0}]{\phi} \ge \wt[\p_{0}]{h}$. On the other hand, for any $\p \in S$, $\p \neq \p_{0}$, we clearly have $\wdt[\p]{\phi} = \wdt[\p]{\phi_{0}}$ and Lemma \ref{lem:nop_tree_vals} shows that $\wdt[\p]{\phi_{0}} = \wdt[\p]{h}$. Since $\deg(\phi) = \deg(h)$ we deduce that $\wt[\p]{\phi} = \wt[\p]{h}$. This proves \eqref{eqn:phi_as_good_as_h}.
	\end{proofcase}
	
	\begin{proofcase}
		$\type$ is not a leaf.
		
		For a certain choice $\phi_{\type}$ of a representative of $\type$, the node $\type$ has several branches in the non-optimised tree, of the form
		\begin{align*}
			\type_{\lambda,\psi} = (\type; (\phi_{\type}, \lambda, \psi)).
		\end{align*}
		
		By the maximality of $\type$, we have $\type_{\lambda,\psi} \ndivides h$ for all these branch nodes. Let $\lmax$ be the greatest slope (in absolute size) of these branches and let $\type_{\max}$ be any branch node of $\type$ with slope $\lmax$.

		Since $\type \divides h$, Theorem \ref{thm:type_F_props} shows that $\Newton_{v_{\type},\phi_{\type}}(h)$ is one-sided of slope $-\lambda_{h}$ and $\deg(h) = \ell \deg(\phi_{\type})$, for certain positive $\lambda_{h} \in \Q$, $\ell \in \Z$.

		If for some $\p_{0} \in S_{\type}$ we take $\phi_{j,\p_{0}}$ satisfying:
		\begin{align}
			\label{eq:phij_reqs}
			\deg(\phi_{j,\p_{0}}) = \deg(\phi_{\type}), \qquad \type \divides \phi_{j,\p_{0}},
		\end{align}
		then Lemma \ref{lem:nop_tree_vals} shows that $\wdt[\p]{h} = \wdt[\p]{\phi_{j,\p_{0}}}$ for all $\p \not\in S_{\type}$. As in Case 1, for $\phi = \phi_{j,\p_{0}}^{\ell} \in \Phi(S)$ this implies $\wt[\p]{h} = \wt[\p]{\phi}$ for all $\p \not\in S_{\type}$. Therefore, we need only to find some $\phi_{j,\p_{0}}$ satisfying \eqref{eq:phij_reqs} and $\wt[\p]{\phi_{j,\p_{0}}} \ge \ell^{-1} \wt[\p]{h}$ for all $\p \in S_{\type}$. Then we shall have \eqref{eqn:phi_as_good_as_h}.

%		Suppose that $\type_{\max}$ gives rise to a node of the optimised tree $\tree$. In this case $\phi_{\type} = \phi_{j,\p_{0}}$ for all $\p_{0} \in \pp_{\type_{\max}}$, where $j$ is the order of $\type_{\max}$ as a type of the optimised tree. Clearly, $\phi_{j,\p_{0}} = \phi_{\type}$ satisfies \eqref{eq:phij_reqs}. Let us compare $\wt[\p]{h}$ and $\wt[\p]{\phi_{\type}}$ for $\p \in \pp_{\type}$.
%		
%		Take any $\p \in \pp_{\type}$. The polynomial $F_{\p}$ will be divided by exactly one of the branches $\type_{\lambda,\psi}$. Since $\type_{\lambda,\psi} \ndivides h$, Theorems \ref{thm:type_F_props} and \ref{thm:other_poly_value} show that:
%		%
%		\begin{align*}
%			\ell^{-1} \wt[\p]{h} = \frac{V_{i+1} + \Min\set{\lambda_{h}, \lambda}}{\eto{i}} \le \frac{V_{i+1} + \lambda}{\eto{i}} = \wt[\p]{\phi_{\type}}.
%		\end{align*}
%		%
%		Hence, $\phi = \phi_{\type}^{\ell} \in \Phi(\pp)$ satisfies \eqref{eqn:phi_as_good_as_h}, as explained above.
%		
%		Suppose now that $\type_{\max}$ does not correspond to a node in the optimised tree.
		
		Let $\type'$ be any node of the optimised tree which has been derived from $\type_{\max}$ by a series of refinement steps as indicated in \eqref{eq:path_non_optimised} and \eqref{eq:path_optimised}.
					
		\begin{figure}[htb]
			\centering
			\begin{tikzpicture}[scale=0.7]
				\draw[anchor=east] (-0.2, 0) node {$\type$};
				
				\filldraw (0, 0) circle (2.5pt);
				\filldraw (3, 1.5) circle (2.5pt);
				\filldraw (3, 0.5) circle (2.5pt);
		%		\filldraw (2, -0.5) circle (2.5pt);
				\filldraw (3, -1.5) circle (2.5pt);
				
				\draw[anchor=west] (3.1, -1.5) node {$\type_{\lambda,\psi}$};
				
				\filldraw (6, 1.5) circle (2.5pt);
				\filldraw (8, 1.5) circle (2.5pt);
				\filldraw (11, 1.5) circle (2.5pt);
		
				\draw (0, 0) -- (3, 1.5) -- (6, 1.5);
				\draw (0, 0) -- (3, 0.5);
		%		\draw (0, 0) -- (2, -0.5);
				\draw (3.02, -0.25) node {$\vdots$};
				\draw (0, 0) -- (3, -1.5);

				\draw (7, 1.5) node {$\cdots$};
				\draw (8, 1.5) -- (11, 1.5);
				
				\draw[anchor=south] (9.5, 1.5) node {$(\phi_{0}, \lambda_{0}, \psi_{0})$};
				
				\draw[anchor=south west] (2.2, 1.6) node {$\type_{\max}$};
				\draw[anchor=west] (11.1, 1.5) node {$\type'$};		
			\end{tikzpicture}
			\caption{The node $\type'$ corresponds to a node of the optimised tree.}
			\label{fig:nmax_nop}
		\end{figure}
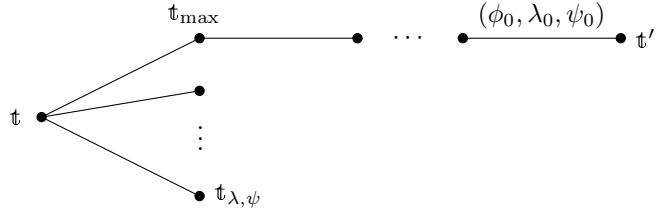

		Let $(\phi_{0}, \lambda_{0}, \psi_{0})$ be the last level of $\type'$ in the non-optimised tree. As explained in Section \ref{sec:nonoptimised_tree} the last level of $\type'$ as a type from the optimised tree will be $(\phi_{0}, \lambda_{0}^{*}, \psi_{0})$, where $\lambda_{0}^{*}$ is the sum of all the slopes of all bad levels between $\type'$ and its previous node in the optimised tree.
		
		Thus, $\phi_{0} = \phi_{j,\p_{0}}$ for all $\p_{0} \in S_{\type'}$, where $j$ is the order of $\type'$ as a type of the optimised tree. Clearly, $\phi_{j,\p_{0}}$ satisfies \eqref{eq:phij_reqs}; let us compare $\wt[\p]{\phi_{j,\p_{0}}}$ and $\wt[\p]{h}$ for $\p \in S_{\type}$. Take $\p \in S_{\type}$ and let $\type_{\lambda,\psi}$ be the unique branch of $\type$ such that $\type_{\lambda,\psi} \divides F_{\p}$. By Theorem \ref{thm:other_poly_value},
		\begin{align*}
%			\label{eq:case2_hval}
			\ell^{-1} \wt[\p]{h} = \frac{V_{i+1} + \Min\set{\lambda_{h}, \lambda}}{\eto{i}} \le \frac{V_{i+1} + \lambda}{\eto{i}}.
		\end{align*}
		Thus, we need only to show that
		\begin{align*}
			\wt[\p]{\phi_{0}} \ge \frac{V_{i+1} + \lambda}{\eto{i}}.
		\end{align*}

		Suppose that $\type_{\max}$ gives rise to a node of the optimised tree. In this case, we have $\type_{\max} = \type'$ and $\phi_{0} = \phi_{\type}$. By Theorem \ref{thm:type_F_props}, $\wt[\p]{\phi_{0}} = (V_{i+1} + \lambda)/(\eto{i})$. From now on we assume that $\type_{\max} \ne \type'$.

		If $\type_{\lambda,\psi} \ne \type_{\max}$, then $\type \divides \phi_{0}$, $\type_{\lambda,\psi} \ndivides \phi_{0}$ by Proposition \ref{prop:type_empy_intersect}; hence Theorem \ref{thm:other_poly_value} shows that
		\begin{align*}
			\wt[\p]{\phi_{0}} = \frac{V_{i+1} + \Min\set{\lambda, \lmax}}{\eto{i}} = \frac{V_{i+1} + \lambda}{\eto{i}}.
		\end{align*}
		Finally, suppose that $\type_{\lambda,\psi} = \type_{\max}$. The order of $\type'$ as a type of the non-optimised tree is $\ge i+2$. By Proposition \ref{prop:explicit_valuations} applied to the non-optimised tree (see Remark \ref{rmk:nonop_valuations}) we have
		\begin{align*}
			\wt[\p]{\phi_{0}} > \frac{V_{i+2,\p}}{\eto{i} e_{i+1,\p}}
				= \frac{e_{i+1,\p} f_{i+1,\p}(e_{i+1,\p} V_{i+1} + h_{i+1,\p})}{\eto{i} e_{i+1,\p}}
				= \frac{V_{i+1} + \lambda}{\eto{i}},
		\end{align*}
		because  $e_{i+1,\p} = f_{i+1,\p} = 1$ and $\lambda = h_{i+1,\p}$. Thus, in all cases we obtain the desired inequality.
		\qedhere
	\end{proofcase}
\end{proof}

%-------------------------------------------------------------------------------
% Section: Okutsu bases
%-------------------------------------------------------------------------------
\subsection{Optimal polynomials as products of numerators of Okutsu bases}

By Proposition \ref{prop:explicit_valuations}, for any $\p \in S$ we have
\begin{align}
	\label{eq:strict_val_inc}
	\wdt[\p]{\phi_{i,\p}} &=  \frac{1}{m_{i}} \frac{V_{i} + \lambda_{i}}{\eto{i-1}} = \frac{1}{m_{i+1}} \frac{V_{i+1}}{\eto{i}} < \wdt[\p]{\phi_{i+1,\p}}, &1 \le i \le r_{\p}.
\end{align}

Let us analyse how closely we can replicate this inequality \eqref{eq:strict_val_inc} for cross valuations, that is to say when the $\phi$-polynomial belongs to a different prime to that of the valuation. The next results follow closely form Proposition \ref{prop:explicit_valuations} too.

\begin{lemma}
	\label{lem:mostly_increasing_cross_values}
	Let $\p, \q \in \pp$ be two different prime ideals with index of coincidence $\ell = i(\p, \q)$. Then:
	\begin{enumerate}
		\item $\wdt[\p]{\phi_{i,\q}} < \wdt[\p]{\phi_{i+1,\q}}$, \quad $1 \leq i < \ell$,
		\item $\wdt[\p]{\phi_{i,\q}} = \wdt[\p]{\phi_{i+1,\q}}$, \quad $\ell < i \leq r_{\p}$.
	\end{enumerate}
\end{lemma}

%\begin{proof}
%	Item (2) is a consequence of Proposition \ref{prop:explicit_valuations}. Let us prove item (1).
%		
%	For $i < \ell-1$, we have $\phi_{i,\q} = \phi_{i,\p}$, $\phi_{i+1,\q} = \phi_{i+1,\p}$ and the inequality of (1) is a direct consequence of \eqref{eq:strict_val_inc}.
%	
%	Assume $i = \ell - 1$. By Proposition \ref{prop:explicit_valuations}, we have
%	%
%	\begin{align*}
%		\wdt[\p]{\phi_{\ell-1,\q}} &= \frac{1}{m_{\ell-1}} \cdot \frac{V_{\ell-1} + \lambda_{\ell-1}}{\eto{\ell-2}} = \frac{1}{\ml} \frac{V_{\ell}}{\eto{\ell-1}}, \\
%		%
%		\wdt[\p]{\phi_{\ell,\q}} &= \frac{1}{\ml} \cdot
%		\begin{cases}
%			\dfrac{V_{\ell} + \clam{\p}{\q}}{\eto{\ell-1}}, \quad \text{or} \medskip \\
%			\dfrac{V_{\ell} + \minclam{\p}{\q}}{\eto{\ell-1}}.
%		\end{cases}
%	\end{align*}
%	%
%	Hence, $\wdt[\p]{\phi_{\ell-1,\q}} < \wdt[\p]{\phi_{\ell,\q}}$ and this ends the proof of (1).
%\end{proof}

It is easy to find examples where
\begin{align}
	\label{eq:badq}
	\wdt[\p]{\phi_{\ell,\q}} > \wdt[\p]{\phi_{\ell+1,\q}}.
\end{align}

This pathology occurs when $\phi_{\ell,\q} = \phi(\p, \q)$ and $\clam{\p}{\q}$ is much larger than $\clam{\q}{\p}$ (see Proposition \ref{prop:explicit_valuations}). Hence, it is also easy to find specific conditions that avoid \eqref{eq:badq}.

\begin{lemma}
	\label{lem:safeq}
	Let $\p, \q \in \pp$ be two different prime ideals with $\ell = i(\p, \q) > 0$, chosen so that $\clam{\q}{\p} \geq \clam{\p}{\q}$. Then,
	\begin{align}
		\label{eq:safeq}
		\wdt[\p]{\phi_{\ell,\q}} &= \wdt[\p]{\phi_{\ell+1,\q}}.
	\end{align}
	In particular, every numerator $g_{i,\q}$ of degree $i$ of the Okutsu $\q$-basis has maximal $\p$-valuation amongst all polynomials $\phi \in \Phi(\q)$ of degree $i$.
\end{lemma}

\begin{proof}
	By the hypothesis, $\minclam{\p}{\q} = \clam{\p}{\q}$ and Proposition \ref{prop:explicit_valuations} gives \eqref{eq:safeq}. Therefore, Lemma \ref{lem:mostly_increasing_cross_values} and \eqref{eq:safeq} show that the $\p$-valuations of the polynomials $\phi_{i,\q}$ increase with their degree up to index $\ell$ and then remain equal. As such, we will always have a maximal valuation by taking higher degree $\phi$-polynomials, rather than products of smaller degree ones.
\end{proof}

Using the extended index of coincidence presented in Definition \ref{def:eioc_types}, the following Lemma gives us a link between the relative similarity of prime ideals and their respective cross-valuations in certain cases.

\begin{lemma}
	\label{lem:greater_cross_value}
	For a prime ideal $\q \in \pp$, let $\p, \l \in \pp \setminus \set{\q}$ be two prime ideals such that either $\l = \p$ or they satisfy:
	\begin{enumerate}
		\item $I(\l, \q) \ge I(\p, \q)$, and
		\item $\clam{\l}{\q} \ge \clam{\p}{\q}$ if $I(\l, \q) = I(\p, \q)$.
	\end{enumerate}
	Then, for $\ell = i(\l, \q)$ and $\ml = m_{\ell,\q} = m_{\ell,\l}$, we have
	\begin{align*}
		\wt[\p]{\phi_{\ell,\l}^{m_{i,\q}/\ml}} &\ge \wt[\p]{\phi_{i,\q}}, \qquad \ell \le i \le r_{\q} + 1.
	\end{align*}
\end{lemma}

\begin{proof}
	We may consider four cases depending on the relationship between the three prime ideals $\p$, $\q$, and $\l$.

	\begin{proofcase}
		$i(\l,\q) > i(\p,\q)$. In this case, $k := i(\p, \l) = i(\p, \q) < \ell$ and $\clam{\q}{\p} = \clam{\l}{\p}$, $\clam{\p}{\q} = \clam{\p}{\l}$. Therefore, for $\ell \le i \le r_{\q} + 1$, Proposition \ref{prop:explicit_valuations} shows that
		\begin{align}
			\label{eq:closer_crossval_easy_case}
			\begin{aligned}
			\wt[\p]{\phi_{\ell,\l}^{m_{i,\q}/\ml}} &= \frac{m_{i,\q}}{\ml} \frac{\ml}{m_{k}} \frac{V_{k} + \minclam{\p}{\l}}{\eto{k-1}} \\
				&= \frac{m_{i,\q}}{m_{k}} \frac{V_{k} + \minclam{\p}{\q}}{\eto{k-1}} = \wt[\p]{\phi_{i,\q}}.
			\end{aligned}
		\end{align}
	\end{proofcase}

	\begin{proofcase}
		$i(\l,\q) = i(\p,\q)$ and either $\l = \p$ or $i(\p,\l) > \ell$. We have
		\begin{align*}
			\wt[\p]{\phi_{\ell,\l}^{m_{i,\q}/\ml}} = \frac{m_{i,\q}}{\ml} \frac{V_{\ell} + \lambda_{\ell,\p}}{\eto{\ell-1}}
		\end{align*}
		by Proposition \ref{prop:explicit_valuations}. On the other hand,
		\begin{align}
			\label{eq:closer_crossval_q_indco}
			\wt[\p]{\phi_{i,\q}} =
			\begin{cases}
				\dfrac{V_{\ell} + \clam{\p}{\q}}{\eto{\ell-1}},		&\text{if } i = \ell \text{ and } \phi_{\ell,\q} = \phi(\p,\q), \medskip \\
				\dfrac{m_{i,\q}}{\ml} \dfrac{V_{\ell} + \minclam{\p}{\q}}{\eto{\ell-1}},	&\text{otherwise}.
			\end{cases}
		\end{align}
		By the first item of Remark \ref{rmk:nonop_valuations}, we have $\clam{\p}{\q} \le \lambda_{\ell,\p}$, so that $\wt[\p]{\phi_{i,\q}} \le \wt[\p]{\phi_{\ell,\l}^{m_{i,\q}/\ml}}$ in both cases.
	\end{proofcase}
		
	\begin{proofcase}
		$i(\l,\q) = i(\p,\q) = i(\p,\l)$, and $I(\l,\q) > I(\p,\q)$. In the non-optimised tree, we find the situation described in Figure \ref{fig:tree_nop_extind} (a), where we have written the optimised hidden slopes instead of the non-optimised ones.
		
		We necessarily have $\phi_{\ell,\q} \ne \phi(\p,\q) = \phi(\p,\l) \ne \phi_{\ell,\l}$, and $\clam{\p}{\l} = \clam{\p}{\q}$, $\clam{\q}{\p} = \clam{\l}{\p}$. Therefore, the equations of \eqref{eq:closer_crossval_easy_case} (where $k=\ell$ now) are again a consequence of Proposition \ref{prop:explicit_valuations}, for all $\ell \le i \le r_{\q} + 1$.
	\end{proofcase}

	\begin{figure}[htb]
		\centering
		\begin{subfigure}{0.45\textwidth}
			\centering
			\begin{tikzpicture}[scale=0.9]
		
				\draw[anchor=east] (-0.15, 0) node {$\cdots$};
				\filldraw (0, 0) circle (2.5pt);
				
				\filldraw (2, 1) circle (2.5pt);
				\filldraw (4, 1) circle (2.5pt);
				\filldraw (6, 2) circle (2.5pt);
				\filldraw (6, 0) circle (2.5pt);
			
				\filldraw (2, -2) circle (2.5pt);
			
				\draw (0, 0) -- (2, 1) -- (4, 1);
				\draw (4, 1) -- (6, 2);
				\draw (4, 1) -- (6, 0);
				\draw (0, 0) -- (2, -2);
				
				\draw[anchor=east] (7.65, 1.990) node {$\cdots\ \type_{\l}$};
				\draw[anchor=east] (7.5, -0.030) node {$\cdots\ \type_{\q}$};
				\draw[anchor=east] (7.5, -2.030) node {$\cdots \cdots \cdots \cdots \cdots \cdots \cdots \cdots\ \type_{\p}$};
		
				\draw[anchor=south east] (1.3, 0.50) node {$\lambda_{\l}^{\p} = \lambda_{\q}^{\p}$};
				\draw[anchor=north east] (1.0, -0.75) node {$\lambda_{\p}^{\l} = \lambda_{\p}^{\q}$};
		
				\draw[anchor=south east] (4.9, 1.40) node {$\lambda_{\l}^{\q}$};
				\draw[anchor=north east] (4.9, 0.60) node {$\lambda_{\q}^{\l}$};
		
			\end{tikzpicture}
			\caption{(a) $I(\l, \q) > I(\l, \p)$.}
			\label{fig:tree_nop_extind_gt}
		\end{subfigure}
		~~~
		\begin{subfigure}{0.45\textwidth}
			\centering
			\begin{tikzpicture}[scale=0.9]
		
				\draw[anchor=east] (-0.15, 0) node {$\cdots$};
				\filldraw (0, 0) circle (2.5pt);
				
				\filldraw (2, 2) circle (2.5pt);
				\filldraw (2, 0) circle (2.5pt);
				\filldraw (2, -2) circle (2.5pt);
			
				\draw (0, 0) -- (2, 2);
				\draw (0, 0) -- (2, 0);
				\draw (0, 0) -- (2, -2);
				
				\draw[anchor=west] (2.2, 1.990) node {$\cdots\ \type_{\l}$};
				\draw[anchor=west] (2.2, -0.030) node {$\cdots\ \type_{\q}$};
				\draw[anchor=west] (2.2, -2.030) node {$\cdots\ \type_{\p}$};
		
				\draw[anchor=south east] (0.9, 0.75) node {$\lambda_{\l}^{\q} = \lambda_{\l}^{\p}$};
				\draw[anchor=south east] (2.3, 0.0) node {$\lambda_{\q}^{\l} = \lambda_{\q}^{\p}$};
				\draw[anchor=north east] (0.9, -0.75) node {$\lambda_{\p}^{\l} = \lambda_{\p}^{\q}$};		
			\end{tikzpicture}
			\caption{(b) $I(\l, \q) = I(\l, \p)$.}
			\label{fig:tree_nop_extind_eq}
		\end{subfigure}
		\caption{Relative positions of $\type_{\l}$, $\type_{\q}$, and $\type_{\p}$ in the non-optimised tree when $i(\l,\q) = i(\p,\q) = i(\p,\l)$.}
		\label{fig:tree_nop_extind}
	\end{figure}
	
	\begin{proofcase}
		$\p \ne \l$, $i(\l,\q) = i(\p,\q) = i(\p, \l)$, and $I(\l,\q) = I(\p,\q)$. In the non-optimised tree, we find the situation described in Figure \ref{fig:tree_nop_extind} (b). We have $\phi(\p,\q) = \phi(\l,\q) = \phi(\p,\l)$. By our assumptions, $\clam{\p}{\l} \le \clam{\l}{\p}$ and Proposition \ref{prop:explicit_valuations} shows that:
		\begin{align*}
			\wt[\p]{\phi_{\ell,\l}^{m_{i,\q}/\ml}} = \frac{m_{i,\q}}{\ml} \frac{V_{\ell} + \clam{\p}{\l}}{\eto{\ell-1}},
		\end{align*}
		whereas $\wt[\p]{\phi_{i,\q}}$ is given by \eqref{eq:closer_crossval_q_indco}. Since $\clam{\p}{\q} = \clam{\p}{\l}$, we get $\wt[\p]{\phi_{i,\q}} \le \wt[\p]{\phi_{i,\q}^{m_{i,\q}/\ml}}$ as desired.
		\qedhere
	\end{proofcase}
\end{proof}

\begin{definition}
	\label{def:disorder}
	Let $g = \prod_{\p} \varphi_{\p} \in \Phi(S)$. The \emph{disorder} of $g$ is calculated as $\Disorder(g) = \sum_{\p \in S} \max\set{\deg(\varphi_{\p}) - n_{\p}, 0}$.
\end{definition}

\begin{definition}
	\label{def:canonical}
	A polynomial $\varphi_{\p} = \phi_{0,\p}^{a_{0}} \phi_{1,\p}^{a_{1}} \cdots \phi_{r_{\p},\p}^{a_{r_{\p}}} \phi_{\p}^{a_{r_{\p}+1}} \in \Phi(\p)$ is said to be \emph{canonical} if $0 \le a_{i} < e_{i,\p} f_{i,\p}$ for all $0 \le i \le r_{\p}$. 

	The canonical polynomials $\varphi_{\p} \in \Phi(\p)$ of degree $\deg(\varphi_{\p}) \le n_{\p}$ coincide with the numerators of the Okutsu $\p$-basis, and hence belong to $\Ok(\p)$ too.
\end{definition}

\begin{remark}
	For a given set of prime ideals $S \subseteq \pp$, all of them with the same root node, it is always possible to choose a prime ideal $\p_{0} \in S$ such that $\clam{\p_{0}}{\p} \le \clam{\p}{\p_{0}}$ for all $\p \in S$.
	
	To do so, begin at the root of the tree of types representing $S$ and move up through the levels. When branching is encountered, take the branch that corresponds to the slope of least absolute value at that level. We continue in this way until we reach a leaf node, which will correspond to a prime $\p_{0}$ with the desired properties.
\end{remark}

\begin{proposition}
	\label{prop:okutsu_basis_comb_optimal}
	Let $S \subseteq \pp$ be a set of prime ideals and consider $\phi \in \Phi(S)$ monic of degree $0 \leq d \leq n_{S}$. For appropriate choices of the Okutsu approximations $\phi_{\p}$, the set $\Ok(S)$ contains a polynomial $g$ of degree $d$ such that,
	\begin{align*}
		\wt[\p]{g} &\geq \wt[\p]{\phi},			& \forall\ \p \in S.
	\end{align*}
\end{proposition}

\begin{proof}
	Consider $\phi = \prod_{\p \in S} \varphi_{\p}$ the separation of the polynomial $\phi$ into its $\p$-parts $\varphi_{\p}$ for each $\p \in S$. Then, let $S_{0} = \set{ \p \in S : \varphi_{\p} \ne \phi_{\p} }$.
	
	We will follow an iterative sequence of three steps to find a polynomial $g \in \Ok(S)$ that meets the requirements of the proposition. Throughout this process we will be modifying $g$, which is initially set to $\phi$, via its individual $\p$-parts:
	\begin{enumerate}
		\item For all $\p \in S_{0}$, make $\varphi_{\p}$ canonical.
		\item If $D(g) = 0$, then for all $\p \in S \setminus S_{0}$ we take $\varphi_{\p} = \phi_{\p}$ to be an Okutsu approximation to $F_{\p}$ with $\wt[\p]{\phi_{\p}} \ge \wt[\p]{\phi}$, and finish the iterative process.
		\item Fix some $\q \in S_{0}$ with $\deg(\varphi_{\q}) > n_{\q}$ and then select $\l \in S_{0}$, the ``closest'' prime ideal to $\q$. Transfer all $\phi$-polynomials in $\varphi_{\q}$, except for a single $\phi_{\q}$, to $\varphi_{\l}$. Remove $\q$ from $S_{0}$.
	\end{enumerate}
	
	Below, we will show that at each step, the $\p$-valuation of $g$ does not decrease for all $\p \in S_{0}$ and that the process will terminate in a polynomial $g$ belonging to $\Ok(S)$ after a finite number of iterations. The inequality $\wt[\p]{g} \ge \wt[\p]{\phi}$ for the primes $\p \in S \setminus S_{0}$ are a consequence of the choices in Step (2).
	
	\begin{proofstep}
		We will make $\varphi_{\p}$ canonical for each $\p \in S_{0}$ in turn. Initially, we set $S' = S_{0}$. Take $\q \in S'$ such that $\clam{\q}{\p} \le \clam{\p}{\q}$ for all $\p \in S' \setminus \set{\q}$ and set $S' = S' \setminus \set{\q}$.
				
		Consider $\varphi_{\q} = \prod_{i = 0}^{r_{\q}+1} \phi_{i,\q}^{a_{i}}$. To make $\varphi_{\q}$ canonical, we wish to have $a_{i} < e_{i,\q} f_{i,\q}$ for all $0 \le i \le r_{\q}$. We will do this iteratively for $i = 0, 1, \dots, r_{\q}$.
		
		\begin{proofcase}
			$a_{i} < e_{i,\q} f_{i,\q}$. In this case, we do nothing.
		\end{proofcase}
	
		\begin{proofcase}
			$a_{i} \ge e_{i,\q} f_{i,\q}$ and $\wdt[\p]{\phi_{i,\q}} \le \wdt[\p]{\phi_{i+1,\q}}$ for all $\p \in S'$. We replace each $\phi_{i,\q}^{e_{i,\q}f_{i,\q}}$ with a single $\phi_{i+1,\q}$ in $\varphi_{\q}$.
			
			For all $\p \in S_{0}$, the $\p$-valuation of $\varphi_{\q}$ will not decrease. In fact, for $\p \in S'$ this is deduced from the condition of this case. For $\p \in S_{0} \setminus S'$, this is a consequence of Lemmas \ref{lem:mostly_increasing_cross_values} and \ref{lem:safeq}.
		\end{proofcase}
		
		\begin{proofcase}
			$a_{i} \ge e_{i,\q} f_{i,\q}$ and $\wdt[\pz]{\phi_{i,\q}} > \wdt[\pz]{\phi_{i+1,\q}}$ for some $\pz \in S'$. In this case, we cannot simply exchange $\phi_{i,\q}^{e_{i,\q}f_{i,\q}}$ for $\phi_{i+1,\q}$ without lowering the $\pz$-valuation of $\varphi_{\q}$.
			
			Instead, take $\l \in S'$ the prime ideal that is ``closest'' to $\q$; that is, $I(\q, \l) \ge I(\q, \p)$ for all $\p \in S'$, and in the case of equality $\clam{\l}{\q} \ge \clam{\p}{\q}$. By Lemma \ref{lem:mostly_increasing_cross_values} and the election of $\l$ we must have $i(\q, \l) \ge i = i(\pz,\q)$.
			
			We remove $\phi_{i,\q}^{a_{i} - e_{i,\q}f_{i,\q} + 1}$ from $\varphi_{\q}$ and insert $\phi_{i,\l}^{a_{i} - e_{i,\q}f_{i,\q} + 1}$ into $\varphi_{\l}$. If $i < \ell := i(\l,\q)$ then $\phi_{i,\q} = \phi_{i,\l}$ and $g$ has not changed (we have only redistributed its $\p$-parts). If $i = \ell$, we must check that $\wt[\p]{\phi_{\ell,\l}} \ge \wt[\p]{\phi_{\ell,\q}}$ for all $\p \in S_{0}$. For $\p \in S'$ this is a consequence of Lemma \ref{lem:greater_cross_value}.
					
			By the maximality of $I(\l,\q)$ amongst all prime ideals in $S'$, the relative situation of $\type_{\q}, \type_{\l}, \type_{\pz}$ in the non-optimised tree is as indicated in Figure \ref{fig:case3_indco}.
			
			\begin{figure}[htb]
				\centering
				\begin{subfigure}{0.47\textwidth}
					\centering
					\begin{tikzpicture}[scale=0.9]
						\draw[anchor=east] (-0.15, 0) node {$\cdots\  \type$};
						\filldraw (0, 0) circle (2.5pt);
						\filldraw (2, 0.8) circle (2.5pt);
						\filldraw (2, 0) circle (2.5pt);
						\filldraw (2, -0.8) circle (2.5pt);
						\draw (0, 0) -- (2, 0.8);
						\draw (0, 0) -- (2, 0);
						\draw (0, 0) -- (2, -0.8);
						\draw[anchor=west] (2.7, 0.770) node {$\cdots\ \type_{\p_{0}}$};
						\draw[anchor=west] (2.7, -0.010) node {$\cdots\ \type_{\l}$};
						\draw[anchor=west] (2.2, -0.830) node {$\type' \ \cdots\ \type_{\q}$};
					\end{tikzpicture}
					\label{fig:case3_indco_eq}
				\end{subfigure}
				~~~
				\begin{subfigure}{0.47\textwidth}
					\centering
					\begin{tikzpicture}[scale=0.9]
						\draw[anchor=east] (-0.15, 0) node {$\cdots\  \type$};
						\filldraw (0, 0) circle (2.5pt);
						\filldraw (2, 0.6) circle (2.5pt);
						\filldraw (2, -0.6) circle (2.5pt);
						\draw (0, 0) -- (2, 0.6);
						\draw (0, 0) -- (2, -0.6);
						\draw[anchor=west] (2.7, 0.570) node {$\cdots\ \type_{\p_{0}},\, \type_{\l}$};
						\draw[anchor=west] (2.2, -0.630) node {$\type' \ \cdots\ \type_{\q}$};
					\end{tikzpicture}
					\label{fig:case3_indco_gt}
				\end{subfigure}
				\caption{Relative position of $\type_{\q}, \type_{\l}, \type_{\pz}$ in the non-optimised tree.}
				\label{fig:case3_indco}
			\end{figure}
			
			Also, by the remarks following \eqref{eq:badq}, we have necessarily $\phi_{\ell,\q} = \phi(\pz,\q)$; hence the node $\type'$ in Figure \ref{fig:case3_indco} corresponds to a type in the optimised tree and $\lambda_{\ell,\q} = \clam{\q}{\pz} = \clam{\q}{\l}$. In particular, for any $\p \in S$ with $i(\p,\q) = \ell$ we cannot have $I(\p,\q) > I(\l,\q)$.
			
			Now consider $\p \in S_{0} \setminus S'$. If $\p \ne \q$ and $I(\p,\q) \le I(\l,\q)$, Lemma \ref{lem:greater_cross_value} is also applicable and yields $\wt[\p]{\phi_{\ell,\l}} \ge \wt[\p]{\phi_{\ell,\q}}$. If $I(\p,\q) > I(\l,\q)$, we have necessarily $i(\p,\q) > \ell$, as we have just remarked above. This clearly implies $i(\p,\l) = \ell$.
			
			In both cases, $\p = \q$ or $i(\p,\q) > i(\l,\q) = \ell$, Proposition \ref{prop:explicit_valuations} shows that:
			\begin{align*}
				\wt[\p]{\phi_{\ell,\q}} &= \frac{V_{\ell} + \lambda_{\ell,\p}}{\eto{\ell-1}} = \frac{V_{\ell} + \lambda_{\ell,\q}}{\eto{\ell-1}}, \\
				\wt[\p]{\phi_{\ell,\l}} &= \frac{V_{\ell} + \clam{\p}{\l}}{\eto{\ell-1}} = \frac{V_{\ell} + \clam{\q}{\l}}{\eto{\ell-1}}.
			\end{align*}
			In the last equality, we used $\clam{\q}{\l} = \minclam{\q}{\l}$, by the choice of $\q$. Since $\lambda_{\ell,\q} = \clam{\q}{\l}$, we get $\wt[\p]{\phi_{\ell,\q}} = \wt[\p]{\phi_{\ell,\l}}$ as desired.
		\end{proofcase}	
		
		We continue in this way until $\#S' = 0$.
	\end{proofstep}
	
	\begin{proofstep}
		Having completed Step 1, all $\p$-parts of $g$ are canonical, so if $\Disorder(g) = 0$, then $g \in \Ok(S)$, completing the process after considering adequate choices of the Okutsu approximations $\varphi_{\p} = \phi_{\p}$ for all $\p \in S \setminus S_{0}$.
	\end{proofstep}
	
	\begin{proofstep}
		If there exists $\q \in S_{0}$ with $\deg \varphi_{\q} > n_{\q}$, we choose $\l \in S_{0}$ such that $I(\q, \l) \ge I(\q, \p)$ for all $\p \in S_{0} \setminus \set{\q}$ and $\clam{\l}{\q} \ge \clam{\p}{\q}$ in the case of equality.
	
		Having chosen $\q$ and $\l$, the transfer occurs as follows for each $0 \le i \le r_{\q} + 1$:
		\begin{align*}
			\phi_{i,\q} \longrightarrow
			\begin{cases}
				\phi_{i,\l}	& \text{if } i < \ell = i(\q, \l), \\
				\phi_{\ell,\l}^{m_{i,\q}/m_{\ell}}	& \text{if } i \ge \ell = i(\q, \l).
			\end{cases}
		\end{align*}
		
		By Lemma \ref{lem:greater_cross_value}, the $\p$-valuation of the resulting $\varphi$ will not decrease, except possibly in the case where $\p = \q$. However, since $\varphi_{\q} = \phi_{\q}$, we may increase the $\q$-valuation of $\varphi_{\q}$ by choosing a better Okutsu approximation $\phi_{\q}$ to compensate any decrease in value.
		
		After this step we remove $\q$ from $S_{0}$, since $\varphi_{\q} = \phi_{\q}$ and then return to Step 1.
	\end{proofstep}
		
	As Step 3 reduces the number of prime ideals in $S_{0}$, this process will clearly end after at most $\#S$ iterations.
\end{proof}

\section{Proof of Theorem \ref{thm:maxmin_works}}
\label{sec:proof_maxmin}

In the interest of clarity, we will only prove Theorem \ref{thm:maxmin_works} in the case $I = \ooL$. The proof in the case of an arbitrary fractional ideal is almost identical. Further details on the required adaptations can be found in \cite[Ch. 6]{Stainsby:2014:thesis}.

%-------------------------------------------------------------------------------
% Section: Precomputation
%-------------------------------------------------------------------------------
\subsection{Precomputation}
\label{sec:precomputation}

Recall that our set $S$ of prime ideals has a total ordering $S = \set{\q_{1},\, \dots,\, \q_{s}}$ satisfying \eqref{eq:ordering}. Denote $n_{S} := \sum_{\p \in S} n_{\p}$, $w_{S} := w_{S,\ooL}$ and consider the following intervals of $S$:
\begin{align*}
	[a, b] &:= \set{ \q_{j} : a \le j \le b },	\qquad 1 \le a \le b \le s.
\end{align*}

\begin{definition}
	An \emph{order preserving} partition of $S$ is a decomposition $S = I_{1} \cup \cdots \cup I_{t}$ of $S$ into the disjoint union of intervals $I_{j} = [a_{j}, b_{j}]$ with increasing end points $b_{1} < \cdots < b_{t}$.
\end{definition}

%Recall that for a set $S \subseteq \pp$, $n_{S} = \sum_{\p \in S} n_{\p}$ and denote
%%
%\begin{align*}
%	n_{j} &:= n_{I_{j}},	\qquad 1 \leq j \leq t. 
%\end{align*}

Take extended families of numerators $g_{0,I_{j}}, \dots, g_{n_{j},I_{j}}$ of Okutsu $I_{j}$-bases, for all $1 \leq j\leq t$, where $n_{j} := n_{I_{j}}$. That is, each $g_{i,I_{j}}$ has degree $i$, belongs to $\Ok(I_{j})$ and $\w[I_{j}]{g_{i,I_{j}}}$ is maximal amongst all monic polynomials in $\oovx$ of degree $i$.

Consider multi-indices $\ii = (i_{1}, \dots, i_{t})$ of degree $\deg \ii = i_{1} +\cdots + i_{t}$ and monic polynomials $g_{\ii} := g_{i_{1},I_{1}} \cdots g_{i_{t},I_{t}} \in \oox$.

We may consider the version of $\mm$ presented in Algorithm \ref{alg:maxmin_s_precomputed}.

\begin{algorithm}
	\caption{$\mm[S = I_{1} \cup \cdots \cup I_{t}]$ algorithm}
	\label{alg:maxmin_s_precomputed}
	\begin{algorithmic}[1]
		\INPUT An order preserving partition $S = I_{1} \cup \cdots \cup I_{t}$ of $S \subseteq \pp$, and extended families $\set{g_{i,I_{j}} :  0 \leq i \leq n_{I_{j}}}$ of numerators of Okutsu $I_{j}$-bases for all $1 \leq j \leq t$.
						
		\OUTPUT A family $\ii_{0}, \ii_{1}, \dots, \ii_{n_{S}} \in \N^{t}$ of multi-indices of degree $0, 1, \dots, n_{S}$, respectively.
				
		\STATE $\ii_{0} \gets \ind{ 0, \dots, 0 } \in \N^{t}$
		
		\FOR {$k = 0 \to n_{S}-1$}
		
			\STATE $j \gets \Min\set{ 1 \leq i \leq t : \w[I_{i}]{g_{\ii_{k}}} = \w[S]{g_{\ii_{k}}} }$
			
			\STATE $\ii_{k+1} \gets \ii_{k} + \uu_{j}$
			
		\ENDFOR
	\end{algorithmic}
\end{algorithm}

Such a decomposition of $\mm$ will be useful for the proof of Theorem \ref{thm:maxmin_works}.

\begin{definition}
	\label{def:precomputation}
	For indices $1\le a\le b\le s$, we say that $I = [a, b]$ \emph{admits precomputation} if, after natural identifications, the algorithm $\mm[S]$ has the same output as
	\begin{align}
		\label{eq:maxmin_precomp_i}
		\mm\left[S=\set{\q_{1}} \cup \cdots \cup \set{\q_{a-1}} \cup I \cup \set{\q_{b+1}} \cup \cdots \cup \set{\q_{s}} \right],
	\end{align}
	where we consider the output of $\mm[I]$ as an extended Okutsu $I$-basis.

	By ``natural identifications'' we mean that if the $k$-th output of $\mm[S]$ is $\ii_{k} = \ind{ i_{\q_{1}}, \dots, i_{\q_{s}} }$, then the $k$-th output of the algoritm \eqref{eq:maxmin_precomp_i} is:
	\begin{align*}
		\jj_{k} = \ind{ i_{\q_{1}}, \dots, i_{\q_{a}-1}, i_I, i_{\q_{b}+1}, \dots, i_{\q_{s}} },
	\end{align*}
	while the $i_{I}$-th output of $\mm[I]$ is $\left(i_{\q_{a}}, \dots, i_{\q_{b}}\right)$.
\end{definition}

%Let $\ind{\ii_{k}}_{0 \leq k \leq n_{S}}$ be the output of $\mm[S]$, leading to numerators
%%
%\begin{align*}
%	\ii_{k} = \ind{i_{\q_{1}}, \dots, i_{\q_{s}}} \implies g_{\ii_{k}} = \prod_{\q \in S} g_{i_{\q},\q}.
%\end{align*}
%
%Let $g'_{0}, g'_{1}, \dots, g'_{n_{I}}$ be the numerators deduced from the application of the algorithm $\mm[I]$. Then, let $(\jj_{k})_{0 \leq k \leq n_{S}}$ be the output of the $\mm$ algorithm (\ref{eq:maxmin_precomp_i}). If $I$ admits precomputation, these multi-indices lead to numerators
%%
%\begin{multline}
%	\label{eq:multi_index_identification}
%	\jj_{k} = \ind{ i_{\q_{1}}, \dots, i_{\q_{a}-1}, i_{I}, i_{\q_{b}+1}, \dots, i_{\q_{s}} } \implies \\
%	g'_{\jj_{k}} = g'_{i_{I}} \prod_{\q \in S \setminus I} g_{i_{\q},\q} = \prod_{\q \in I} g_{i_{\q},\q} \prod_{\q \in S\setminus I} g_{i_{\q},\q} = g_{\ii_{k}},
%\end{multline}
%%
%so that $\mm[S]$ and (\ref{eq:maxmin_precomp_i}) lead to the same family of numerators of an Okutsu $S$-basis. 

The next result is an immediate consequence of the definition.

\begin{lemma}
	\label{lem:precomputation2} % critPrecomp2
	Let $S = I_{1} \cup \cdots \cup I_{t}$ be an order preserving partition of $S$. If all intervals $I_{j}$ admit precomputation, then $\mm[S=I_{1} \cup \cdots \cup I_{t}]$ has the same output as $\mm[S]$, after natural identifications.\qed
\end{lemma}

%Suppose that $\ii_{0}, \dots, \ii_{n_{S}}$ and $\jj_{0}, \dots, \jj_{n_{S}}$ are the outputs of $\mm[S]$ and  $\mm[S=I_{1} \cup \cdots \cup I_{t}]$, respectively. The natural identifications in this case are
%%
%\begin{align*}
%	\ii_{k} = \ind{ i_{\q_{1}}, \dots, i_{\q_{s}} } \implies \jj_{k} = \ind{i_{1}, \dots ,i_{t}},
%\end{align*}
%%
%and for all $1 \leq j \leq t$, the $i_{j}$-th output of $\mm[S_{j}]$ is the multi-index $\ind{i_{\q_{m}}}_{a_{j} \leq m \leq b_{j}}$.

Let us give a criterion for an interval to admit precomputation.

\begin{lemma}
	\label{lem:precomputation} % critPrecomp
	Let $\ii_{0}, \ii_{1}, \dots, \ii_{n_{S}}$ be the output of $\mm[S]$ and let $I \subseteq S$ be an interval of $S$. For each $0 \leq k \leq n_{S}$, let $\ii_{k} = \ind{i_{\q}}_{\q \in S}$ and denote 
	\begin{align*}
		g_{\ii_{k}} &= \prod_{\q \in S} g_{i_{\q},\q}, \qquad
		G_{\ii_{k}} = \prod_{\q \in S \setminus I} g_{i_{\q},\q}.
	\end{align*}

	Suppose that for each $0 \leq k \leq n_{S}$ the following condition holds
	\begin{align*}
		\w[I]{g_{\ii_{k}}} = \w[S]{g_{\ii_{k}}} &\implies \w[\p]{G_{\ii_{k}}} = \w[\q]{G_{\ii_{k}}},	& \forall\  \p,\q \in I.
	\end{align*}
	
	Then, $I$ admits precomputation.
\end{lemma}

\begin{proof}
	Let $(\ii_k)_{0 \leq k \leq n_{S}}$ be the output of $\mm[S]$ and $(\jj_k)_{0 \le k \le n_{S}}$ be the output of the precomputed $\mm$ algorithm \eqref{eq:maxmin_precomp_i}.
	
	Clearly, $\ii_{0}$ and $\jj_{0}$ may be identified. For $k \geq 0$, suppose that $\ii_{k}$ may be identified with $\jj_{k}$.  This means
	\begin{align*}
		\ii_{k} &= \ind{ i_{\q_{1}}, \dots, i_{\q_{s}} }, \\
		\jj_{k} &= \ind{ i_{\q_{1}}, \dots, i_{\q_{a}-1}, i_{I}, i_{\q_{b}+1}, \dots, i_{\q_{s}} },
	\end{align*}
	while the $i_{I}$-th output of $\mm[I]$ is the multi-index $\ind{ i_{\q_{j}} }_{a \leq j \leq b}$.
	
	Let $g'_{0},\, \dots,\, g'_{n_{I}}$ be the numerators deduced from the application of $\mm[I]$ and $g'_{\jj_{0}},\, \dots,\, g'_{\jj_{n_{S}}}$ the numerators deduced from \eqref{eq:maxmin_precomp_i}. Clearly,
	\begin{align*}
		g'_{\jj_{k}} &= g'_{i_{I}} \prod_{\q \in S \setminus I} g_{i_{\q},\q} = \prod_{\q \in I} g_{i_{\q},\q} \prod_{\q \in S\setminus I} g_{i_{\q},\q} = g_{\ii_{k}}, \qquad
		g'_{i_{I}} = \prod_{m = a}^{b} g_{i_{\q_{m}},\q_{m}}.
	\end{align*}
	
	The algorithm $\mm[S]$ outputs $\ii_{k+1} = \ii_{k} + \uu_{j}$, where
	\begin{align*}
		j = \Min\set{ 1 \leq m \leq s : \w[\q_{m}]{g_{\ii_{k}}} = \w[S]{g_{\ii_{k}}} }.
	\end{align*}
	
	If $\q_{j} \not\in I$, then the $\q_{j}$-index in $\jj_{k}$ will also be the least index satisfying $\w[\q_{j}]{g'_{\jj_{k}}} = \w[S]{g'_{\jj_{k}}}$, since $g'_{\jj_{k}} = g_{\ii_{k}}$. Thus, the algorithm in \eqref{eq:maxmin_precomp_i} will also increase the $\q_j$-coordinate.
	
	If $\q_{j} \in I$, then $\w[I]{g_{\ii_{k}}} = \w[S]{g_{\ii_{k}}}$ and $\w[\q_{m}]{g_{\ii_{k}}} > \w[S]{g_{\ii_{k}}}$ for all $m < a$; thus, \eqref{eq:maxmin_precomp_i} will increase $i_{I}$ by one. In this case, we must show that  the $(i_{I}+1)$-th output of $\mm[I]$ is the multi-index obtained from $\ind{i_{\q_{j}}}_{a \leq j \leq b}$ by increasing the $\q_{j}$-coordinate by one.
	
	The index increased by $\mm[I]$ will be:
	\begin{align*}
		J &= \Min\set{ a \leq m \leq b : \w[\q_{m}]{g'_{i_{I}}} = \w[I]{g'_{i_{I}}} }.
	\end{align*}
	
	By hypothesis, $\nu := \w[\q]{G_{\ii_{k}}}$ is independent of the choice of $\q \in I$. Since $g_{\ii_{k}} = G_{\ii_{k}} g'_{i_{I}}$,  we have:
	\begin{align*}
		\w[\q]{g_{\ii_{k}}} &= \w[\q]{g'_{i_{I}}} + \nu,	&\forall\  \q \in I.
	\end{align*}
	
	In particular, $\w[S]{g_{\ii_{k}}} = \w[I]{g_{\ii_{k}}} = \w[I]{g'_{i_{I}}} + \nu$, so that $J = j$.
\end{proof}

One specific case of precomputation which we will make use of, is the precomputation of certain intervals $S_{\type} \subseteq S$ defined by a type $\type$. 

\begin{lemma}
	\label{lem:st_precomputable}
	For any $\type \in \tree$, if the interval $S_{\type}$ is non-empty, it admits precomputation. 
\end{lemma}

\begin{proof}
	For every $\p \in S_{\type}$ and every $\q \not\in S_{\type}$, the explicit formulas from Proposition \ref{prop:explicit_valuations} show that $\w[\p]{\phi_{i,\q}}$ is independent of $\p$, for all $i$. Hence, the same is true for all polynomials $G_{\ii_{k}}$ that are a product of these $\phi$-polynomials.

	Thus, $S_{\type}$ meets the criterion of Lemma \ref{lem:precomputation}.
\end{proof}

In \cite{Stainsby:2014:thesis} we give concrete examples of intervals $I \subset S$ which do not admitprecomputation.

%-------------------------------------------------------------------------------
% Section: Reduction to the subset of indices divisible by $\ml$ 
%-------------------------------------------------------------------------------
\subsection{The block-wise MaxMin algorithm}
\label{sec:maxmin_mod_ml}

Consider an ordered subset $S = \set{ \q_{1}, \dots, \q_{s} } \subseteq \pp$. Let $\tree_{S}$ be the tree gathering all the paths of all leaves $\type_{\p} \in \tree$ for $\p \in S$. Take
\begin{align*}
	\ell = i(S) := \Min\set{i(\p,\q) : \p, \q \in S}.
\end{align*}
The tree $\tree_{S}$ is disconnected if and only if $\ell = 0$. Assume from now on that $\ell \ge 1$; in this case, $\ell-1$ is the order of the greatest common node of all paths joining the leaves of $\tree_{S}$ with the root node. In this case, the Okutsu frames of all primes $\p \in S$ have the same first $\ell-1$ key polynomials $\phi_{1}, \dots, \phi_{\ell-1}$. Thus, the first $\ml$ numerators of the Okutsu $\p$-bases coincide for all $\p \in S$. Let
\begin{align*}
	\Num = \set{ 1 = h_{0}, h_{1}, \dots, h_{\ml-1} },
\end{align*}
be the family of these common numerators. Note that
\begin{align}
	\label{eq:shared_basis_vals}
	\w[\p]{h} &= \w[\q]{h},		&\forall\ \p, \q \in S,\  \forall\  h \in \Num.
\end{align}

\begin{lemma}
	\label{lem:shared_basis_vals}
	For all $\p, \q \in S$ and all $0 \leq r,t < \ml$:
	\begin{align*}
		 \w[\q]{h_{r} h_{t}} \leq
		\begin{cases}
			\w[\q]{h_{r+t}},					& \text{if } r + t < \ml, \\
			\w[\q]{\phi_{\ell,\p} h_{k}},	& \text{if } r + t = \ml + k,\  k \ge 0.
		\end{cases}
	\end{align*}
\end{lemma}

\begin{proof}
	By Lemma \ref{lem:mostly_increasing_cross_values} (1), in any product of powers of $\phi_{1}, \dots, \phi_{\ell-1}$ we may replace $\phi_{i-1}^{e_{i-1} f_{i-1}}$ with $\phi_{i,\p}$ to increase the $\q$-valuation. This proves both inequalities.
\end{proof}

\begin{lemma}
	\label{lem:m_ell} % mell
	Let $\ii$ be a maximal multi-index of degree divisible by $\ml$.
	\begin{enumerate}
		\item There exists a maximal multi-index $\ii' = (i'_{\p})_{\p \in S}$ of the same degree, having all its coordinates $i'_{\p}$ divisible by $\ml$.
		\item All elements in the family $g_{\ii} \Num$ are maximal numerators.
	\end{enumerate}
\end{lemma}

\begin{proof}
	For $0 \leq j < \ml$, let $\jj = (j_{\p})_{\p \in S}$ be a multi-index of degree $i \ml + j$. Each index $j_{\p}$ may be written
	\begin{align*}
		j_{\p} &= q_{\p} \ml + k_{\p},	& 0 \leq k_{\p} < \ml,
	\end{align*}
	and the numerators $g_{j_{\p},\p}$ of the Okutsu $\p$-basis may be written
	\begin{align*}
		g_{j_{\p},\p} &= G_{\p} h_{k_{\p}},	& \deg G_{\p} = q_{\p} \ml.
	\end{align*}
	
	Since all polynomials $G_{\p}$ have a degree which is a multiple of $\ml$, we have $\sum_{\p \in S} k_{\p} = q \ml + j$,	for some non-negative integer $q$.
	
	Let $g = \prod_{\p \in S} G_{\p}$ and choose any fixed prime ideal $\pz \in S$. By \eqref{eq:shared_basis_vals}, an iterative application of the inequalities in Lemma \ref{lem:shared_basis_vals} shows that for any $\q \in S$ we have
	\begin{align*}
		\w[\q]{g_{\jj}} &= \w[\q]{ \prod_{\p \in S} G_{\p} h_{k_{\p}} } 
			= \w[\q]{ g } + \w[\q]{ \prod_{\p \in S} h_{k_{\p}} }
			\le \w[\q]{g} + \w[\q]{\phi_{\ell,\pz}^{q} h_{j}}.
	\end{align*}
	Since this holds for all $\q \in S$, we deduce that $\w[S]{g_{\jj}} \le \w[S]{g \phi_{\ell,\q}^{q} h_{j}}$.
	
%	where $\p_{0}$ is an arbitrary choice of a prime ideal in $S$. 
%
%	Consider the polynomial $h := \phi_{\ell-1}^{e_{\ell-1}f_{\ell-1}}$ of degree $\ml$. By an iterative application of the inequalities in \eqref{eq:sum_common_vals}, we get
%	%
%	\begin{align*}
%		\w[\q]{\prod_{\p \in S} h_{k_{\p}}} &\leq \w[\q]{h^{q}} + \w[\q]{h_{j}},	& \forall\  \q \in S.
%	\end{align*}
%	
%	Hence,
%	%
%	\begin{align*}
%		\w[S]{g_\jj)} &\le \w[S]{h^{q} \prod_{\p \in S} G_{\p}} + \w[\p_{0}]{h_{j}} \\
%			&= \w[S]{h^{q} h_{j} \prod_{\p \in S} G_{\p}} 
%			< \w[S]{\phi_{\ell,\p_{0}}^{q} h_{j} \prod_{\p \in S} G_{\p}}.
%	\end{align*}
%	%
%	The final inequality is a consequence of $\wdg[\p]{\phi_{k,\p_{0}}} < \wdg[\p]{\phi_{\ell,\p_{0}}}$, for all $k < \ell$ and $\p \in S$, which is shown in Lemma \ref{lem:mostly_increasing_cross_values}.
	
	These arguments, applied to $\jj = \ii$ (and $j = 0$) prove item (1). Also, applied to an arbitrary $\jj$ of degree $\deg(\ii) + j$ show that
	\begin{align*}
		\w[S]{g_{\jj}} &\leq \w[S]{\phi_{\ell,\p_{0}}^{q} \cdot g} + \w[\p_{0}]{h_{j}} \le \w[S]{g_{\ii}} + \w[\p_{0}]{h_{j}} = \w[S]{g_{\ii} h_{j}},
	\end{align*}
	by the maximality of $g_{\ii}$. This proves item (2).
\end{proof}

\begin{lemma}
	\label{lem:maxmin_blocks} % MMmell
	Let $\ii = (i_{\q})_{\q \in S}$ be an output multi-index of $\mm[S]$ of degree divisible by $\ml$.
	\begin{enumerate}
		\item All coordinates $i_{\q}$ are divisible by $\ml$.
		\item Let $j = \Min\set{ 1 \leq m \leq s : \w[\q_{m}]{g_{\ii}} = \w[S]{g_{\ii}} }$. Then, the next $\ml$ iterations of $\mm[S]$ increase the coordinate $\q_{j}$.
	\end{enumerate}
\end{lemma}

\begin{proof}
	All coordinates of $\ii_{0}$ are zero; hence divisible by $\ml$. Thus, it suffices to prove that any output multi-index $\ii=(i_{\q})_{\q \in S}$ whose coordinates are all divisible by $\ml$ satisfies (2).
	
	Let $j = \Min\set{ 1 \leq m \leq s : \w[\q_m]{g_{\ii}} = \w[S]{g_{\ii}} }$. If $\ii = \ii_{k}$ is the $k$-th output multi-index of $\mm[S]$, the algorithm selects $\ii_{k+1} = \ii_{k} + \uu_{j}$. Since $i_{\q_{j}}$ is a multiple of $\ml$, we have $g_{\ii_{k+1}} = g_{\ii_{k}} h_{1}$; hence, 
	\begin{align*}
		\w[\q]{g_{\ii_{k+1}}} &= \w[\q]{g_{\ii_{k}}} + \w[\q]{h_{1}}
			\geq \w[S]{g_{\ii_{k}}} + \w[\q]{h_{1}}
			= \w[S]{g_{\ii_{k+1}}},
	\end{align*}
	for all $\q \in S$. Thus, $\w[\q]{g_{\ii_{k+1}}} = \w[S]{g_{\ii_{k+1}}}$ if and only if $\w[\q]{g_{\ii_{k}}} = \w[S]{g_{\ii_{k}}}$. Thus, the next iteration increases the $\q_{j}$-coordinate again. By iterating this argument, we get $g_{\ii_{k+\ml-1}} = g_{\ii_{k}} h_{\ml-1}$. At this point, the $\q_{j}$-coordinate will be increased once more to yield $\ii_{k+\ml} = \ii_{k} + \ml \uu_{j}$.
\end{proof}

This result shows that $\mm[S]$ works by blocks of length $\ml$. Thus, we may consider Algorithm \ref{alg:maxmin_ml}, where we agree that $\ml = 1$ if $\tree_{S}$ is disconnected. Note that for $\ml = 1$, $\mm[S;\ml]$ coincides with $\mm[S]$.

\begin{algorithm}
	\caption{$\mm[S;\ml]$ algorithm}
	\label{alg:maxmin_ml}
	\begin{algorithmic}[1]
		\INPUT An ordered subset $S = \set{\q_{1}, \dots, \q_{s}} \subseteq \pp$ and extended families $\set{g_{i,\q} : 0 \le i \le n_{\q}}$ of numerators of Okutsu $\q$-bases of each $\q \in S$.
						
		\OUTPUT A family $\ii_{0}, \ii_{\ml}, \ii_{2\ml}, \dots, \ii_{n_{S}/\ml}$ of multi-indices with  $\deg \ii_{k} = k$, having all coordinates divisible by $\ml$.
				
		\STATE $\ii_{0} \gets \ind{ 0, \dots, 0 }$
		
		\FOR {$k = 0 \to (n_{S}/\ml)-1$}
		
			\STATE $j \gets \Min\set{ 1 \leq i \leq s : \w[\q_{i}]{g_{\ii_{k\ml}}} = \w[S]{g_{\ii_{k\ml}}} }$
						
			\STATE $\ii_{(k+1)\ml} \gets \ii_{k\ml} + \ml \uu_{j}$
			
		\ENDFOR
	\end{algorithmic}
\end{algorithm}

%Theorem \ref{thm:maxmin_works} will be a consequence of the following result.

\begin{theorem}
	\label{thm:maxmin_ml_works} % MMmlworks
	The output multi-indices of $\mm[S;\ml]$ are maximal amongst all multi-indices of the same degree with coordinates divisible by~$\ml$.
\end{theorem}

Theorem \ref{thm:maxmin_works} follows from Theorem \ref{thm:maxmin_ml_works}. In fact, by Lemma \ref{lem:m_ell}, all output multi-indices of $\mm[S;\ml]$ will be maximal and by Lemma \ref{lem:maxmin_blocks} these multi-indices coincide with the output multi-indices of degree divisible by $\ml$ of $\mm[S]$.

Finally, Lemma \ref{lem:maxmin_blocks} shows how to derive all intermediary output multi-indices of $\mm[S]$ and Lemma \ref{lem:m_ell} shows that these multi-indices are maximal too.

%-------------------------------------------------------------------------------
% Section: Proof of Theorem \ref{thm:maxmin_ml_works} 
%-------------------------------------------------------------------------------
\subsection{Branching cases}
%Proof of Theorem \ref{thm:maxmin_ml_works}}
\label{sec:proof_maxmin_ml_works}

The basic idea for the proof of Theorem \ref{thm:maxmin_ml_works} is to split $S = U \cup D$ ($U$ for ``up'' and $D$ for ``down'') into the disjoint union of two intervals which admit precomputation and then analyse the behaviour of $\mm[S=U \cup D]$ for which the multi-indices have only two coordinates.

Lemma \ref{lem:maxmin_blocks} shows that the output multi-indices of $\mm[S; \ml]$ coincide with the output of an ordinary application of the 2-dimensional $\mm$ applied to the precomputations $\mm[U; \ml]$ and $\mm[D; \ml]$. We shall denote this algorithm by $\mm[S = U \cup D; \ml]$.

We distinguish four cases according to the structure of the tree $\tree_{S}$:

\begin{alphcase} % Case (A) - disconnected
	\label{case:disconnected_tree}
	The tree $\tree_{S}$ is disconnected, composed of $t$ connected trees with root nodes $\psi_{0,1}, \dots, \psi_{0,t}$. We take $D$ to be the connected component of $\tree_{S}$ with root node $\psi_{0,t}$.
\end{alphcase}

For $\tree_{S}$ connected, the proof of Theorem \ref{thm:maxmin_ml_works} makes use of the structure of the non-optimised tree with base type $\type_{\ell-1}$ which is the greatest common node in all paths joining the leaves of $\tree_{S}$ with the root node.

Let $\phi_{\ell}$ be the first representative of $\type_{\ell-1}$ which leads to branching. Thus, before constructing $\phi_{\ell}$, the Montes algorithm may have constructed other representatives of $\type_{\ell-1}$ admitting unibranch refinements.

Let $\lmin$ be the least slope (in absolute size) occurring in the branching based on $\phi_{\ell}$. Let $\Smin \subseteq S$ be the subset of all prime ideals derived from branches of slope $\lmin$
of $\phi_{\ell}$.

\begin{alphcase} % Case (B) - refined lambda_min branch
	\label{case:refined_lmin}
	There exists a branch $\type$ with slope $\lmin$ which suffered refinement. In this case, we take $D = S_{\type}$ to be the set of all prime ideals derived from this branch. Note that $\phi_{\ell,\p} \neq \phi_{\ell}$ $\forall\ \p \in D$, and that there may be other $\lmin$-branches.
	
	\begin{figure}[!h]
		\centering
		\begin{tikzpicture}[scale=0.6]
			\draw[anchor=east] (0, 0) node {$\psi_{0}$};
			\filldraw (0, 0) circle (2.5pt);
			\filldraw (2, 0) circle (2.5pt);
			\filldraw (4, 0) circle (2.5pt);
			\filldraw (6, 0) circle (2.5pt);
			\draw (0, 0) -- (2, 0);
			\draw (3, 0) node {$\cdots$};
			\draw (4, 0) -- (6, 0);
			\draw[anchor=south] (6, 0) node {$\type_{\ell-1}$};
			\draw (9, 0) node {$\cdots$};
			\draw[dotted] (6, 0) -- (8, 0);
			\draw[dotted] (10, 0) -- (13, 2);
			\draw[dotted] (10, 0) -- (13, -2);
			\draw[anchor=south] (10, 0) node {$\phi_{\ell}$};
			\draw[anchor=west] (11.5, 0.7) node {$\lambda, \psi'$};
			\draw[anchor=west] (11.5, -0.7) node {$\lmin, \psi$};
			\draw[black,fill=white] (8, 0) circle (2.5pt);
			\draw[black,fill=white] (10, 0) circle (2.5pt);
			\draw[black,fill=white] (13, 2) circle (2.5pt);
			\draw[black,fill=white] (13, -2) circle (2.5pt);
			\draw[anchor=west] (13.2, -2) node {$\type$};
		\end{tikzpicture}
		\caption{Case \ref{case:refined_lmin}. Tree $\treenop_{S}$ with at least one refined $\lmin$-branch.}
	\end{figure}
\end{alphcase}

\begin{alphcase} % Case (C) - all branches unrefined lambda_min
	\label{case:all_lmin}
	None of the $\lmin$-branches suffered refinement, and there are no other slopes. In other words, $\lambda_{\ell,\p} = \lmin$ and $\phi_{\ell,\p} = \phi_{\ell}$ for all $\p \in S$.
	
	In this case, we take $D = S_{\type}$, for any choice of a $\lmin$-branch $\type$.
	
	\begin{figure}[!h]
		\centering
		\begin{tikzpicture}[scale=0.6]
			\draw[anchor=east] (0, 0) node {$\psi_{0}$};
			\filldraw (0, 0) circle (2.5pt);
			\filldraw (2, 0) circle (2.5pt);
			\filldraw (4, 0) circle (2.5pt);
			\filldraw (6, 0) circle (2.5pt);
			\draw (0, 0) -- (2, 0);
			\draw (3, 0) node {$\cdots$};
			\draw (4, 0) -- (6, 0);
			\draw[anchor=south] (6, 0.3) node {$\type_{\ell-1}$};
			\filldraw (8.5, 2) circle (2.5pt);
			\filldraw (8.5, 0) circle (2.5pt);
			\filldraw (8.5, -2) circle (2.5pt);
			\draw (6, 0) -- (8.5, 2);
			\draw (6, 0) -- (8.5, 0);
			\draw (6, 0) -- (8.5, -2);
			\draw[anchor=west] (8.5, 2) node {$(\type_{\ell-1}; (\phi_{\ell}, \lmin, \psi'))$};
			\draw[anchor=west] (8.5, 0) node {$(\type_{\ell-1}; (\phi_{\ell}, \lmin, \psi^{*}))$};
			\draw[anchor=west] (8.5, -2) node {$(\type_{\ell-1}; (\phi_{\ell}, \lmin, \psi)) = \type$};
		\end{tikzpicture}
		\caption{Case \ref{case:all_lmin}. Tree $\tree_{S}$ with only unrefined $\lmin$-branches.}
	\end{figure}
\end{alphcase}

\begin{alphcase} % Case (D) - non-refined lambda_min + other branches
	\label{case:unrefined_lmin}
	None of the $\lmin$-branches suffered refinement, but there are other slopes. In other words, $\lambda_{\ell,\p} = \lmin$ and $\phi_{\ell,\p} = \phi_{\ell}$, for all $\p \in \Smin \subsetneq S$.

	In this case, we take $D = \Smin$.
	
	\begin{figure}[!h]
		\centering
		\begin{tikzpicture}[scale=0.6]
			\draw[anchor=east] (0, 0) node {$\psi_{0}$};
			\filldraw (0, 0) circle (2.5pt);
			\filldraw (2, 0) circle (2.5pt);
			\filldraw (4, 0) circle (2.5pt);
			\filldraw (6, 0) circle (2.5pt);
			\draw (0, 0) -- (2, 0);
			\draw (3, 0) node {$\cdots$};
			\draw (4, 0) -- (6, 0);
			\draw[anchor=south] (6, 0.3) node {$\type_{\ell-1}$};
			\filldraw (8.5, 2.5) circle (2.5pt);
			\filldraw (8.5, 0.5) circle (2.5pt);
			\filldraw (8.5, -0.5) circle (2.5pt);
			\filldraw (8.5, -2.5) circle (2.5pt);
			\draw (6, 0) -- (8.5, 2.5);
			\draw (6, 0) -- (8.5, 0.5);
			\draw (6, 0) -- (8.5, -0.5);
			\draw (6, 0) -- (8.5, -2.5);
			\draw (8.5, 1.65) node {$\vdots$};
			\draw (8.5, -1.25) node {$\vdots$};
			\draw[anchor=west] (8.5, 2.5) node {$(\type_{\ell-1}; (\phi''_{\ell}, \lambda'', \psi''))$};
			\draw[anchor=west] (8.5, 0.5) node {$(\type_{\ell-1}; (\phi'_{\ell}, \lambda', \psi'))$};
			\draw[anchor=west] (8.5, -0.5) node {$(\type_{\ell-1}; (\phi_{\ell}, \lmin, \psi^{*}))$};
			\draw[anchor=west] (8.5, -2.5) node {$(\type_{\ell-1}; (\phi_{\ell}, \lmin, \psi))$};
			\draw [decorate,decoration={brace,amplitude=6pt}] (14.5, -0.1) -- (14.5, -3);
			\draw[anchor=west] (15, -1.5) node {$\lmin$-branches};
		\end{tikzpicture}
		\caption{Case \ref{case:unrefined_lmin}. Tree $\tree_{S}$ with with unrefined $\lmin$-branches and other slopes.}
	\end{figure}
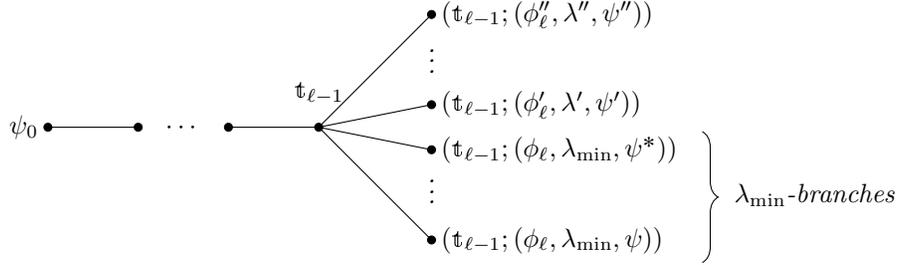
\end{alphcase}

In all cases, we may change the ordering of $\pp$ so that $D$ and $U = S \setminus D$ are intervals.

\subsubsection{Proof of Theorem \ref{thm:maxmin_ml_works} in cases \ref{case:disconnected_tree}, \ref{case:refined_lmin} and \ref{case:all_lmin}}

Denote
\begin{align*}
	c := \begin{cases}
		\dfrac{V_{\ell} + \lmin}{\eto{\ell-1}}, 	& \text{ if } \tree_{S} \text{ is connected}, \\
		0,							&\text{ if } \tree_{S} \text{ is disconnected}.
	\end{cases}
\end{align*}
The explicit formulas from Proposition \ref{prop:explicit_valuations} show that  
\begin{align}
	\label{eq:pvalq_eq_qvalp} % PQ
	\w[\p]{\phi_{m,\q}} &= (m/\ml) c = \w[\q]{\phi_{m,\p}},		& \forall\  \p \in U,\  \q \in D,\  \forall\  m \ge \ell.
\end{align}

On the other hand, all ideas and criteria about precomputation apply to the $\mm$ algorithms restricted to all multi-indices whose coordinates are divisible by $\ml$. Hence, \eqref{eq:pvalq_eq_qvalp} shows that $D$ and $U = S \setminus D$ meet the condition of Lemma \ref{lem:precomputation} and both intervals admit precomputation.

Denote the respective output families of numerators of $\mm[U;\ml]$ and $\mm[D;\ml]$ by:
\begin{align*}
	&1, g_{1}, \dots, g_{n_{U}/\ml}; \qquad
	1, g'_{1}, \dots, g'_{n_{D}/\ml}.
\end{align*}
Note that $\deg g_{k} = k\ml$, for all $k \le n_{U}/\ml$ and $\deg g'_{k} = k\ml$, for all $k \le n_{D}/\ml$.

By Lemma \ref{lem:precomputation2}, $\mm[S;\ml]$ has the same output as $\mm[S=U \cup D;\ml]$, after natural identifications of the respective multi-indices. In other words, if $(i,j)$ is the $k$-th output of $\mm[S=U \cup D;\ml]$ (so that $k = i + j$), then the $k$-th numerator provided by $\mm[S;\ml]$ is $g_{i} g'_{j}$.

\begin{definition}
	We say that a monic polynomial $G \in \oox$ \emph{has support} in a subset $S' \subset S$ if it is a product of polynomials $\phi_{m,\p}$ for $\p \in S'$ and $m \ge \ell$.
	
	Note that the degree of $G$ is necessarily a multiple of $\ml$.
\end{definition}

In order to prove Theorem \ref{thm:maxmin_ml_works}, we must show that the output numerators of $\mm[S; \ml]$ are maximal amongst all polynomials of the same degree with support in $S$.
 
We proceed by induction on $\#S$. The case $\#S=1$ being trivial, we may assume by the induction hypothesis that both sequences of numerators are maximal amongst all polynomials of the same degree with support in $U$ and $D$, respectively.

For all $0\le i\le n_U/\ml$ and all $0\le j\le n_D/\ml$, denote 
\begin{align}\label{eq:def_nu}
	\nu_{i} &:= \w[U]{g_{i}} - ic, \qquad
	\nu'_{j} := \w[D]{g'_{j}} - jc.
\end{align}

We agree that $\nu_{-1} = \nu'_{-1} = -1$.

\begin{lemma}
	\label{lem:increasing} % increasing
	For all $i, j \geq 0$,
	\begin{align*}
		\nu_{i} &\leq \nu_{i+1}, \qquad
		\nu'_{j} \leq \nu'_{j+1}.
	\end{align*}
\end{lemma}

\begin{proof}
	By Proposition \ref{prop:explicit_valuations}, $\wt[\q]{\phi_{\ell,\q}} = (V_\ell + \lambda_{\ell,\q})/(\eto{\ell-1})$ for all $\q \in U$. Since $\lambda_{\ell,\q} \ge \lmin$, the maximality of $g_{i+1}$ implies
	\begin{align*}
		\w[U]{g_{i+1}} &\geq \w[U]{g_{i} \phi_{\ell,\q}} \geq \w[U]{g_{i}} + c,	& \forall\  \q \in U.
	\end{align*}
	Similarly, the maximality of $g'_{j+1}$ implies
	\begin{align*}
		\w[D]{g'_{j+1}} \geq \w[D]{g'_{j} \phi_{\ell}} = \w[D]{g'_{j}} + c.
	\end{align*}
	By the definition \eqref{eq:def_nu} of $\nu_{i}$, $\nu'_{j}$, this ends the proof of the lemma.
\end{proof}

For any bi-index $\ii = (i,j)$, and any $\p \in U$, $\q \in D$, \eqref{eq:pvalq_eq_qvalp} shows that
%
%\begin{align*}
%	\w[\p]{g_{\ii}} &= \w[\p]{g_{i} g'_{j}} = \wt[\p]{g_{i}} + j c, \\
%	\w[\q]{g_{\ii}} &= \w[\q]{g_{i} g'_{j}} = \wt[\q]{g'_{j}} + i c.
%\end{align*}
%%
%Hence,
%
\begin{align*}
	\w[U]{g_{\ii}} &= \w[U]{g_{i}} + jc = \nu_{i} + (i+j) c = \nu_{i} + (\deg \ii) c, \\
	\w[D]{g_{\ii}} &= \w[D]{g'_{j}} + ic = \nu'_{j} + (i+j) c = \nu'_{j} + (\deg \ii) c, \\
	\w[S]{g_{\ii}} &= \w[S]{g_{i} g'_{j}} = \Min\set{\nu_{i}, \nu'_{j}} + (\deg\ii) c.
\end{align*}

Therefore, these numbers $\nu_{i}$, $\nu'_{j}$ determine the flow of $\mm[S=U\cup D;\ml]$. If $(i,j)$ is an output pair, the next output pair is decided as follows:
\begin{align*}
	\w[U]{g_{i} g'_{j}} &= \w[S]{g_{i} g'_{j}} \sii \nu_{i} \leq \nu'_{j},	& \text{``$U$-minimal''}, \\
	\w[D]{g_{i} g'_{j}} &= \w[S]{g_{i} g'_{j}} \sii \nu'_{j} < \nu_i,	& \text{``$D$-minimal''}.
\end{align*}

The next output pair is $(i+1,j)$ in the $U$-minimal case, and $(i,j+1)$ in the $D$-minimal case. 

\begin{proposition}
	\label{prop:separated_maximal} % sepmax
	The output bi-indices $(i,j)$ of $\mm[S=U\cup D;\ml]$ satisfy the following properties:
	\begin{enumerate}
		\item Either $\nu'_{j-1} \leq \nu_{i} \leq \nu'_{j}$, or $\nu_{i-1} \leq \nu'_{j} < \nu_{i}$.
		\item The output multi-indices of $\mm[S = U \cup D; \ml]$ are maximal amongst all polynomials of the same degree with suppose in $S$.
	\end{enumerate}
\end{proposition}

\begin{proof}
	Clearly, the initial output pair $(0,0)$ satisfies (1). Let us check that if an output pair $(i,j)$ satisfies (1), then the next output pair satisfies (1) as well.
	
	Suppose that $\nu'_{j-1} \leq \nu_{i} \leq \nu'_{j}$, so that the next output pair is $(i+1,j)$.
	\begin{align*}
		\nu_{i+1} \leq \nu'_{j} &\implies \nu'_{j-1} \leq \nu_{i} \le \nu_{i+1} \leq \nu'_{j}, \\
		\nu_{i+1} > \nu'_{j} &\implies \nu_{i} \leq \nu'_{j} < \nu_{i+1}.
	\end{align*}
	
	Suppose that $\nu_{i-1}\le\nu'_j<\nu_i$, so that the next output pair is $(i,j+1)$.
	\begin{align*}
		\nu_{i} \leq \nu'_{j+1} &\implies \nu'_{j} < \nu_{i} \leq \nu'_{j+1}, \\
		\nu_{i} > \nu'_{j+1} &\implies \nu_{i-1} \leq \nu'_{j} \leq \nu'_{j+1} < \nu_{i}.
	\end{align*}
	
	This proves item (1). As a consequence, for any $k \in \Z$ such that $0 \leq i-k \leq n_{U}/\ml$ and $0 \leq j+k \leq n_{D}/\ml$, we have:
	\begin{align}
		\label{eq:min_leq_min}
		\Min\set{\nu_{i-k}, \nu'_{j+k}} \leq \Min\set{\nu_{i}, \nu'_{j}}.
	\end{align}
	
	In fact, if $\nu'_{j-1} \leq \nu_{i} \leq \nu'_{j}$, then  $\Min\set{\nu_{i-k}, \nu'_{j+k}} \leq \nu_{i}$, whereas in the case  $\nu_{i-1} \leq \nu'_{j} < \nu_{i}$, we have $\Min\set{\nu_{i-k}, \nu'_{j+k}} \leq \nu'_{j}$.
	
	In order to prove (2), suppose that $(i,j)$ is an output pair of $\mm[S=U\cup D;\ml]$ and let $g$ be a polynomial of degree $(i+j)\ml$ with support in $S$. We may write $g = GG'$, with $G$, $G'$ polynomials with support in $U$ and $D$, respectively.

	Suppose $\deg G=(i-k)\ml$, $\deg G'=(j+k)\ml$, for certain $k\in\Z$. By \eqref{eq:pvalq_eq_qvalp} and the maximality of the numerators $g_{i-k}$, $g'_{j+k}$, we have:
	\begin{align*}
		 \w[U]{g} &= \w[U]{G G'} = \w[U]{G} + (j+k) c \\
			&\le \nu_{i-k} + (i-k) c + (j+k) c = \nu_{i-k} + (i+j) c,
	\end{align*}
	\begin{align*}
		\w[D]{g} &= \w[D]{G G'} = \w[D]{G'} + (i-k) c \\
			&\le \nu'_{j+k} + (j+k) c + (i-k) c = \nu'_{j+k} + (i+j) c.
	\end{align*}
	
	Hence, by using \eqref{eq:min_leq_min}, we get:
	\begin{align*}
		\w[S]{g} &= \Min\set{\w[U]{g}, \w[D]{g}} = \Min\set{\nu_{i-k}, \nu'_{j+k}} + (i+j) c \\
			&\le \Min\set{\nu_{i}, \nu'_{j}} + (i+j) c = \w[S]{g_{i} g'_{j}}.
			\qedhere
	\end{align*}
\end{proof}

This ends the proof of Theorem \ref{thm:maxmin_ml_works} in cases \ref{case:disconnected_tree}, \ref{case:refined_lmin} and \ref{case:all_lmin}.

\subsubsection{Precomputation in Case \ref{case:unrefined_lmin}}

Recall that $D = \Smin$ and $U = S \setminus D$. In this case, we have:
\begin{align*}
	\phi_{\ell,\q} &= \phi_{\ell} = \phi(\p,\q),		\qquad \forall\  \p \in U, \q \in D.
\end{align*}

For each $\p \in S$ we denote by $\lambda_{\p}$ the slope of the branch of $\phi_{\ell}$ in the non-optimised tree to which the leaf of $\p$ belongs. Also, we denote
\begin{align*}
	c := \frac{V_{\ell} + \lmin}{\eto{\ell-1}}, \qquad
	\delta_{\p} := \frac{\lambda_{\p} - \lmin}{\eto{\ell-1}} \ge 0.
\end{align*}

The explicit formulas presented in Proposition \ref{prop:explicit_valuations} show that for all $\p \in U, \q \in D$:
\begin{align}
	\label{eq:pqval_phi} % PQC
	\begin{aligned}
		\w[\q]{\phi_{i,\p}} &= (m_{i}/\ml) c,	& \forall\  i \ge \ell, \\
		\w[\p]{\phi_{i,\q}} &=
		\begin{cases}
			(m_{i}/\ml) c,		& \text{if } i > \ell,\\
			\delta_{\p} + c,	& \text{if } i = \ell.
		\end{cases}
	\end{aligned}
\end{align}

Let $G$ be a polynomial of degree $i \ml$ with support in $U$, and let $G'$ be a polynomial of degree $j \ml$ with support in $D$. If $m := \ord_{\phi_{\ell}}(G')$, the formulas \eqref{eq:pqval_phi} show that:
\begin{align}
	\label{eq:udval_gg} % gg
	\begin{aligned}
		\w[U]{G G'} &= \Min\set{\w[\p]{G} + m \delta_{\p}}_{\p \in U} + jc, \\
		\w[D]{G G'} &= \w[D]{G'} + ic.
	\end{aligned}
\end{align}

The first formula of \eqref{eq:pqval_phi} shows that $D$ meets the criterion of Lemma \ref{lem:precomputation} and admits precomputation. In order to show that $U$ admits precomputation too, we need another lemma.\medskip

\begin{notation}
	For each $\p \in D$, we denote $m_{\p} := m_{\ell+1,\p} = e_{\ell,\p} f_{\ell,\p} \ml$.
\end{notation}

%Note that $e_{\ell,\p}$ is independent of $\p$, because it is the least positive denominator of $\lmin$.

\begin{lemma}
	\label{lem:bb} % BB
	Let $\ii = (i_\p)_{\p \in S}$ be an output of $\mm[S;\ml]$ and $g = g_{\ii}$ the corresponding numerator. Let $\p \in S$ be the least prime with $\w[\p]{g} = \w[S]{g}$.
	\begin{enumerate}
		\item If $\p \in D$ and $m_{\p} \divides i_{\p}$, then the next $e_{\ell,\p}f_{\ell,\p}$ output numerators are
		\begin{align*}
			g \phi_{\ell},\  g \phi_{\ell}^{2},\  \dots,\  g \phi_{\ell}^{e_{\ell,\p} f_{\ell,\p}-1},
		\end{align*}
		and finally
		\begin{align*}
			\left( \prod_{\q \neq \p} g_{i_{\q},\q} \right) \cdot g_{i_{\p} + m_{\p},\p}.
		\end{align*}	
				
		\item If $\p \in U$, then $m_{\q} \divides i_{\q}$ for all $\q \in D$.
	\end{enumerate}
\end{lemma}

\begin{proof}
	Suppose $\p \in D$ and $m_{\p} \divides i_{\p}$. Since the element $g_{i_{\p},\p}$, a numerator of the Okutsu $\p$-basis, has degree divisible by $m_{\p}$, it is not divisible by $\phi_{\ell}$, and $g_{i_{\p}+1,\p} = g_{i_{\p},\p} \phi_{\ell}$. Hence, the next output numerator is $g \phi_{\ell}$. 
	
	By \eqref{eq:pqval_phi}, $\w[\p]{g \phi_{\ell}} = \w[\p]{g} + c$, while $\w[\q]{g \phi_{\ell}} \ge \w[\q]{g} + c$ for all $\q \in S$. Thus, the least prime with $\w[\q]{g \phi_{\ell}} = \w[S]{g \phi_{\ell}}$ is, once again, the prime $\p$.
	
	This argument may be iterated as long as $\ord_{\phi_{\ell}}(g_{i_{\p} + k \ml,\p}) = k < e_{\ell,\p} f_{\ell,\p}$. For $k = e_{\ell,\p} f_{\ell,\p} - 1$, the prime $\p$ is still the least one satisfying $\w[\p]{g \phi_{\ell}^{k}} = \w[S]{g \phi_{\ell}^{k}}$, so that the component of the multi-index corresponding to $\p$ is increased and the output multi-index is $\ii + m_{\p} \uu_{\p}$.
		
	The second item follows immediately from the first. 
\end{proof}

\begin{corollary}
	$U$ admits precomputation.
\end{corollary}

\begin{proof}
	Let us show that $U$ meets the criterion of Lemma \ref{lem:precomputation}.
	  
	Let $\ii = (i_{\p})_{\p \in S}$ be an output of $\mm[S;\ml]$ and let $g = g_{\ii}$ be the corresponding numerator. Suppose that $\w[U]{g} = \w[S]{g}$. With respect to the ordering of $S$, all elements in $U$ are less than all elements in $D$; hence, the least prime $\p$ with $\w[\p]{g} = \w[S]{g}$ belongs to $U$. By (2) of Lemma \ref{lem:bb}, $m_{\q} \divides i_{\q}$ for all $\q \in D$, and this implies that none of the numerators $g_{i_{\q},\q}$, for $\q \in D$, is divisible by $\phi_{\ell}$.
	
	Therefore, \eqref{eq:pqval_phi} shows that $\w[\p]{g_{i_{\q},\q}} = (i_{\q}/\ml)c$ for all $\p \in U$, and the value $\w[\p]{G_{\ii}} = \w[\p]{\prod_{\q \in D} g_{i_{\q},\q}}$ is independent of $\p \in U$.
\end{proof}

\subsubsection{Proof of Theorem \ref{thm:maxmin_ml_works} in Case \ref{case:unrefined_lmin}}

Denote the respective output families of numerators of $\mm[U;\ml]$ and  $\mm[D;\ml]$ by:
\begin{align*}
	&1, g_{1}, \dots, g_{n_{U}/\ml}; \qquad
	1, g'_{1}, \dots, g'_{n_{D}/\ml}.
\end{align*}
Note that $\deg g_{k} = k\ml$, for all $k \le n_{U}/\ml$ and $ \deg g'_{k} = k\ml$ for all $k \le n_{D}/\ml$.

Let $\ii=(i_{\p})_{\p \in S}$ be an output of $\mm[S;\ml]$. Since $U$ and $D$ admit precomputation, Lemma \ref{lem:precomputation2} shows that $g_{\ii} = g_{i} g'_{j}$, for the $k$-th output $(i,j)$ of $\mm[S=U\cup D;\ml]$.

\begin{notation}
	We denote $[j] := \ord_{\phi_{\ell}}(g'_{j})$, for $0 \leq j \le n_{D}/\ml$.
\end{notation}

By Lemma \ref{lem:bb}, all indices $i_{\q}$, for $\q \in D$ are divisible by $m_{\q}$ except eventually for one, say $i_{\q_{0}}$. Hence, $[j]$ is the residue of the euclidian division of $i_{\q_{0}}$ by $m_{\q_{0}}$. Note that $[j]=0$ if and only if $m_{\q} \divides i_{\q}$ for all $\q \in D$. 

Consider rational numbers $\nu_{i}, \nu'_{j}$ as in \eqref{eq:def_nu}. The formulas \eqref{eq:udval_gg} translate into
\begin{align}
	\label{eq:udval_gg2} % gg2
	\begin{aligned}
		\w[U]{g_{i} g'_{j}} &=\Min\set{\w[\p]{g_{i}} + [j] \delta_{\p}}_{\p \in U} + jc,\\
		\w[D]{g_{i} g'_{j}} &= \w[D]{g'_{j}} + ic = \nu'_{j} + (i+j) c.  
	\end{aligned}
\end{align}

\begin{lemma}
	\label{lem:increasing2} % increasing2
	These data $\nu_{i}, \nu'_{j}$ satisfy the following properties for all $i, j > 0$:
	\begin{enumerate}
		\item $\nu_{i-1} < \nu_{i}$.
		\item $\nu'_{j-1} \leq \nu'_{j}$ and if $[j] \neq 0$ then equality holds.
	\end{enumerate}
\end{lemma}

\begin{proof}
	Take any $\p \in U$. By Proposition \ref{prop:explicit_valuations}, $\w[\p]{\phi_{\ell,\p}} = (V_{\ell} + \lambda_{\ell,\p})/(\eto{\ell-1}) > c$, since $\lambda_{\ell,\q} > \lmin$. The maximality of $g_{i}$ implies
	\begin{align*}
		\w[U]{g_{i}} \geq \w[U]{g_{i-1} \phi_{\ell,\p}} > \w[U]{g_{i-1}} + c.
	\end{align*}
		
	This proves (1). Similarly, the maximality of $g'_{j}$ implies $\nu'_{j-1} \leq \nu'_{j}$.
	
	By Lemma \ref{lem:bb}, $[j] \ne 0$ implies that $g'_{j} = g'_{j-1} \phi_{\ell}$. Since $\w[\q]{\phi_{\ell}} = c$ for all $\q \in D$, this implies $\w[D]{g'_{j}} = \w[D]{g'_{j-1}} + c$. This proves (2).
\end{proof}

\begin{lemma}
	\label{lem:sepmax21} % sepmax21
	Let  $\ii = (i,j)$ be an output pair of $\mm[S=U\cup D;\ml]$. Then,
	\begin{enumerate}
		\item Either $\nu'_{j-1} \leq \nu_{i} \leq \nu'_{j}$, or $\nu_{i-1} \leq \nu'_{j} < \nu_{i}$.
		\item The next output pair is $(i+1,j)$ in the first case, and $(i, j+1)$ in the second case. 
	\end{enumerate}
\end{lemma}

\begin{proof} 
	Clearly, the initial pair $(0,0)$ satisfies (1), the next output pair is $(1,0)$, and it satisfies (1) too. Let us show by induction that if an output pair satisfies (1) then the next output pair is given as indicated in (2) and it satisfies (1) as well.
	
	Suppose that $\nu'_{j-1} \leq \nu_{i} \leq \nu'_{j}$. If the previous output pair was $(i-1,j)$, the induction hypothesis implies that we had $\nu'_{j-1} \leq \nu_{i-1} \leq \nu'_{j}$. Since $\nu'_{j-1} \leq \nu_{i-1} < \nu_{i} \leq\nu'_j$, we have $[j]=0$ by item (2) of Lemma \ref{lem:increasing2}.  If the previous output pair was $(i,j-1)$, then  $\nu'_{j-1} <  \nu_{i}$ by the induction hypothesis. This leads again to $\nu'_{j-1} < \nu'_{j}$ and to $[j]=0$. Thus, \eqref{eq:udval_gg2} shows that
	\begin{align*}
		\w[U]{g_{i} g'_{j}} &= \w[U]{g_{i}} + jc = \nu_{i} + (i+j) c \\
			&\leq \nu'_{j} + (i+j) c = \w[D]{g_{i} g'_{j}}.
	\end{align*}
	
	Thus, $\w[U]{g_{i} g'_{j}} = \w[S]{g_{i} g'_{j}}$ and the next output pair is $(i+1,j)$. The arguments of the proof of Proposition \ref{prop:separated_maximal} show that $(i+1,j)$ satisfies (1).
	
	Suppose that $\nu_{i-1} \leq \nu'_{j} < \nu_{i}$. By \eqref{eq:udval_gg2}, we have
	\begin{align*}
		\w[D]{g_{i} g'_{j}} &= \nu'_{j} + (i+j) c < \nu_{i} + (i+j) c \\
			&= \w[U]{g_{i}} + jc = \w[U]{g_{i} g'_{j}}.
	\end{align*}
	
	Thus, $\w[D]{g_{i} g'_{j}} = \w[S]{g_{i} g'_{j}}$ and the next output pair is $(i,j+1)$. The arguments of the proof of Proposition \ref{prop:separated_maximal} show that $(i,j+1)$ satisfies (1).
\end{proof}

\begin{lemma}
	\label{lem:cas1} % aux
	Consider indices $0 \leq k \leq i$ and let $G$ be a polynomial of degree $(i-k)\ml$ with support in $U$. Then,
	\begin{align*}
		\Min\set{ \w[\p]{G} + k \delta_{\p} }_{\p \in U} &\leq \nu_i + (i-k) c.
	\end{align*}
\end{lemma}

\begin{proof}
	Let $\q \in U$ be a prime ideal with a maximal value of $\lambda_{\q}$. The statement follows from the following chain of inequalities:
	\begin{align}
		\label{eq:cas1_proof} % proof
		\Min\set{\w[\p]{G} + k \delta_{\p}}_{\p \in U} + kc &\leq \w[U]{G \phi_{\ell,\q}^{k}} 
			\leq \w[U]{g_{i}} = \nu_{i} + ic.
	\end{align}
	
	The second inequality of \eqref{eq:cas1_proof} follows from the maximality of $g_{i}$. The first inequality is deduced from the formulas from Proposition \ref{prop:explicit_valuations}. In fact, for any $\p \in U$, these formulas yield $\w[\p]{\phi_{\ell,\q}} = (V_{\ell} + \lambda)/(\eto{\ell-1})$, for a certain rational number $\lambda$, depending on $\p$, such that $\lambda \geq \lambda_{\p}$; hence,
	\begin{align*}
		\w[\p]{G \phi_{\ell,\q}^{k}} &= \w[\p]{G} + k \frac{V_{\ell} + \lambda}{\eto{\ell-1}} \\
			&\ge \w[\p]{G} + k (c + \delta_{\p}),
	\end{align*}
	for all $\p \in U$, which implies the first inequality in \eqref{eq:cas1_proof}.
	
	More precisely, if $i(\p,\q) = \ell$, then $\lambda = \clam{\p}{\q}$ or $\lambda = \minclam{\p}{\q}$,  according to $\phi(\p,\q)$ being equal to $\phi_{\ell,\q}$ or not. Now, if $\p$ and $\q$ belong to the same $\phi_{\ell}$-branch of the non-optimised tree, we have (see Definition \ref{def:cphi_hiddenslope})
	\begin{align*}
		\lambda_{\p} = \lambda_{\q} < \minclam{\p}{\q} \leq \lambda.
	\end{align*}
		
	If $\p$ and $\q$ belong to different $\phi_{\ell}$-branches of the non-optimised tree, then $\clam{\p}{\q} = \lambda_{\p}$, $\clam{\q}{\p} = \lambda_{\q}$, so that, again,
	\begin{align*}
		\lambda_{\p} &=  \Min\set{\lambda_{\p}, \lambda_{\q}} =  \minclam{\p}{\q} \leq \lambda.
	\end{align*}
	
	Finally, if $i(\p,\q) > \ell$, then $\lambda = \lambda_{\ell,\q} = \lambda_{\ell,\p} \ge \lambda_{\p}$, by Remark \ref{rmk:nonop_valuations}.
\end{proof}

We are ready to prove Theorem \ref{thm:maxmin_ml_works} in Case \ref{case:unrefined_lmin}.

\begin{proposition}
	\label{prop:separated_maxmin22} % sepmax22
	In Case \ref{case:unrefined_lmin}, all output multi-indices of $\mm[S;\ml]$ are maximal amongst the multi-indices of the same degree whose coordinates are all divisible by $\ml$.
\end{proposition}

\begin{proof}
	Let $\ii = (i_{\p})_{\p \in S}$ be an output multi-index of $\mm[S;\ml]$, obtained by joining the $i$-th output of $\mm[U]$ and the $j$-th output of $\mm[D]$.
	
	Let $g$ be a polynomial of degree $(i+j)\ml$ with support in $S$. We may write $g = GG'$, with $G$, $G'$ polynomials with support in $U$ and $D$, respectively.
	
	Suppose $\deg G = (i-k) \ml$, $\deg G' = (j+k) \ml$, for certain $k \in \Z$. Let us write
	\begin{align*}
		G' &= H \phi_{\ell}^m,	& \phi_{\ell} \ndivides H, \  \deg H = q\ml.
	\end{align*}
	
	Note that $q + m = j + k$. By \eqref{eq:udval_gg},
	\begin{align*}
		\w[U]{G G'} &= \Min\set{\w[\p]{G} + m \delta_{\p}}_{\p \in U} + (j+k) c, \\
		\w[D]{G G'} &= \w[D]{G'} + (i-k) c.
	\end{align*}
	
	Since $\w[\p]{\phi_{\ell}} = c$ for all $\q \in D$, the last equality leads to
	\begin{align}
		\label{eq:nuprima} % nuprima
		\begin{aligned}
			\w[D]{G G'} &= \w[D]{H} + (m+i-k) c \le \w[D]{g'_{q}} + (m+i-k) c \\
				&= \nu'_{q} + (q+m+i-k) c = \nu'_{q} + (i+j) c.
		\end{aligned}
	\end{align}
	 
	By Lemma \ref{lem:sepmax21}, we may distinguish two cases according to the comparison of $\nu_{i}$ with $\nu'_{j}$.
	
	\begin{proofcase}
		$\nu'_{j-1} \leq \nu_{i} \leq \nu'_{j}$. In this case, we saw during the proof of Lemma \ref{lem:sepmax21} that $[j] = 0$. Hence, $\w[S]{g_{i} g'_{j}} = \w[U]{g_{i} g'_{j}} = \nu_{i} + (i+j) c$, by \eqref{eq:udval_gg2}. We want to show that
		\begin{align*}
			\w[S]{G G'} &= \Min\set{\w[U]{G G'}, \w[D]{G G'}} \leq \nu_{i} + (i+j) c.
		\end{align*}
		
		If $m \leq k$, then Lemma \ref{lem:cas1} shows that
		\begin{align*}
			\w[U]{G G'} &= \Min\set{\w[\p]{G} + m \delta_{\p}}_{\p \in U} + (j+k) c \\
				& \leq \Min\set{\w[\p]{G} + k \delta_{\p}}_{\p \in U} + (j+k) c \\
				&\leq \nu_{i} + (i-k) c + (j+k) c = \nu_{i} + (i+j) c.
		\end{align*}
			
		If $m > k$, then $q < j$, or equivalently $q \leq j-1$. Thus, \eqref{eq:nuprima} shows that
		\begin{align*}
			\w[D]{G G'} &\le \nu'_{q} + (i+j) c \leq \nu'_{j-1} + (i+j) c \leq \nu_{i} + (i+j) c.
		\end{align*}
	\end{proofcase}
	
	\begin{proofcase}
		$\nu_{i-1} \leq \nu'_{j} < \nu_{i}$. In this case, we saw during the proof of Lemma \ref{lem:sepmax21} that $\w[S]{g_{i} g'_{j}} = \w[D]{g_{i} g'_{j}} = \nu'_{j} + (i+j) c$. We want to show that
		\begin{align*}
			\w[S]{G G'} &= \Min\set{\w[U]{G G'}, \w[D]{G G'}} \le \nu'_{j} + (i+j) c.
		\end{align*}
		
		If $m < k$, then $m \leq k-1$. Having in mind that $\deg G/\ml = i-k = (i-1) - (k-1)$, Lemma \ref{lem:cas1} shows that 
		\begin{align*}
			\w[U]{G G'} &= \Min\set{\wt[\p]{G} + m \delta_{\p}}_{\p \in U} + (j+k) c \\
				&\leq \Min\set{\wt[\p]{G} + (k-1) \delta_{\p}}_{\p \in U} + (j+k) c \\
				&\leq \nu_{i-1} + (i-k) c + (j+k) c = \nu_{i-1} + (i+j) c \le \nu'_{j} + (i+j) c.
		\end{align*}
		
		If $m \geq k$, then $q \leq j$ and \eqref{eq:nuprima} shows that
		\begin{align*}
			\w[D]{G G'} &\leq \nu'_{q} + (i+j) c \leq \nu'_{j} + (i+j) c.
		\end{align*}
	\end{proofcase}
\end{proof}

%-------------------------------------------------------------------------------
% REFERENCES
%-------------------------------------------------------------------------------
\bibliographystyle{plain}
\bibliography{./bibliography/montes}

\begin{thebibliography}{10}

\bibitem{Bauch:2014:thesis}
Jens-Dietrich Bauch.
\newblock {\em Lattices over polynomial Rings and Applications to Function
  Fields}.
\newblock PhD thesis, Universitat Aut{\`o}noma de Barcelona, July 2014.

\bibitem{Bosma:1997:magma}
Wieb Bosma, John Cannon, and Catherine Playoust.
\newblock The {M}agma algebra system. {I}. {T}he user language.
\newblock {\em Journal of Symbolic Computation}, 24(3-4):235--265, 1997.
\newblock Computational algebra and number theory (London, 1993).

\bibitem{Fernandez:2013wg}
Julio Fern{\'a}ndez, Jordi Gu{\`a}rdia, Jes{\'u}s Montes, and Enric Nart.
\newblock {Residual ideals of MacLane valuations}.
\newblock {\em Journal of Algebra}, 427:30--75, April 2015.

\bibitem{Guardia:2011:algorithm}
Jordi Gu{\`a}rdia, Jes{\'u}s Montes, and Enric Nart.
\newblock {Higher Newton polygons in the computation of discriminants and prime
  ideal decomposition in number fields}.
\newblock {\em Journal de Th{\'e}orie des Nombres de Bordeaux}, 23(3):667--696,
  2011.

\bibitem{Guardia:2012:hn}
Jordi Gu{\`a}rdia, Jes{\'u}s Montes, and Enric Nart.
\newblock {Newton polygons of higher order in algebraic number theory}.
\newblock {\em Transactions of the American Mathematical Society},
  364:361--416, 2012.

\bibitem{Guardia:2013:newapp}
Jordi Gu{\`a}rdia, Jes{\'u}s Montes, and Enric Nart.
\newblock {A new computational approach to ideal theory in number fields}.
\newblock {\em Foundations of Computational Mathematics}, 13(5):729--762, 2013.

\bibitem{Guardia:2009:bases}
Jordi Gu{\`a}rdia, Jes{\'u}s Montes, and Enric Nart.
\newblock {Higher Newton polygons and integral bases}.
\newblock {\em Journal of Number Theory}, 147:549--589, February 2015.

\bibitem{Guardia:2013:genetics}
Jordi Gu{\`a}rdia and Enric Nart.
\newblock {Genetics of polynomials over local fields}.
\newblock In St{\'e}phane Ballet, Marc Perret, and Alexey Zaytsev, editors,
  {\em {Contemporary Mathematics: Proceedings of the 14th International
  Conference on Arithmetic, Geometry, Cryptography, and Coding Theory (AGCT)}},
  volume 637, pages 207--244. American Mathematical Society, 2015.

\bibitem{Guardia:2011:sfl}
Jordi Gu{\`a}rdia, Enric Nart, and Sebastian Pauli.
\newblock {Single-factor lifting and factorization of polynomials over local
  fields}.
\newblock {\em Journal of Symbolic Computation}, 47(11):1318--1346, 2012.

\bibitem{Nart:2014:equiv}
Enric Nart.
\newblock On the equivalence of types.
\newblock {\em Journal de Th\'{e}orie des Nombres de Bordeaux, to appear},
  September 2014.

\bibitem{Okutsu:1982:i+ii}
K{\=o}saku Okutsu.
\newblock {Construction of integral basis. I, II}.
\newblock {\em Proceedings of the Japan Academy, Series A, Mathematical
  Sciences}, 58(1):47--49, 87--89, 1982.

\bibitem{Serre:1968}
Jean-Pierre Serre.
\newblock {\em Corps locaux}.
\newblock Hermann, second edition, 1968.

\bibitem{Stainsby:2014:thesis}
Hayden~D. Stainsby.
\newblock {\em Triangular bases of integral closures}.
\newblock PhD thesis, {Universitat Aut{\`o}noma de Barcelona}, December 2014.

\end{thebibliography}

\end{document}